\documentclass[10pt,reqno,final]{article}
\UseRawInputEncoding
\usepackage{cite}
\usepackage{amsmath}
\usepackage{graphicx}
\usepackage{float}
\usepackage{caption}
\usepackage{amsxtra}
\usepackage{cases}
\usepackage{float}
\usepackage{amsmath}
\usepackage{fancyhdr}
\usepackage{caption,subcaption}
\usepackage{overpic}
\usepackage{epstopdf}
\usepackage{color}
\usepackage[justification=centering]{caption}
\DeclareCaptionLabelFormat{cont}{#1~#2\alph{ContinuedFloat}}
\usepackage{amsthm}
\usepackage[title]{appendix}

\theoremstyle{plain}
\newtheorem{theorem}{Theorem}
\newtheorem{lemma}{Lemma}

\theoremstyle{definition}

\theoremstyle{remark}
\newtheorem{remark}{Remark}

\usepackage[justification=centering]{caption}
\usepackage{indentfirst}
\usepackage{enumerate}

\title{\Large
	A Degenerate Bifurcation Perspective on High Sensitivity in a Modified Gower-Leslie Model with Additive Allee Effect
	\thanks{ {E-mail: mathwangxiaoling@163.com, wukuilin@126.com, zoulan@cnu.edu.cn}.}
}

\author{{ Xiaoling Wang}$^{a}$, { Kuilin Wu}$^{a}$ and { Lan Zou}$^{b}\thanks{Corresponding author}$
\\
$^{a}$
{\small School of Mathematics and Statistics, Guizhou University, Guiyang 550025, PR China}
\\
$^{b}$
{\small School of Mathematical Sciences, Capital Normal University, Beijing 100048, PR China}
}

\date{}

\begin{document}
\captionsetup{justification=raggedright, singlelinecheck=false}
\captionsetup[figure]{labelfont={bf},name={Fig.},labelsep=period}

\maketitle
\maketitle
\begin{abstract}
The population dynamics in a modified Leslie-Gower model with an additive Allee effect are highly sensitive to both parameters and initial population densities, leading to outcomes ranging from coextinction to sustained multistable steady states. This work links this sensitivity to complicated bifurcations. We establish the existence of a codimension 4 nilpotent cusp and a corresponding degenerate Bogdanov-Takens bifurcation with codimension 4, which critically shape the system's response to parameter changes. Most significantly, we prove that the Hopf bifurcation occurring at a center-type equilibrium can give rise to up to five limit cycles¡ªa phenomenon scarcely documented in previous ecological studies¡ªthereby inducing a pronounced dependence of oscillatory regimes on initial conditions. Numerical simulations confirming heteroclinic loops and multiple limit cycles provide consistent support for the theoretical analysis.

		\smallskip
		
\end{abstract}

\section{Introduction}

Predation is one of the most basic interactions in population biology. Since the classic Lotka-Volterra model was independently developed by Lotka and Volterra \cite{AJA,V}, many people have considered various factors \cite{GSK,EE,A} and explored some interesting biological phenomena more extensively, One of the most popular systems has been proposed and studied by Freedman and Mathsen \cite{IM}, Hsu and Huang \cite{BW} respectivel, called Leslie type with the following form
\begin{equation}
\begin{array}{ll}\label{1.1}
\left\{
\begin{aligned}
\dot{x}&=xg(x,K)-yp(x),\\[2ex]
\dot{y}&=y(h-\frac{ny}{x}),
\end{aligned}
\right.
\end{array}
\end{equation}
where $x$ and $y$ denote densities of the prey and predators at time $t$, respectively. $K$ denotes the carrying capacity of the prey in the absence of predators. $h$ is the intrinsic growth rates of predator and $n$ is a measure of the food quality that the prey provides for conversion into predator birth. $g(x,K)$ describes the specific growth rate of the prey in the absence of predators and satisfies $g(0,K)=r>0$, $g(K,K)=0$, $g_{x}(K,K)<0$, $g_{x}(x,K)\leq0$, and $g_{K}(x,K)>0$ for any $x>0$. $p(x)$ describes the change in the density of the prey attacked per unit time per predator as the prey density changes. In model (\ref{1.1}), Leslie-Gower term $ny/x$ descries the carrying capacity of the predator's environment is proportional to the number of prey. It is applicable for many species of animals. For example, Hanski et al. believed that it fit the dynamics of vole-weasel interaction \cite{HHH}.

Some complicated and challenging dynamics of  system \eqref{1.1} have been detected mathematically. Li and Xiao \cite{YD} detected that system \eqref{1.1} with simplified Holling type IV function can undergo Bogdanov-Takens bifurcation of codimension 2 and subcritical Hopf bifurcation. Later, Huang et al. \cite{JS} showed that system \eqref{1.1} with generalized Holling type III functional response can exhibit subcritical Hopf bifurcation and degenerate focus type Bogdanov-Takens bifurcation of codimension 3. Dai and Zhao \cite{YY} showed that system \eqref{1.1} can exhibit a weak focus and its order is at most 2 which has Hopf cyclicity 2. Besides, for two anti-saddles, they showed that the Hopf cyclicity for positive equilibrium with smaller abscissa (resp. bigger abscissa) is 2 (resp. 1). Huang et al. \cite{JX} showed that system \eqref{1.1} with simplified Holling type IV functional response can undergo degenerate Bogdanov-Takens singularity (focus case) of codimension 3. Zhu and Zou found three limit cycles generating from Hopf bifurcation for system \eqref{1.1} with mating difficulty and Holling type IV functional response in \cite{ZZ}. Different types of functional response functions were introduced in \cite{JS}. Here are some other excellent research, see \cite{HA,PA,AE,ST,T}.

In fact, some predators may switch to alternative food to avoid extinction, when their growth is limited by the scarcity or insufficient supply of their preferred prey $x$, although they have their favorite food. This situation can be modified by adding a positive additional constant $d$ to the denominator of Leslie-Gower term  and system \eqref{1.1} is modified as
\begin{equation}
\begin{array}{ll}\label{1.2}
\left\{
\begin{aligned}
\dot{x}&=xg(x,K)-yp(x),\\[2ex]
\dot{y}&=y(h-\frac{ny}{x+d}),
\end{aligned}
\right.
\end{array}
\end{equation}
where $d$ denotes the measure of alternative food available to predators. This system is called the modified Leslie-Gower predator-prey model. In fact, although the least weasel {\it Mustela nivalis} modeled in \cite{HHH} is a highly specialized predator of small rodents, with a particular preference for {\it Microtus} voles, the alternative prey in weasels' diet has been shown as passerine bird nestlings and young rabbits, when {\it Microtus} voles were scarce or absent \cite{PH}.

Jia et al. \cite{XM} showed that system \eqref{1.2} exhibit complicated dynamics including the existence of a limit cycle, a homoclinic loop, two limit cycles, or both a limit cycle and a homoclinic loop, as well as Bogdanov-Takens bifurcations of codimension 2 and 3. Tian and Liu \cite{JP} investigated the steady-state bifurcation and the existence, direction and stability of periodic orbits of system \eqref{1.2} with Beddington-DeAngelis functional response. Aziz-Alaoui and Okiye \cite{AD} detected the boundedness of solutions, existence of an attracting set and global stability of the coexisting interior equilibrium for system \eqref{1.2} with Holling II functional response.

The Allee effect is characterized by a positive relationship between population density and per capita growth rate at low densities. This occurs because reduced fitness in sparse populations, due to mechanisms like impaired mate-finding or cooperation, drives a trajectory toward extinction \cite{AJ,PW,AF,ZZ}.
In many predator-prey models, the Allee effect is treated as a function independent of the predation function, yet the specific form of this function governs the expansion of the bistability regime. The most usual continuous growth equation to express the Allee effect is given by multiplicative Allee effect:
\begin{equation}
\begin{array}{ll}\label{1.3}
\begin{aligned}
g(x,K)&=r(1-\frac{x}{K})(x-M),
\end{aligned}
\end{array}
\end{equation}
where $r$ and $K$ denote the intrinsic growth rate and the carrying capacity of the prey in the absence of predators, respectively. $-K<M<K$ represents the threshold of Allee effects in prey, and $0<M<K$ denotes the strong Allee effect while $-K<M<0$ denotes the weak Allee effect.

For system \eqref{1.1} with \eqref{1.3}, Shang and Qiao \cite{ZY} showed that system \eqref{1.1} with simplified Holling type IV functional response and strong Allee effect on prey can exhibit saddle-node bifurcation, Hopf bifurcation, degenerate Hopf bifurcation and Bogdanov-Takens bifurcation of codimensions 2 and 3. Huang et al. found the coexistence of 3 limit cycles in this system in \cite{JMC}. For system \eqref{1.2} with \eqref{1.3}, Arancibia-Ibarra and Gonz$\acute{a}$lez-Olivares \cite{CE} detected system \eqref{1.2} can exhibit a stable limit cycle. Khanghahi and Ghaziani \cite{JK} found that \eqref{1.3} system \eqref{1.2} can undergo saddle-node, Hopf, Bogdanov-Takens and generalized Hopf bifurcations. Guo et al. \cite{XL} examined the properties of the equilibrium and existence of saddle-node bifurcation.

Another population growth  deduced in \cite{PW,B}
\begin{equation}
\begin{array}{ll}\label{1.4}
\begin{aligned}
\dot{x}&=rx(1-\frac{x}{K}-\frac{A}{x+B}),
\end{aligned}
\end{array}
\end{equation}
is called the additive Allee effect. The term $\frac{A}{x+B}$ can induce either weak or strong Allee effects in the absence of other population intersections. The positive constants $A$ and $B$ characterize the Allee effect, where
$B$ denotes the half-maximum fitness population size and $A$ scales its severity. From \cite{HM}, it then follows: if $0<A<B$, the Allee effect in \eqref{1.4} is the weak one; if $A>B$, the Allee effect in \eqref{1.4} is the strong one.
The additive Allee effect describes a population dynamic in which the per capita growth rate increases at low densities due to heightened positive interactions among individuals, such as cooperative reproduction and group defense. Different from \eqref{1.3}, the additive Allee effect assumes that the impact of positive interactions on population growth is additive, which is independent of population density.


An extensive number of studies have incorporated the additive Allee effect into predator-prey models.  Aguirre et al. \cite{PE} showed that the system allows the existence of a stable limit cycle surrounding an unstable limit cycle generated by Hopf bifurcation. Besides, they also found that the system allows the coexistence of three limit cycles in \cite{PEE}. Suryanto, Darti and Anam \cite{AI} analyzed the existence and local stability of equilibria for system \eqref{1.2} with additive Allee effect. Cai et al. \cite{YC} studied the impact of additive Allee effect on system \eqref{1.2} with \eqref{1.4} by demonstrating the existence of a Hopf bifurcation and confirming that this effect elevates the risk of ecological extinction.

In this paper, we focus on a modified Leslie-Gower model \eqref{1.2} with Lotka-Volterra type functional response and the additive Allee effect on prey as follows
\begin{equation}
\begin{array}{ll}\label{1.5}
\left\{
\begin{aligned}
\dot{x}&=rx(1-\frac{x}{K}-\frac{A}{x+B})-pxy,\\[2ex]
\dot{y}&=y(h-\frac{ny}{x+d}).
\end{aligned}
\right.
\end{array}
\end{equation}
Taking $x=K\bar{x}$, $y=\frac{Kh}{n}\bar{y}$ and $t=\frac{1}{r}\bar{t}$,  system \eqref{1.5} can be written as follows (still denote $\bar{x}$, $\bar{y}$ and $\bar{t}$ by $x$, $y$ and $t$, respectively).
\begin{equation}
\begin{array}{ll}\label{2.1}
\left\{
\begin{aligned}
\dot{x}&=x(1-x)-\gamma xy-\frac{\beta x}{x+\alpha},\\[2ex]
\dot{y}&=\delta y(1-\frac{y}{x+\eta}),
\end{aligned}
\right.
\end{array}
\end{equation}
where $\gamma=\frac{Kph}{rn}$, $\alpha=\frac{B}{K}$, $\delta=\frac{h}{r}$, $\beta=\frac{A}{K}$ and $\eta=\frac{d}{K}$. They are all positive parameters. From the biological point view, we consider system \eqref{2.1} in $\mathbf{R}^2_+=\{(x,y)|x\geq0, y\geq0\}$.

We  show that system \eqref{1.5} can exhibit saddle-node bifurcation,  nilpotent cusp bifurcation of codimension  4 and Hopf bifurcation. In particular, we detect that  system \eqref{1.5} can emerge as many as 5 limit cycles via Hopf bifurcation. As our known, this is a new upper bound for the number of  limit cycles of predator-prey models. Compared with system \eqref{1.2} without the additive Allee effect, our results indicate that the additive Allee effect can cause richer dynamical behaviors and bifurcation phenomena. Besides, it can cause the coextinction of both populations even with some positive initial densities. Ecologically, the interaction of steppe polecat ({\it{Mustela eversmanni larvatus}}) and Plateau zoker ({\it{Myospalax baileyi}}) in Qinghai-Tibet Plateau is a proper example of predator-prey for model \eqref{1.5}.

This paper is structured as follows. In section 2, we consider the extinction and persistence of populations, and the existence and types of boundary equilibria and non-hyperbolic positive equilibria of system \eqref{2.1} are presented. In section 3, we show that system \eqref{2.1} can undergo saddle-node bifurcation and degenerate Bogdanov-Takens bifurcation of codimension 4. In section 4, we investigate system \eqref{2.1} admits five limit cycles bifurcated from Hopf bifurcation, and exhibit a pronounced dependence of oscillatory regimes on initial conditions by numerical simulations. A brief discussion is given in the last section.

\section{Equilibria analysis}



In order to clarify the extinction and persistence of both populations with positive initial densities, we deal with the boundedness for system \eqref{2.1} in $\Omega$ firstly, where
\begin{equation*}
\begin{array}{ll}
\begin{aligned}
\Omega=\{(x,y)|x>0,y>0\},\quad\Gamma=\{(x,y)|0<x<1,y<\eta+1\}.
\end{aligned}
\end{array}
\end{equation*}
Letting $t=(x+\alpha)(x+\eta)\tau$,  system \eqref{2.1} can be written as (still denote $\tau$ by $t$)
\begin{equation}
\begin{array}{ll}\label{1.11}
\left\{
\begin{aligned}
\dot{x}&=x(x+\eta)\big((1-x)(x+\alpha)-\gamma y(x+\alpha)-\beta\big),\\[2ex]
\dot{y}&=\delta y(x+\alpha)(x+\eta-y).
\end{aligned}
\right.
\end{array}
\end{equation}
It is easy to see that system \eqref{2.1} and \eqref{1.11} are topologically equivalent in the region $\Omega$ since $(x+\alpha)(x+\eta)>0$ holds for $x>0$. It is clear that there exists $T>0$ such that all trajectories $(x(t),y(t))$ of system \eqref{1.11} in $\Omega$ enter and remain in $\Gamma$ for all $t>T$.
\begin{lemma}\label{t6}
There exists $T>0$ such that all trajectories $(x(t),y(t))$ of system \eqref{2.1} originating in $\Omega$ enter and remain in $\Gamma$ for all $t>T$.
\end{lemma}

\subsection{Boundary equilibria}

To understand the details of extinction of either prey or predator, we need to analysis boundary equilibria for system \eqref{2.1}.
Denoting $\Delta_1=(1+\alpha)^2-4\beta$, system \eqref{2.1} has at most four boundary equilibria $E_0(0,0)$, $E_1(0,\eta)$, $E_2\big(\frac{1}{2}(1-\alpha-\sqrt{\Delta_1}),0\big)$ and $E_3\big(\frac{1}{2}(1-\alpha+\sqrt{\Delta_1}),0\big)$. By Chapter 2 of \cite{ZT}, we have the following results.
\begin{lemma}\label{l2}
System \eqref{2.1} has at most four boundary equilibria, and the following statements hold.
\begin{enumerate}
\item[\bf(i)] When $\alpha<1$ and $\alpha<\beta<\frac{(1+\alpha)^2}{4}$, system \eqref{2.1} has four boundary equilibria $E_0$, $E_1$, $E_2$ and $E_3$, where $E_0$ is a saddle, $E_1$ is a stable node, $E_2$ is an unstable node and $E_3$ is a saddle, whose phase portrait is shown in Fig. \ref{z1}(a).

\item[\bf(ii)] When $\beta<\alpha$, system \eqref{2.1} has three boundary equilibria $E_0$, $E_1$ and $E_3$, where $E_0$ is an unstable node, $E_1$ is a saddle (stable node) when $\gamma<\frac{\alpha-\beta}{\alpha\eta}$ ($\gamma>\frac{\alpha-\beta}{\alpha\eta}$) and $E_3$ is a saddle, the phase portraits are shown in Fig. \ref{z1}(b) and Fig. \ref{z1}(c).

\item[\bf(iii)] When either $\alpha>1$ and $\alpha<\beta<\frac{(1+\alpha)^2}{4}$, or $\beta>\frac{(1+\alpha)^2}{4}$, system \eqref{2.1} has two boundary equilibria $E_0$ and $E_1$, where $E_0$ is a saddle and $E_1$ is stable node, the phase portraits are shown in Fig. \ref{z1}(d).

\item[\bf(iv)] When $\beta=\frac{(1+\alpha)^2}{4}$, the following statements holds.
  \begin{enumerate}
  \item[\bf(iv-1)] If $\alpha<1$, system \eqref{2.1} has three boundary equilibria $E_0$, $E_1$ and $E_4(\frac{1-\alpha}{2},0)$, where $E_0$ is a saddle, $E_1$ is a stable node and $E_4$ is a saddle-node with a stable parabolic sector in the right half plane of $\mathbf{R}^2_+$, whose phase portrait is shown in Fig. \ref{z1}(e).

  \item[\bf(iv-2)] If $\alpha>1$, system \eqref{2.1} has two boundary equilibria $E_0$ and $E_1$, where $E_0$ is a saddle and $E_1$ is a stable node, whose phase portrait is shown in Fig. \ref{z1}(f).
\end{enumerate}
\item[\bf(v)] When $\alpha=1$ and $\beta<1$, system \eqref{2.1} has three boundary equilibria $E_0$, $E_1$ and $E_5(\sqrt{1-\beta},0)$, where $E_0$ is an unstable node, $E_1$ is a stable node (saddle) when $\gamma>\frac{1-\beta}{\eta}$ ($\gamma<\frac{1-\beta}{\eta}$) and $E_5$ is a saddle, whose phase portrait is shown in Fig. \ref{z0}(a) and Fig. \ref{z0}(b).

\item[\bf(vi)] When $\alpha=\beta$, the following statements holds.
  \begin{enumerate}
  \item[\bf(vi-1)] If $\beta<1$, system \eqref{2.1} has three boundary equilibria $E_0$, $E_1$ and $E_6(1-\beta,0)$, where $E_0$ is saddle-node, $E_1$ is a stable node and $E_6$ is a saddle, whose phase portrait is shown in Fig. \ref{z0}(c).

  \item[\bf(vi-2)] If $\beta>1$, system \eqref{2.1} has two boundary equilibria $E_0$ and $E_1$, where $E_0$ is a saddle-node and $E_1$ is a stable node, whose phase portraits are shown in Fig. \ref{z0}(d).
  \end{enumerate}
\end{enumerate}
\end{lemma}

\begin{figure}[ht!]
\centering
\begin{subfigure}{0.45\linewidth}
\centering
\includegraphics[width=0.9\linewidth]{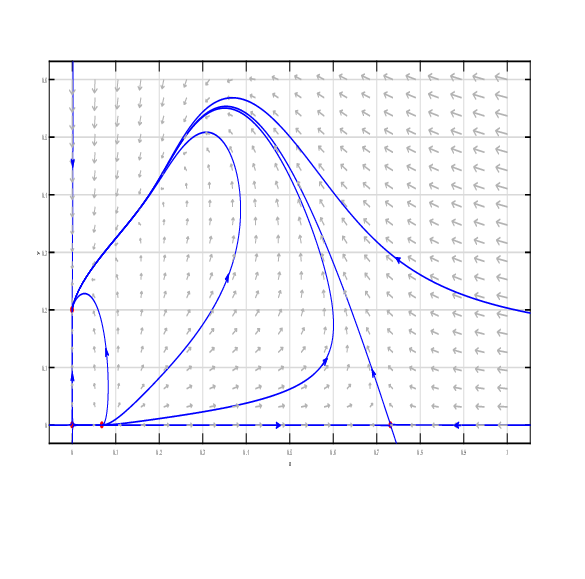}
\put(-162,41){$E_0$}
\put(-163,79){$E_1$}
\put(-139,41){$E_2$}
\put(-58,41){$E_3$}
\put(-101,10){$(a)$}
\end{subfigure}
\centering
\begin{subfigure}{0.45\linewidth}
\centering
\includegraphics[width=0.9\linewidth]{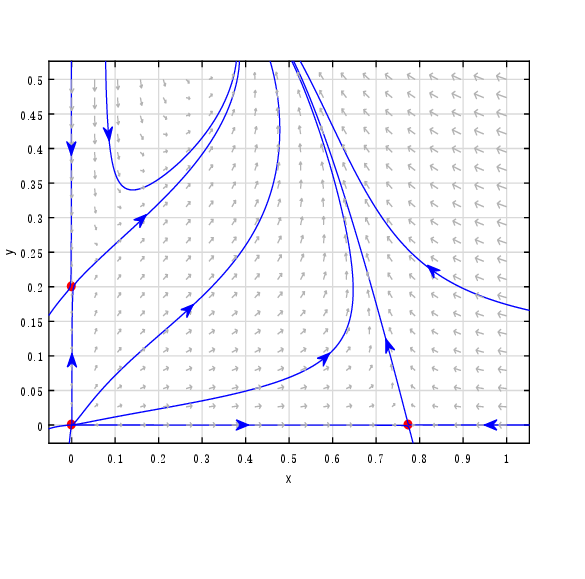}
\put(-162,41){$E_0$}
\put(-163,79){$E_1$}
\put(-53,41){$E_3$}
\put(-90,10){$(b)$}
\end{subfigure}

\centering
\begin{subfigure}{0.45\linewidth}
\centering
\includegraphics[width=0.9\linewidth]{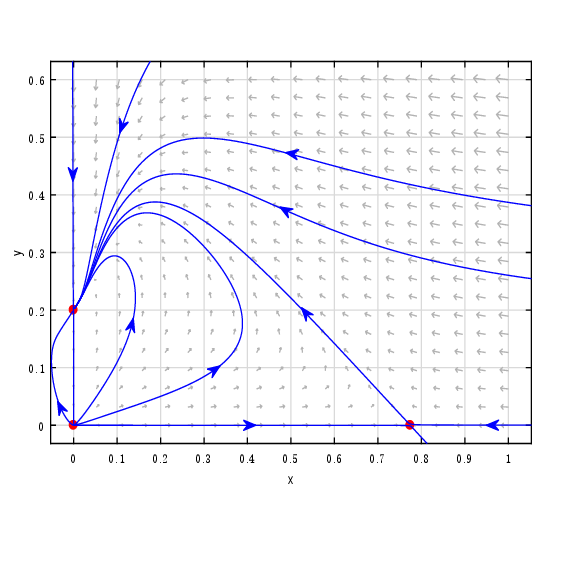}
\put(-162,41){$E_0$}
\put(-163,79){$E_1$}
\put(-53,41){$E_3$}
\put(-101,10){$(c)$}
\end{subfigure}
\centering
\begin{subfigure}{0.45\linewidth}
\centering
\includegraphics[width=0.9\linewidth]{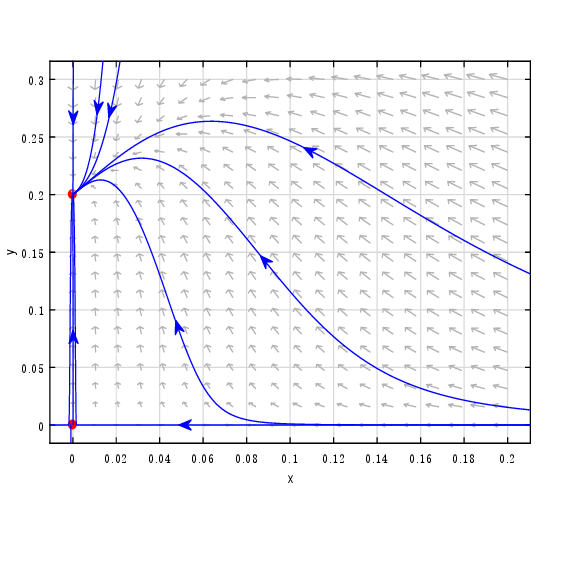}
\put(-162,41){$E_0$}
\put(-163,108){$E_1$}
\put(-90,10){$(d)$}
\vspace{3mm}
\end{subfigure}

\centering
\begin{subfigure}{0.45\linewidth}
\centering
\includegraphics[width=0.9\linewidth]{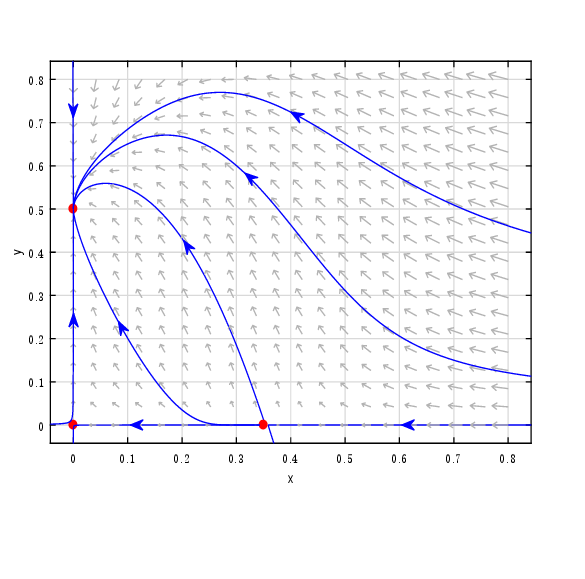}
\put(-162,41){$E_0$}
\put(-163,103){$E_1$}
\put(-96,41){$E_4$}
\put(-101,10){$(e)$}
\end{subfigure}
\centering
\begin{subfigure}{0.45\linewidth}
\centering
\includegraphics[width=0.9\linewidth]{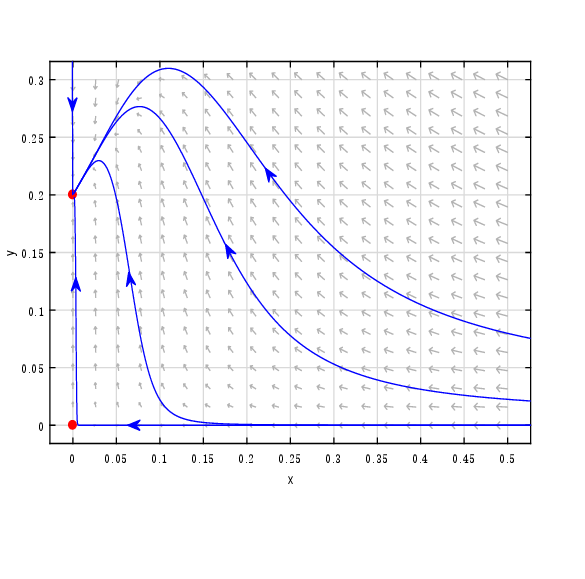}
\put(-162,41){$E_0$}
\put(-163,108){$E_1$}
\put(-90,10){$(f)$}
\vspace{3mm}
\end{subfigure}
\captionsetup{justification=centering}
\caption{The local phase portrait for boundary equilibria for system \eqref{2.1} in Lemma \ref{l2}. (a) case (i); (b) case (ii) with and $\gamma<\frac{\alpha-\beta}{\alpha\eta}$; (c) case (ii) with  $\gamma>\frac{\alpha-\beta}{\alpha\eta}$; (d) case (iii); (e) case (iv-1); (f) case (iv-2).}
\label{z1}
\end{figure}

\begin{figure}[ht!]
\centering
\begin{subfigure}{0.45\linewidth}
\centering
\includegraphics[width=0.9\linewidth]{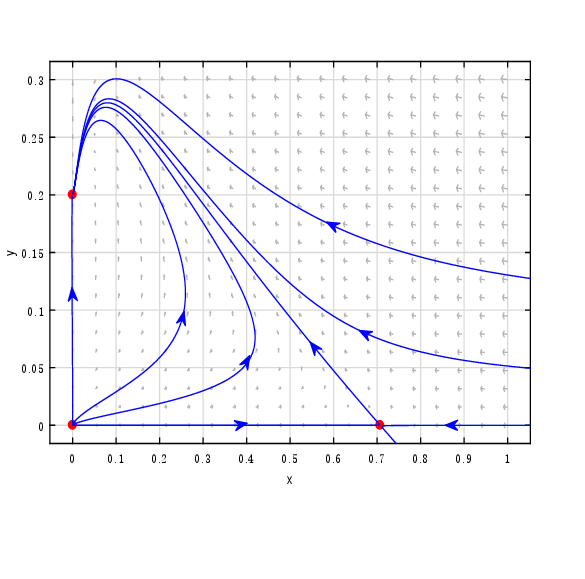}
\put(-162,41){$E_0$}
\put(-163,109){$E_1$}
\put(-58,41){$E_5$}
\put(-101,10){$(a)$}
\end{subfigure}
\centering
\begin{subfigure}{0.45\linewidth}
\centering
\includegraphics[width=0.9\linewidth]{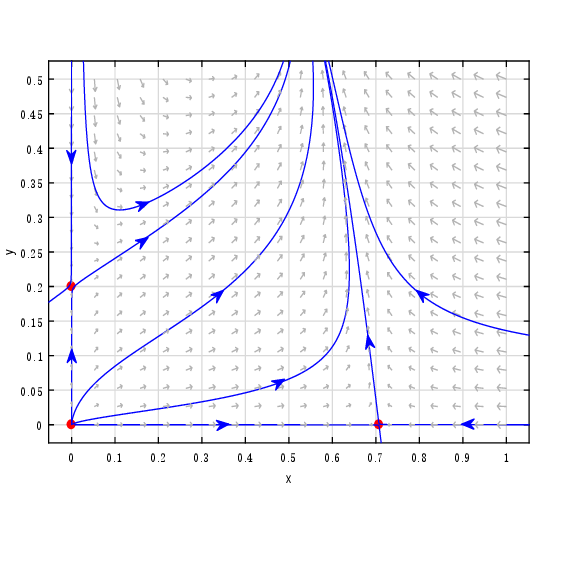}
\put(-162,41){$E_0$}
\put(-163,79){$E_1$}
\put(-59,41){$E_5$}
\put(-90,10){$(b)$}
\end{subfigure}

\centering
\begin{subfigure}{0.45\linewidth}
\centering
\includegraphics[width=0.9\linewidth]{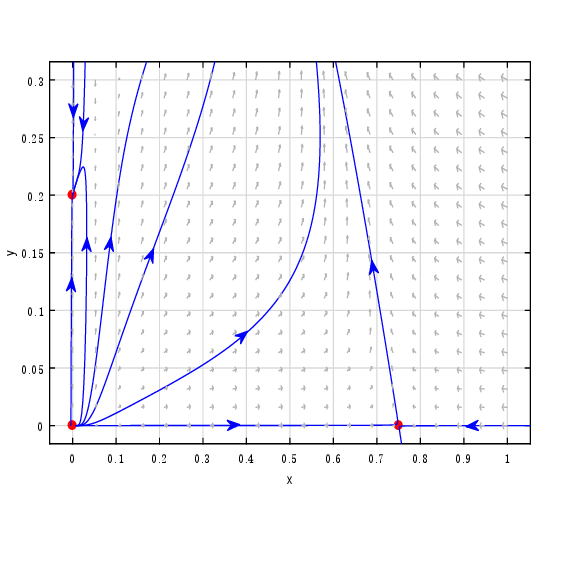}
\put(-162,41){$E_0$}
\put(-163,105){$E_1$}
\put(-53,41){$E_6$}
\put(-101,10){$(c)$}
\end{subfigure}
\centering
\begin{subfigure}{0.45\linewidth}
\centering
\includegraphics[width=0.9\linewidth]{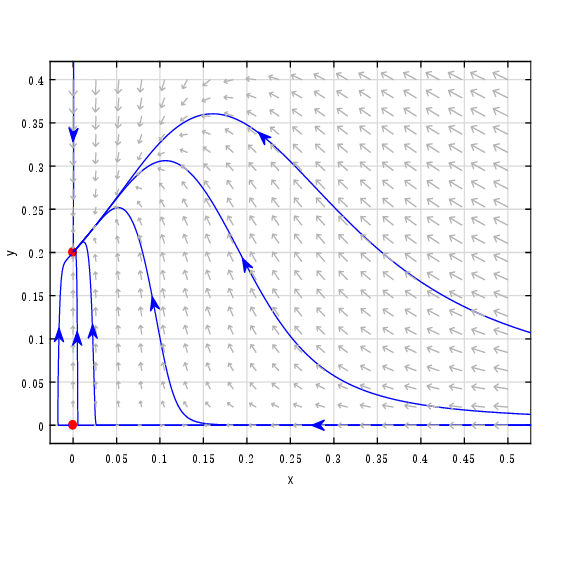}
\put(-162,41){$E_0$}
\put(-163,100){$E_1$}
\put(-90,10){$(d)$}
\vspace{3mm}
\end{subfigure}
\captionsetup{justification=centering}
\caption{The local phase portrait for boundary equilibria for system \eqref{2.1} in Lemma \ref{l2}. (a) case (v) with $\gamma>\frac{1-\beta}{\eta}$; (b) case (v) with $\gamma<\frac{1-\beta}{\eta}$; (c) case (vi-1); (d) case (vi-2).}
\label{z0}
\end{figure}
The proof of $E_4$ in Lemma \ref{l2} is deferred to Appendix A, whereas the results for the remaining equilibria are readily derived through a qualitative analysis of them.
Furthermore, from the simulations, we found lots of heteroclinic loops in figures \ref{z1} and \ref{z0}.

\subsection{Positive equilibria}

In this subsection, we focus on the positive equilibria of system \eqref{2.1}. The determinant and the trace of Jacobian matrix of system \eqref{2.1} at the positive equilibrium $E(x,y)$ are given by
\begin{equation*}
\begin{array}{ll}
\begin{aligned}
Det\big(J(E)\big)&=\delta x\bigg(1+\gamma-\frac{\beta}{(x+\alpha)^2}\bigg),\\[2ex]
Tr\big(J(E)\big)&=x(-1+\frac{\beta}{(x+\alpha)^2})-\delta.
\end{aligned}
\end{array}
\end{equation*}
The positive equilibrium $E(x,y)$ of system \eqref{2.1} must satisfy $y=x+\eta$, where $x$ is a positive real root of the equation
\begin{equation}
\begin{array}{ll}\label{2.2}
\begin{aligned}
(1+\gamma)x^2+\big(\gamma(\alpha+\eta)+\alpha-1\big)x+\alpha(\gamma\eta-1)+\beta=0.
\end{aligned}
\end{array}
\end{equation}
Letting
\begin{equation*}
\begin{array}{ll}
\begin{aligned}
\Delta_2=\big(\gamma(\alpha+\eta)+\alpha-1\big)^2-4(1+\gamma)\big(\alpha(\gamma\eta-1)+\beta\big),
\end{aligned}
\end{array}
\end{equation*}
by Theorem 7.1 in chapter 2 of \cite{ZT}, we derive the following results.
\begin{lemma}\label{t9}
For system \eqref{2.1}, the following statements hold.
\begin{enumerate}
\item[\bf(i)]  When $\beta<\frac{(1+\alpha+\alpha\gamma-\gamma\eta)^2}{4(1+\gamma)}$, system \eqref{2.1} admits at most two equilibria $E_1^*(x_1^*,y_1^*)$ and $E_2^*(x_2^*,y_2^*)$, where $x_{1,2}^*=\frac{-\big(\gamma(\alpha+\eta)+\alpha-1\big)\pm\sqrt{\Delta_2}}{2(1+\gamma)}$, $y_{1,2}^*=x_{1,2}^*+\eta$. In addition, the following two statements hold:
\begin{enumerate}
\item[\bf(a1)] For $\alpha<1$, $\gamma<\frac{1-\alpha}{\alpha+\eta}$ and $\alpha(1-\gamma\eta)<\beta$, system \eqref{2.1} possesses two positive equilibria $E_1^*$ and $E_2^*$, where $E_1^*$ is a saddle and $E_2^*$ is either a node or a focus, the phase portraits are shown in Fig. \ref{z2}(a) and Fig. \ref{z2}(b).

\item[\bf(a2)] For $\gamma<\frac{1}{\eta}$ and $\beta<\alpha(1-\gamma\eta)$, system \eqref{2.1} has a unique positive equilibrium $E_2^*$, which is either a node or a focus, the phase portraits are shown in Fig. \ref{z2}(c) and Fig. \ref{z2}(d).
\end{enumerate}

\item[\bf(ii)]  When $\beta=\frac{(1+\alpha+\alpha\gamma-\gamma\eta)^2}{4(1+\gamma)}$, $\alpha<1$ and $\gamma<\frac{1-\alpha}{\alpha+\eta}$, system \eqref{2.1} has an unique equilibrium $E^*(x^*,y^*)$ is a saddle-node with an unstable parabolic sector in $\mathbf{R}^2_+$ when $\delta\neq\frac{\gamma(1-\alpha-\alpha\gamma-\gamma\eta)}{2(1+\gamma)}$, where $x^*=\frac{1-\alpha-\alpha\gamma-\gamma\eta}{2(1+\gamma)}$, $y^*=x^*+\eta$, and the phase portrait is shown in Fig. \ref{2}.

\item[\bf(iii)] Otherwise, system \eqref{2.1} has no positive equilibria.
\end{enumerate}
\end{lemma}
\begin{figure}
\begin{center}
\begin{overpic}[scale=0.50]{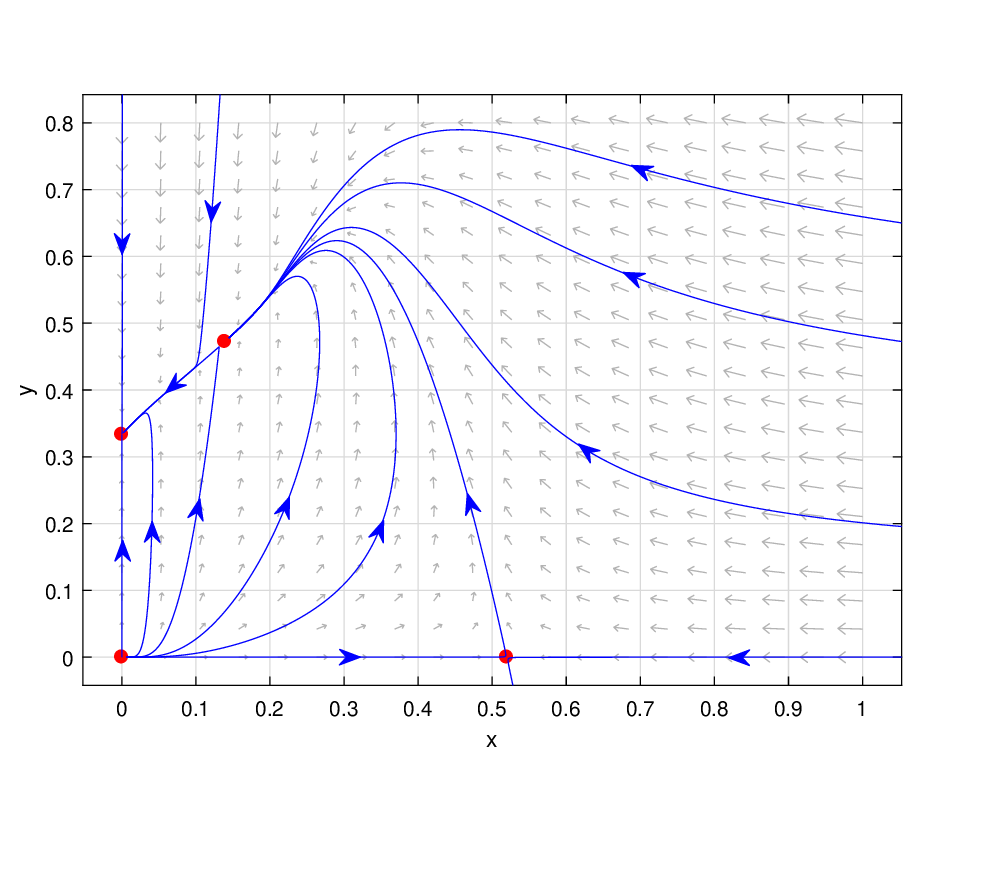}
\end{overpic}
\vspace{-15mm}
\put(-194,120){$E^*$}
\put(-214,50){$E_0$}
\put(-214,103){$E_1$}
\put(-122,50){$E_3$}
\end{center}
\caption{Case (ii) in Lemma \ref{t9} ($E^*$ is a saddle-node with an unstable parabolic sector).\label{2}}
\end{figure}
\begin{figure}[ht!]
\centering
\begin{subfigure}{0.45\linewidth}
\centering
\includegraphics[width=0.9\linewidth]{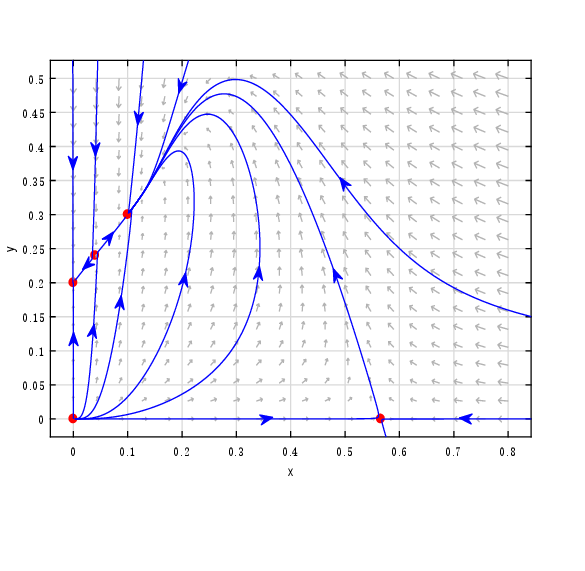}
\put(-162,41){$E_0$}
\put(-163,85){$E_1$}
\put(-58,41){$E_3$}
\put(-150,90){$E_1^{**}$}
\put(-145,103){$E_2^{**}$}
\put(-101,10){$(a)$}
\end{subfigure}
\centering
\begin{subfigure}{0.45\linewidth}
\centering
\includegraphics[width=0.9\linewidth]{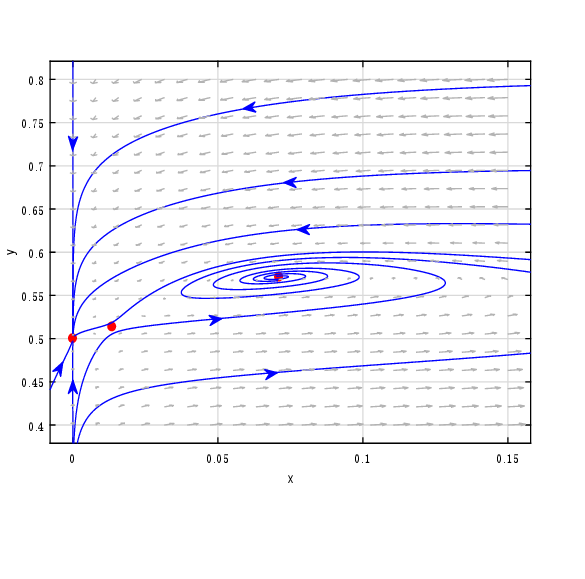}
\put(-163,65){$E_1$}
\put(-150,67){$E_1^{**}$}
\put(-87,91){$E_2^{**}$}
\put(-90,10){$(b)$}
\end{subfigure}

\centering
\begin{subfigure}{0.45\linewidth}
\centering
\includegraphics[width=0.9\linewidth]{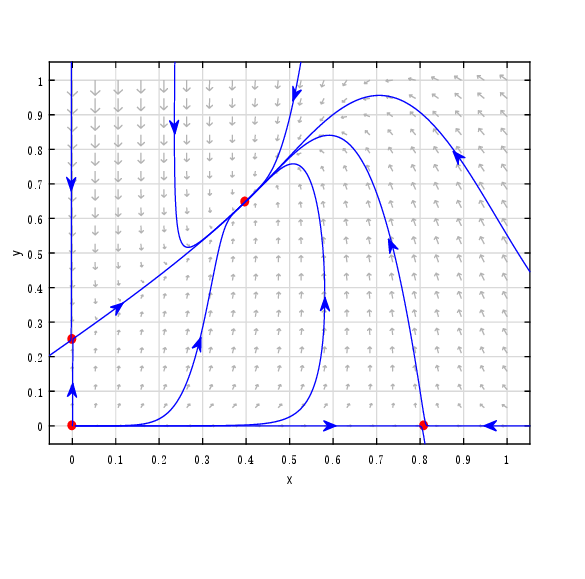}
\put(-162,41){$E_0$}
\put(-163,63){$E_1$}
\put(-59,41){$E_3$}
\put(-105,103){$E_2^{**}$}
\put(-101,10){$(c)$}
\end{subfigure}
\centering
\begin{subfigure}{0.45\linewidth}
\centering
\includegraphics[width=0.9\linewidth]{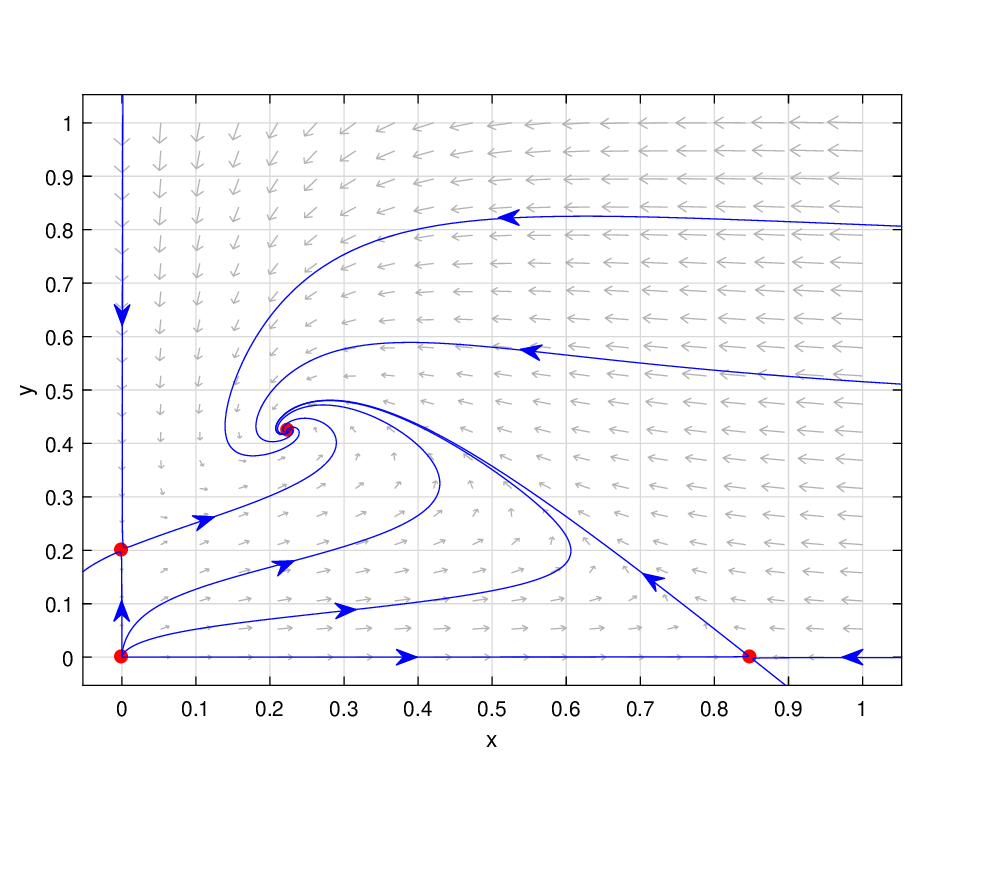}
\put(-162,41){$E_0$}
\put(-163,72){$E_1$}
\put(-56,41){$E_3$}
\put(-120,75){$E_2^{**}$}
\put(-90,10){$(d)$}
\vspace{3mm}
\end{subfigure}
\captionsetup{justification=centering}
\caption{The local phase portrait for positive equilibria for system \eqref{2.1} in Lemma \ref{t9}. (a) Case (i-a1) with a stable node $E_2^*$; (b) Case (i-a1) with a stable focus $E_2^*$; (c) Case (i-a2) with a stable node $E_2^*$; (d)  Case (i-a2) with a stable focus $E_2^*$.}
\label{z2}
\end{figure}
\begin{proof}
The proofs (i) and (ii) of the lemma can be obtained by simply calculation. We will prove (iii) of the lemma in details. Making the following transformations successively
\begin{equation*}
\begin{array}{ll}
\begin{aligned}
x&=X+\frac{1-\alpha-\alpha\gamma-\gamma\eta}{2(1+\gamma)},\quad y=Y+\frac{1-\alpha-\alpha\gamma+\gamma\eta+2\eta}{2(1+\gamma)}\quad and\\[2ex]
X&=x+\frac{\gamma(1-\alpha-\alpha\gamma-\gamma\eta)}{2\delta(1+\gamma)}y,\quad Y=x+y, t=\frac{2(1+\gamma)}{\gamma(1-\alpha-\alpha\gamma-\gamma\eta)-2\delta(1+\gamma)}\tau,
\end{aligned}
\end{array}
\end{equation*}
 system \eqref{2.1} can be written as (still denote $\tau$ by $t$)
\begin{equation}
\begin{array}{ll}\label{1.6}
\left\{
\begin{aligned}
\dot{x}&=\sum\limits_{i+j=2}\hat{a}_{ij}x^{i}y^{j}+o(|x,y|^{3}),\\[2ex]
\dot{y}&=y+\sum\limits_{i+j=2}\hat{b}_{ij}x^{i}y^{j}+o(|x,y|^{3}),
\end{aligned}
\right.
\end{array}
\end{equation}
where $$\hat{a}_{20}=-\frac{4\delta(1+\gamma)^3(\alpha+\alpha\gamma+\gamma\eta-1)}{(1+\alpha+\alpha\gamma-\gamma\eta)\big(\gamma(1-\alpha-\alpha\gamma-\gamma\eta)-2\delta(1+\gamma)\big)^2}>0$$ when $\alpha<1$, $\gamma<\frac{1-\alpha}{\alpha+\eta}$ and $\delta\neq\frac{\gamma(1-\alpha-\alpha\gamma-\gamma\eta)}{2(1+\gamma)}$. The expressions of $\hat{a}_{ij}$ and $\hat{b}_{ij}$ are omitted here for brevity.
By the Center Manifold Theorem, we have the reduced equation restricted to the center manifold as follows
\begin{equation}
\begin{array}{ll}\label{1.7}
\begin{aligned}
\dot{x}=\hat{a}_{20}x^2+o(|x|^2).
\end{aligned}
\end{array}
\end{equation}
By Theorem 7.1 in Chapter 2 of  \cite{ZT}, $E^*$ is a saddle-node with an unstable parabolic sector.
\end{proof}

In what follows, we deal with the cusp of system \eqref{2.1}, and obtain the following result:
\begin{theorem}\label{cusp}
When $\beta=\frac{(1+\alpha+\alpha\gamma-\gamma\eta)^2}{4(1+\gamma)}$, $\alpha<1$ and $\gamma<\frac{1-\alpha}{\alpha+\eta}$, two positive equilibria coincide to a positive equilibrium. Moreover, if $\delta=\frac{\gamma(1-\alpha-\alpha\gamma-\gamma\eta)}{2(1+\gamma)}$, system  may appear nilpotent cusp of codimension up to 4.
\end{theorem}
The proof will be completed by the following analysis and lemmas.

When $\beta=\frac{(1+\alpha+\alpha\gamma-\gamma\eta)^2}{4(1+\gamma)}$ and $\delta=\frac{\gamma(1-\alpha-\alpha\gamma-\gamma\eta)}{2(1+\gamma)}$, then system \eqref{2.1} has a degenerate positive equilibrium $E^*(\frac{1-\alpha-\alpha\gamma-\gamma\eta}{2(1+\gamma)},\frac{1-\alpha-\alpha\gamma+\gamma\eta+2\eta}{2(1+\gamma)})$. It is not difficult to see that $\beta>0$ and $\delta>0$ if and only if $(\beta,\delta)\in\Pi$, where
\begin{equation*}
\begin{array}{ll}
\begin{aligned}
\Pi:=\{(\beta,\delta)\in\mathbf{R}^2_+|0<\alpha<1, 0<\gamma<\frac{1-\alpha}{\alpha+\eta}, \eta>0\}.
\end{aligned}
\end{array}
\end{equation*}
In order to investigate  equilibrium  $E^*$, we make the following transformations successively
\begin{equation*}
\begin{array}{ll}
\begin{aligned}
x&=X+\frac{1-\alpha-\alpha\gamma-\gamma\eta}{2(1+\gamma)},\quad y=Y+\frac{1-\alpha-\alpha\gamma+\gamma\eta+2\eta}{2(1+\gamma)};\\[2ex]
X&=x+\frac{2(1+\gamma)}{\gamma(1-\alpha-\alpha\gamma-\gamma\eta)}y,\quad Y=x,
\end{aligned}
\end{array}
\end{equation*}
Then system \eqref{2.1} is changed into
\begin{equation}
\begin{array}{ll}\label{2.4}
\left\{
\begin{aligned}
\dot{x}&=y+\sum\limits_{2\leq i+j\leq5}a_{ij}x^{i}y^{j}+o(|x,y|^{5}),\\[2ex]
\dot{y}&=\sum\limits_{2\leq i+j\leq5}b_{ij}x^{i}y^{j}+o(|x,y|^{5}),
\end{aligned}
\right.
\end{array}
\end{equation}
where $a_{ij}$ and $b_{ij}$ are given in Appendix C.

Setting $x=X+\frac{b_{02}}{2}X^2+a_{02}XY$, $y=Y+b_{02}XY$, system \eqref{2.4} can be transformed into
\begin{equation}
\begin{array}{ll}\label{2.5}
\left\{
\begin{aligned}
\dot{X}&=Y+\sum\limits_{3\leq i+j\leq5}c_{ij}X^{i}Y^{j}+o(|X,Y|^{5}),\\[2ex]
\dot{Y}&=d_{20}X^2+d_{11}XY+\sum\limits_{3\leq i+j\leq5}d_{ij}X^{i}Y^{j}+o(|X,Y|^{5}),
\end{aligned}
\right.
\end{array}
\end{equation}
where $c_{ij}$ and $d_{ij}$ are given in Appendix D.
Through above analysis, we obtain $d_{20}=-\frac{\gamma(1-\alpha-\alpha\gamma-\gamma\eta)^2}{2(1+\alpha+\alpha\gamma-\gamma\eta)}<0$ and $d_{11}=\frac{\alpha(1+\gamma)(2+3\gamma)-(2+\gamma)(1-\gamma\eta)}{1+\alpha+\alpha\gamma-\gamma\eta}$. From \cite{RI}, we have the following result.
\begin{lemma}\label{t3}
If $\beta=\frac{(1+\alpha+\alpha\gamma-\gamma\eta)^2}{4(1+\gamma)}$, $\delta=\frac{\gamma(1-\alpha-\alpha\gamma-\gamma\eta)}{2(1+\gamma)}$ and $d_{11}\neq0$, then $E^*$  is a cusp of codimension 2 of system \eqref{2.1} when $\alpha\neq\beta$.
\end{lemma}

When  $\alpha=\beta=\frac{1+\gamma\eta-2\sqrt{\gamma\eta}}{1+\eta}$ and $\delta=\frac{\gamma\sqrt{\gamma\eta}-\eta\gamma^2}{1+\gamma}$, system \eqref{2.1} admits a positive equilibrium $E^{**}$.
Similar to Lemma \ref{t3},  we have the following result.
\begin{lemma}
If $0<\gamma<\frac{1}{\eta}$ and $d_{11}^*=\gamma-2\sqrt{\gamma\eta}-2\gamma\sqrt{\gamma\eta}\neq0$, i.e $\eta\neq\frac{\gamma}{4(1+\gamma)^2}$, then $E^{**}$  is a cusp of codimension 2 of system \eqref{2.1}.
\end{lemma}

If $d_{11}=0$ or $d_{11}^*=0$,  we need to further discuss $E^*$ and $E^{**}$, respectively. Making
\begin{equation*}
\begin{array}{ll}
\begin{aligned}
X&=u+\frac{2c_{21}+d_{12}}{6}u^3+\frac{c_{12}+d_{03}}{2}u^2v+c_{03}uv^2,\\[2ex]
Y&=v-c_{30}u^3+\frac{d_{12}}{2}u^2v+d_{03}uv^2,
\end{aligned}
\end{array}
\end{equation*}
system \eqref{2.5} becomes
\begin{equation}
\begin{array}{ll}\label{2.6}
\left\{
\begin{aligned}
\dot{u}&=v+\sum\limits_{4\leq i+j\leq5}e_{ij}u^{i}v^{j}+o(|u,v|^{5}),\\[2ex]
\dot{v}&=h_{20}u^2+h_{11}uv+h_{30}u^3+h_{31}u^3v+\sum\limits_{4\leq i+j\leq5}h_{ij}u^{i}v^{j}+o(|u,v|^{5}),
\end{aligned}
\right.
\end{array}
\end{equation}
where $e_{ij}$ and $h_{ij}$ are given in Appendix E.

Next, letting
\begin{equation*}
\begin{array}{ll}
\begin{aligned}
u&=x_1+\frac{3e_{31}+h_{22}}{12}x_1^4+\frac{2e_{22}+h_{13}}{6}x_1^3x_2+\frac{e_{13}+h_{04}}{2}x_1^2x_2^2+e_{04}x_1x_2^3,\\[2ex]
v&=x_2-e_{40}x_1^4+\frac{h_{22}}{3}x_1^3x_2+\frac{h_{13}}{2}x_1^2x_2^2+h_{04}x_1x_2^3,
\end{aligned}
\end{array}
\end{equation*}
system \eqref{2.6} can be written as
\begin{equation}
\begin{array}{ll}\label{2.7}
\left\{
\begin{aligned}
\dot{x_1}&=x_2+\sum\limits_{i+j=5}k_{ij}x_1^{i}x_2^{j}+o(|x_1,x_2|^{5}),\\[2ex]
\dot{x_2}&=l_{20}x_1^2+l_{11}x_1x_2+l_{30}x_1^3+l_{31}x_1^3x_2+l_{40}x_1^4\\[2ex]
&\quad+l_{41}x_1^4x_2+\sum\limits_{i+j=5}l_{ij}x_1^{i}x_2^{j}+o(|x_1,x_2|^{5}),
\end{aligned}
\right.
\end{array}
\end{equation}
where
\begin{equation*}
\begin{array}{ll}
\begin{aligned}
k_{50}&=e_{50}-\frac{1}{6}h_{20}(2e_{22}+h_{13}),\quad k_{32}=e_{32}-3e_{04}h_{20}-h_{11}(e_{13}+h_{04}),\\[2ex]
k_{23}&=e_{23}-3e_{04}h_{11},\quad k_{41}=e_{41}-\frac{1}{6}h_{11}(2e_{22}+h_{13})-h_{20}(e_{13}+h_{04}),\\[2ex]
k_{14}&=e{14},\quad k_{05}=e_{05},\quad l_{20}=h_{20},\quad l_{11}=h_{11},\quad l_{30}=h_{30},\quad l_{21}=h_{21},\\[2ex]
l_{40}&=h_{40},\quad l_{31}=h_{31}+4e_{40},\quad l_{50}=\frac{1}{6}h_{20}(3e_{31}-h_{22})-e_{40}h_{11}+h_{50},\\[2ex]
l_{41}&=\frac{1}{12}\big(h_{11}(3e_{31}+h_{22})+8h_{20}(e_{22}-h_{13})\big)+h_{41},\quad l_{14}=h_{14}+e_{04}h_{11},\\[2ex]
l_{32}&=\frac{1}{3}h_{11}(e_{22}-h_{13})+h_{20}(e_{13}-2h_{04})+h_{32},\quad l_{05}=h_{05},\\[2ex]
l_{23}&=\frac{1}{2}h_{11}(e_{13}+h_{04})-2h_{04}h_{11}+2e_{04}h_{20}+h_{23}.
\end{aligned}
\end{array}
\end{equation*}
Finally, setting
\begin{equation*}
\begin{array}{ll}
\begin{aligned}
x_1&=y_1+\frac{4k_{41}+l_{32}}{20}y_1^5+\frac{3k_{32}+l_{23}}{12}y_1^4y_2+\frac{2k_{23}+l_{14}}{6}y_1^3y_2^2\\[2ex]
&\quad+\frac{k_{14}+l_{05}}{2}y_1^2y_2^3+k_{05}y_2^5,\\[2ex]
x_2&=y_2-k_{50}y_1^5+\frac{l_{32}}{4}y_1^4y_2+\frac{l_{23}}{3}y_1^3y_2^2+\frac{l_{14}}{2}y_1^2y_2^3+l_{05}y_1y_2^4,
\end{aligned}
\end{array}
\end{equation*}
system \eqref{2.7} can be transformed as
\begin{equation}
\begin{array}{ll}\label{2.8}
\left\{
\begin{aligned}
\dot{y_1}&=y_2+o(|y_1,y_2|^{5}),\\[2ex]
\dot{y_2}&=m_{20}y_1^2+m_{11}y_1y_2+m_{30}y_1^3+m_{21}y_1^2y_2+m_{40}y_1^4\\[2ex]
&\quad+m_{31}y_1^3y_2+m_{50}y_1^5+m_{41}y_1^4y_2+o(|y_1,y_2|^{5}),
\end{aligned}
\right.
\end{array}
\end{equation}
where
\begin{equation*}
\begin{array}{ll}
\begin{aligned}
m_{20}&=l_{11},\quad m_{11}=l_{11},\quad m_{30}=l_{30},\quad m_{21}=l_{21},\quad m_{40}=l_{40},\\[2ex]
m_{31}&=l_{31},\quad m_{50}=l_{50},\quad m_{41}=5k_{50}+l_{41}.
\end{aligned}
\end{array}
\end{equation*}

From above analysis, system \eqref{2.5} is locally topologically equivalent to \eqref{2.8}. Since $d_{11}=0$ and $d_{20}\neq0$ in \eqref{2.5}, the following results hold by \cite{SNCD}.
\begin{lemma}\label{t4}
When $\alpha=\alpha_0=\frac{(2+\gamma)(1-\gamma\eta)}{(1+\gamma)(2+3\gamma)}$, $\beta=\beta_0=\frac{4(1+\gamma)(1-\gamma\eta)^2}{(2+3\gamma)^2}$ and $\delta=\delta_0=\frac{\gamma^2(1-\gamma\eta)}{3\gamma^2+5\gamma+2}$, then system \eqref{2.1} is locally topologically equivalent
\begin{equation}
\begin{array}{ll}\label{2.9}
\left\{
\begin{aligned}
\dot{x}&=y,\\[2ex]
\dot{y}&=x^2+Mx^3y+Nx^4y+o(|x,y|^5),
\end{aligned}
\right.
\end{array}
\end{equation}
where $M$ and $N$ can be expressed by $c_{ij}$ and $d_{ij}$. Moreover, $E^*$  is a cusp of codimension 3 of system \eqref{2.1} if $0<\gamma<\frac{1}{\eta}$ and $\eta\neq\frac{\gamma^2+8\gamma+8}{4\gamma^3+19\gamma^2+20\gamma+4}$, i.e., $M\neq0$, a cusp of  codimension 4 if $\eta=\eta_0=\frac{\gamma^2+8\gamma+8}{4\gamma^3+19\gamma^2+20\gamma+4}$, i.e., $M=0$ and $N\neq0$.
\end{lemma}
The detailed proof of Lemma \ref{t4} is given in Appendix B.
When $\alpha=\beta$, $d_{11}^*=0$ and $d_{20}^*=-\frac{\gamma^3(2+\gamma)}{8(1+\gamma)^3}$, we have the following result.
\begin{lemma}\label{t11}
If $\alpha=\beta$, $\beta=\frac{(2+\gamma)^2}{4(1+\gamma)^3}$, $\delta=\frac{\gamma^2(2+\gamma)}{4(1+\gamma)^3}$ and $\eta=\frac{\gamma}{4(1+\gamma)^2}$, then system \eqref{2.1} is locally topologically equivalent
\begin{equation}
\begin{array}{ll}\label{4.2}
\left\{
\begin{aligned}
\dot{x}&=y,\\[2ex]
\dot{y}&=x^2+\bar{M}x^3y+\bar{N}x^4y+o(|x,y|^5),
\end{aligned}
\right.
\end{array}
\end{equation}
where
\begin{equation*}
\begin{array}{ll}
\begin{aligned}
\bar{M}=\frac{4\sqrt{2}(1+\gamma)^{11/2}(8+7\gamma)\sqrt{\gamma(2+\gamma)}}{\gamma^3(2+\gamma)^2(3+2\gamma)}>0.
\end{aligned}
\end{array}
\end{equation*}
Thus, $E^{**}$ is a cusp of codimension 3 of system \eqref{2.1}.
\end{lemma}
The proof is similarly to that of Lemma \ref{t4}, and is omitted here.

\section{Bifurcation analysis}

In this section, we focus on bifurcations especially nilpotent bifurcations of system \eqref{2.1}.

\subsection{Saddle-node bifurcation}

 From Lemma \ref{l2},
\begin{equation*}
\begin{array}{ll}
\begin{aligned}
SN_1=\{(\gamma,\alpha,\beta,\delta,\eta):\beta=\frac{1}{4}(1+\alpha)^2,0<\alpha<1,\gamma>0,\delta>0,\eta>0\}
\end{aligned}
\end{array}
\end{equation*}
is a saddle-node bifurcation surface. When the parameters vary from one side of the surface $SN_1$ to the other side, the number of boundary equilibria of system \eqref{2.1} changes from zero to two, and the two equilibria are a hyperbolic saddle and a node. $SN_1$ is the first saddle-node bifurcation surface of system \eqref{2.1}. The biological interpretation for the first saddle-node bifurcation is that $\beta=\frac{1}{4}(1+\alpha)^2$, the prey species is driven to extinction, and the system collapses for $\beta>\frac{1}{4}(1+\alpha)^2$, but the
prey species does not go to extinction for some initial values when $0<\beta<\frac{1}{4}(1+\alpha)^2$. Besides, from Lemma \ref{t9}, we know that the surface
\begin{equation*}
\begin{array}{ll}
\begin{aligned}
SN_2&=\bigg\{(\gamma,\alpha,\beta,\delta,\eta):\beta:=\beta^*=\frac{(1+\alpha+\alpha\gamma-\gamma\eta)^2}{4(1+\gamma)},\delta\neq\frac{\gamma(1-\alpha-\alpha\gamma-\gamma\eta)}{2(1+\gamma)},0<\alpha<1,\\[2ex]
&\qquad0<\gamma<\frac{1-\alpha}{\alpha+\eta}\bigg\}
\end{aligned}
\end{array}
\end{equation*}
is the saddle-node bifurcation surface too. Two positive equilibria arise from the saddle-node bifurcation, which implies that there exists a critical Allee effect constant $\beta^*$ such that the predator species goes extinct when the Allee effect constant $\beta$ is greater than $\beta^*$, and coexistence for system \eqref{2.1} is certain in the form of a positive equilibrium for certain choices of initial values when $(\gamma,\alpha,\beta,\delta,\eta)\in SN_2$.

\begin{lemma}\label{l16}
System \eqref{2.1} undergoes saddle-node bifurcation as $(\gamma,\alpha,\beta,\delta,\eta)$ varies near
 $SN_1$ or $SN_2$.
\end{lemma}

\subsection{Bogdanov-Takens bifurcation}

From Lemma \ref{t4}, it follows  that system \eqref{2.1} may exhibit cusp type Bogdanov-Takens bifurcation of codimension 4 near $E^*$. In order to make sure  such degenerate Bogdanov-Takens bifurcation can be fully unfolding inside the class of system \eqref{2.1}, we choose $\alpha$, $\beta$, $\delta$ and $\eta$ as bifurcation parameters, and consider the following system
\begin{equation}
\begin{array}{ll}\label{2.14}
\left\{
\begin{aligned}
\dot{x}&=x(1-x)-\gamma xy-\frac{(\beta_0+\lambda_1) x}{x+\alpha_0+\lambda_2},\\[2ex]
\dot{y}&=y(\delta_0+\lambda_3)(1-\frac{y}{x+\eta_0+\lambda_4}),
\end{aligned}
\right.
\end{array}
\end{equation}
where $\lambda=(\lambda_1,\lambda_2,\lambda_3,\lambda_4)=(0,0,0,0)$. If we can transform system \eqref{2.14} as follows
\begin{equation}
\begin{array}{ll}\label{2.15}
\left\{
\begin{aligned}
\dot{x}&=y,\\[2ex]
\dot{y}&=\sigma_1+\sigma_2y+\sigma_3xy+\sigma_4x^3y+x^2-x^4y+R(x,y,\lambda),
\end{aligned}
\right.
\end{array}
\end{equation}
where
\begin{equation}
\begin{array}{ll}\label{2.16}
\begin{aligned}
R(x,y,\lambda)&=y^2O(|x,y|^2)+O(|x,y|^6)+O(\lambda)\big(O(y^2)+O(|x,y|^3)\big)\\[2ex]
&\quad+O(\lambda^2)O(|x,y|),
\end{aligned}
\end{array}
\end{equation}
and validate $\frac{\partial(\sigma_1,\sigma_2,\sigma_3,\sigma_4)}{\partial(\lambda_1,\lambda_2,\lambda_3,\lambda_4)}\big|_{\lambda=0}\neq0$, then we conclude that system \eqref{2.14} undergoes cusp type Bogdanov-Takens bifurcation of codimension 4. Based on this, we establish the following findings:
\begin{theorem}\label{t5}
System \eqref{2.1} can undergo the cusp bifurcation of codimension 4 near $E^*$ as $(\alpha,\gamma,\delta,\eta)$ varies near $(\alpha_0,\gamma_0,\delta_0,\eta_0)$. There are a series of bifurcation with codimension 2, 3 and 4 originating from $E^*$:
\begin{enumerate}
\item[\bf(i)] If $\beta=\frac{(1+\alpha+\alpha\gamma-\gamma\eta)^2}{4(1+\gamma)}$, $\delta=\frac{\gamma(1-\alpha-\alpha\gamma-\gamma\eta)}{2(1+\gamma)}$, then system \eqref{2.1} can undergo the cusp  bifurcation of codimension 2 near $E^*$ for $\alpha<1$ and $\gamma<\frac{1-\alpha}{\alpha+\eta}$.

\item[\bf(ii)] If $\beta=\frac{4(1+\gamma)(1-\gamma\eta)^2}{(2+3\gamma)^2}$, $\delta=\frac{\gamma^2(1-\gamma\eta)}{(1+\gamma)(2+3\gamma)}$ and $\alpha=\frac{(2+\gamma)(1-\gamma\eta)}{3\gamma^2+5\gamma+2}$, then system \eqref{2.1} can undergo the cusp  bifurcation of codimension 3 near $E^*$ for $\eta<\frac{1}{\gamma}$.

\item[\bf(iii)] If $\beta=\frac{4(1+\gamma)^3(2+\gamma)^2}{(4\gamma^3+19\gamma^2+20\gamma+4)^2}$, $\delta=\frac{\gamma^2(2+\gamma)}{4\gamma^3+19\gamma^2+20\gamma+4}$, $\alpha=\frac{(2+\gamma)^2}{4\gamma^3+19\gamma^2+20\gamma+4}$ and $\eta=\frac{\gamma^2+8\gamma+\beta}{4\gamma^3+19\gamma^2+20\gamma+4}$, then system \eqref{2.1} can undergo the cusp  bifurcation of codimension 4 near $E^*$.
\end{enumerate}
\end{theorem}
\begin{proof}
Firstly, translating $E^*$ to $(0,0)$ by $x=X+\frac{\gamma(2+\gamma)}{4\gamma^3+19\gamma^2+20\gamma+4}$ and $y=Y+\frac{2(1+\gamma)(4+\gamma)}{4\gamma^3+19\gamma^2+20\gamma+4}$ and expanding system \eqref{2.14} in power series near $(0,0)$, then system \eqref{2.14} can be changed into
\begin{equation}
\begin{array}{ll}\label{2.17}
\left\{
\begin{aligned}
\dot{X}&=\bar{a}_{00}+\bar{a}_{10}X+\bar{a}_{01}Y+\bar{a}_{11}XY+\bar{a}_{20}X^2+\bar{a}_{30}X^3\\[2ex]
&\quad+\bar{a}_{40}X^4+\bar{a}_{50}X^5+o(|X,Y|^5),\\[2ex]
\dot{Y}&=\bar{b}_{00}+\bar{b}_{10}X+\bar{b}_{01}Y+\bar{b}_{20}X^2+\bar{b}_{11}XY+\bar{b}_{02}Y^2+\bar{b}_{30}X^3\\[2ex]
&\quad+\bar{b}_{21}X^2Y+\bar{b}_{12}XY^2+\bar{b}_{40}X^4+\bar{b}_{31}X^3Y+\bar{b}_{22}X^2Y^2\\[2ex]
&\quad+\bar{b}_{50}X^5+\bar{b}_{32}X^3Y^2+\bar{b}_{41}X^4Y+o(|X,Y|^5),
\end{aligned}
\right.
\end{array}
\end{equation}
where $\bar{a}_{ij}$ and $\bar{b}_{ij}$ are given in Appendix F.

Secondly, setting $u=X$ and $v=\dot{X}$, system \eqref{2.17} can be written as
\begin{equation}
\begin{array}{ll}\label{2.18}
\left\{
\begin{aligned}
\dot{u}&=v,\\[2ex]
\dot{v}&=\bar{c}_{00}+\bar{c}_{10}u+\bar{c}_{01}v+\bar{c}_{20}u^2+\bar{c}_{11}uv+\bar{c}_{02}v^2+c_{30}u^3\\[2ex]
&\quad+\bar{c}_{21}u^2v+\bar{c}_{12}uv^2+\bar{c}_{40}u^4+\bar{c}_{31}u^3v+\bar{c}_{22}u^2v^2\\[2ex]
&\quad+\bar{c}_{50}u^5+\bar{c}_{32}u^3v^2+\bar{c}_{41}u^4v+o(|u,v|^5),
\end{aligned}
\right.
\end{array}
\end{equation}
where $\bar{c}_{ij}$ are given in Appendix G.

Thirdly, letting $u=x+\frac{\bar{c}_{02}}{2}x^2$ and $v=y+\bar{c}_{02}xy$, system \eqref{2.18} can be transformed into
\begin{equation}
\begin{array}{ll}\label{2.19}
\left\{
\begin{aligned}
\dot{x}&=y,\\[2ex]
\dot{y}&=\bar{d}_{00}+\bar{d}_{10}x+\bar{d}_{01}y+\bar{d}_{20}x^2+\bar{d}_{11}xy+\bar{d}_{30}x^3+\bar{d}_{12}xy^2+\bar{d}_{40}x^4+\bar{d}_{50}x^5\\[2ex]
&\quad+\bar{d}_{21}x^2y+\bar{d}_{31}x^3y+\bar{d}_{22}x^2y^2+\bar{d}_{32}x^3y^2+\bar{d}_{41}x^4y+o(|x,y|^5),
\end{aligned}
\right.
\end{array}
\end{equation}
where
\begin{equation*}
\begin{array}{ll}
\begin{aligned}
\bar{d}_{00}&=\bar{c}_{00},\quad \bar{d}_{10}=\bar{c}_{10}-\bar{c}_{00}\bar{c}_{02},\quad \bar{d}_{01}=\bar{c}_{01},\quad \bar{d}_{20}=\bar{c}_{20}+\bar{c}_{00}\bar{c}_{02}^2-\frac{\bar{c}_{02}\bar{c}_{10}}{2},\\[2ex]
\bar{d}_{30}&=\bar{c}_{30}-\bar{c}_{00}\bar{c}_{02}^3+\frac{\bar{c}_{02}^2\bar{c}_{10}}{2},\quad \bar{d}_{21}=\bar{c}_{21}+\frac{\bar{c}_{02}\bar{c}_{11}}{2},\quad \bar{d}_{12}=\bar{c}_{12}+2\bar{c}_{02}^2,\quad \bar{d}_{11}=\bar{c}_{11},\\[2ex]
\bar{d}_{40}&=\bar{c}_{40}+\bar{c}_{00}\bar{c}_{02}^4-\frac{\bar{c}_{02}(\bar{c}_{02}^2\bar{c}_{10-\bar{c}_{30}})}{2}+\frac{\bar{c}_{02}^2\bar{c}_{20}}{4},\quad \bar{d}_{22}=\bar{c}_{22}-\bar{c}_{02}^3+\frac{3\bar{c}_{02}\bar{c}_{12}}{2},\\[2ex]
\bar{d}_{50}&=\bar{c}_{50}-\bar{c}_{02}(\bar{c}_{00}\bar{c}_{02}^4-\bar{c}_{40})+\frac{1}{4}\bar{c}_{02}^2(2\bar{c}_{02}^2\bar{c}_{10}-\bar{c}_{02}\bar{c}_{20}+\bar{c}_{30}),\quad \bar{d}_{31}=\bar{c}_{31}+\bar{c}_{02}\bar{c}_{21},\\[2ex] \bar{d}_{41}&=\bar{c}_{41}+\frac{1}{4}\bar{c}_{02}(\bar{c}_{02}\bar{c}_{21}+6\bar{c}_{31}),\quad \bar{d}_{32}=\bar{c}_{32}+2\bar{c}_{02}\bar{c}_{22}+\bar{c}_{02}^4+\frac{\bar{c}_{02}^2\bar{c}_{12}}{2}.
\end{aligned}
\end{array}
\end{equation*}

Fourthly, making $x=x_1+\frac{\bar{d}_{12}}{6}x_1^3$ and $y=y_1+\frac{\bar{d}_{12}}{2}x_1^2y_1$, system \eqref{2.19} can be written as
\begin{equation}
\begin{array}{ll}\label{2.20}
\left\{
\begin{aligned}
\dot{x_1}&=y_1,\\[2ex]
\dot{y_1}&=\bar{e}_{00}+\bar{e}_{10}x_1+\bar{e}_{01}y_1+\bar{e}_{20}x_1^2+\bar{e}_{11}x_1y_1+\bar{e}_{30}x_1^3+\bar{e}_{21}x_1^2y_1+\bar{e}_{31}x_1^3y_1\\[2ex]
&\quad+\bar{e}_{40}x_1^4+\bar{e}_{22}x_1^2y_1^2+\bar{e}_{50}x_1^5+\bar{e}_{32}x_1^3y_1^2+\bar{e}_{41}x_1^4y_1+o(|x_1,y_1|^5),
\end{aligned}
\right.
\end{array}
\end{equation}
where
\begin{equation*}
\begin{array}{ll}
\begin{aligned}
\bar{e}_{00}&=\bar{d}_{00},\quad \bar{e}_{10}=\bar{d}_{10},\quad \bar{e}_{01}=\bar{d}_{01},\quad \bar{e}_{20}=\bar{d}_{20}-\frac{\bar{d}_{00}\bar{d}_{12}}{2},\quad \bar{e}_{11}=\bar{d}_{11},\\[2ex]
\bar{e}_{30}&=\bar{d}_{30}-\frac{\bar{d}_{10}\bar{d}_{12}}{3},\quad \bar{e}_{40}=\bar{d}_{40}-\frac{\bar{d}_{00}\bar{d}_{12}^2}{4}-\frac{\bar{d}_{12}\bar{d}_{20}}{6},\quad \bar{e}_{31}=\bar{d}_{31}+\frac{\bar{d}_{11}\bar{d}_{12}}{6},\\[2ex]
\bar{e}_{50}&=\bar{d}_{50}+\frac{\bar{d}_{10}\bar{d}_{12}^2}{5},\quad \bar{e}_{41}=\bar{d}_{41}+\frac{\bar{d}_{12}\bar{d}_{21}}{3},\quad \bar{e}_{32}=\bar{d}_{32}+\frac{7\bar{d}_{12}^2}{6},\quad \bar{e}_{21}=\bar{d}_{21}.
\end{aligned}
\end{array}
\end{equation*}

Fifthly, letting $x_1=u+\frac{\bar{e}_{22}}{12}u^4$ and $y_1=v+\frac{\bar{e}_{22}}{3}u^3v$, system\eqref{2.20} can be changed into
\begin{equation}
\begin{array}{ll}\label{2.21}
\left\{
\begin{aligned}
\dot{u}&=v,\\[2ex]
\dot{v}&=\bar{h}_{00}+\bar{h}_{10}u+\bar{d}_{01}v+\bar{h}_{20}u^2+\bar{h}_{11}uv+\bar{h}_{30}u^3+\bar{h}_{21}u^2y_1+\bar{h}_{31}u^3v\\[2ex]
&\quad+\bar{h}_{40}u^4+\bar{h}_{50}u^5+\bar{h}_{32}u^3v^2+\bar{h}_{41}u^4v+o(|u,v|^5),
\end{aligned}
\right.
\end{array}
\end{equation}
where
\begin{equation*}
\begin{array}{ll}
\begin{aligned}
\bar{h}_{00}&=\bar{e}_{00},\quad \bar{h}_{10}=\bar{e}_{10},\quad \bar{h}_{01}=\bar{e}_{01},\quad \bar{h}_{20}=\bar{e}_{20},\quad \bar{h}_{11}=\bar{e}_{11},\\[2ex]
\bar{h}_{30}&=\bar{e}_{30}-\frac{\bar{e}_{00}\bar{e}_{22}}{3},\quad \bar{h}_{21}=\bar{e}_{21},\quad \bar{h}_{40}=\bar{e}_{40}-\frac{\bar{e}_{10}\bar{e}_{22}}{4},\quad \bar{h}_{31}=\bar{e}_{31},\\[2ex]
\bar{h}_{50}&=\bar{e}_{50}-\frac{\bar{e}_{20}\bar{e}_{22}}{5},\quad \bar{h}_{41}=\bar{e}_{41}+\frac{\bar{e}_{11}\bar{e}_{22}}{12},\quad \bar{h}_{32}=\bar{e}_{32}.
\end{aligned}
\end{array}
\end{equation*}

Sixthly, setting $u=x+\frac{\bar{h}_{32}}{20}x^5$ and $v=y+\frac{\bar{h}_{32}}{4}x^4y$, then system \eqref{2.21} can be transformed into
\begin{equation}
\begin{array}{ll}\label{2.22}
\left\{
\begin{aligned}
\dot{x}&=y,\\[2ex]
\dot{y}&=\bar{k}_{00}+\bar{k}_{10}x+\bar{k}_{01}y+\bar{k}_{20}x^2+\bar{k}_{11}xy+\bar{k}_{30}x^3+\bar{k}_{21}x^2y+\bar{k}_{31}x^3y\\[2ex]
&\quad+\bar{k}_{40}x^4+\bar{k}_{50}x^5+\bar{k}_{41}x^4y+o(|x,y|^5),
\end{aligned}
\right.
\end{array}
\end{equation}
where
\begin{equation*}
\begin{array}{ll}
\begin{aligned}
\bar{k}_{00}&=\bar{h}_{00},\quad \bar{k}_{10}=\bar{h}_{10},\quad \bar{k}_{01}=\bar{h}_{01},\quad \bar{k}_{20}=\bar{h}_{20},\quad \bar{k}_{11}=\bar{h}_{11},\quad \bar{k}_{30}=\bar{h}_{30},\\[2ex]
\bar{k}_{21}&=\bar{h}_{21},\quad \bar{k}_{40}=\bar{h}_{40}-\frac{\bar{h}_{00}\bar{h}_{32}}{4},\quad \bar{k}_{31}=\bar{h}_{31},\quad \bar{k}_{50}=\bar{h}_{50}-\frac{\bar{h}_{10}\bar{h}_{32}}{5},\quad \bar{k}_{41}=\bar{h}_{41}.
\end{aligned}
\end{array}
\end{equation*}

Seventhly, making
\begin{equation*}
\begin{array}{ll}
\begin{aligned}
x&=X-\frac{\bar{k}_{30}}{4\bar{k}_{20}}X^2+\frac{15\bar{k}_{30}^2-16\bar{k}_{20}\bar{k}_{40}}{80\bar{k}_{20}^2}X^3
+\frac{336\bar{k}_{20}\bar{k}_{30}\bar{k}_{40}-175\bar{k}_{30}^3-160\bar{k}_{20}^2\bar{k}_{50}}{960\bar{k}_{20}^3}X^4,\\[2ex]
y&=Y,\\[2ex]
t&=(1-\frac{\bar{k}_{30}}{2\bar{k}_{20}}X+\frac{45\bar{k}_{30}^2-48\bar{k}_{40}}{80\bar{k}_{20}^2}X^2
+\frac{336\bar{k}_{20}\bar{k}_{30}\bar{k}_{40}-175\bar{k}_{30}^3-160\bar{k}_{20}^2\bar{k}_{50}}{240\bar{k}_{20}^3}X^3)\tau,
\end{aligned}
\end{array}
\end{equation*}
system \eqref{2.22} can be written as (still denote $\tau$ by $t$)
\begin{equation}
\begin{array}{ll}\label{2.23}
\left\{
\begin{aligned}
\dot{X}&=Y,\\[2ex]
\dot{X}&=\bar{l}_{00}+\bar{l}_{10}X+\bar{l}_{01}Y+\bar{l}_{20}X^2+\bar{l}_{11}XY+\bar{l}_{30}X^3+\bar{l}_{21}X^2Y+\bar{l}_{40}X^4\\[2ex]
&\quad+\bar{l}_{31}X^3Y+\bar{l}_{50}X^5+\bar{l}_{41}X^4Y+o(|X,Y|^5),
\end{aligned}
\right.
\end{array}
\end{equation}
where $\bar{l}_{ij}$ are given in Appendix H.

Eighthly, letting $X=x_2$, $Y=y_2+\frac{\bar{l}_{21}}{3\bar{l}_{20}}y_2^2+\frac{\bar{l}_{21}^2}{36\bar{l}_{20}^2}y_2^3$ and $t=(1+\frac{\bar{l}_{21}}{3\bar{l}_{20}}y_2+\frac{\bar{l}_{21}^2}{36\bar{l}_{20}^2}y_2^2)\tau$, system \eqref{2.23} can be transformed into (still denote $\tau$ by $t$)
\begin{equation}
\begin{array}{ll}\label{2.24}
\left\{
\begin{aligned}
\dot{x_2}&=y_2,\\[2ex]
\dot{y_2}&=\bar{m}_{00}+\bar{m}_{10}x_2+\bar{m}_{01}y_2+\bar{m}_{20}x_2^2+\bar{m}_{11}x_2y_2\\[2ex]
&\quad+\bar{m}_{31}x_2^3y_2+\bar{m}_{41}x_2^4y_2+o(|x_2,y_2|^5),
\end{aligned}
\right.
\end{array}
\end{equation}
where
\begin{equation*}
\begin{array}{ll}
\begin{aligned}
\bar{m}_{00}&=l_{00},\quad \bar{m}_{10}=\bar{l}_{10},\quad \bar{m}_{01}=\bar{l}_{01}-\frac{\bar{l}_{00}\bar{l}_{21}}{\bar{l}_{20}},\quad \bar{m}_{20}=\bar{l}_{20},\\[2ex]
\bar{m}_{11}&=\bar{l}_{11}-\frac{\bar{l}_{10}\bar{l}_{21}}{\bar{l}_{20}},\quad \bar{m}_{31}=\bar{l}_{31}-\frac{\bar{l}_{21}\bar{l}_{30}}{\bar{l}_{20}},\quad \bar{m}_{41}=\bar{l}_{41}-\frac{\bar{l}_{21}\bar{l}_{40}}{\bar{l}_{20}}.
\end{aligned}
\end{array}
\end{equation*}

Ninthly, setting $x_2=\bar{m}_{20}^{\frac{1}{7}}\bar{m}_{41}^{-\frac{2}{7}}u$, $y_2=-\bar{m}_{20}^{\frac{5}{7}}\bar{m}_{41}^{-\frac{3}{7}}v$ and $t=-\bar{m}_{20}^{-\frac{4}{7}}\bar{m}_{41}^{\frac{1}{7}}\tau$, system \eqref{2.24} can be changed into
\begin{equation}
\begin{array}{ll}\label{2.25}
\left\{
\begin{aligned}
\dot{u}&=v,\\[2ex]
\dot{v}&=\bar{n}_{00}+\bar{n}_{10}u+\bar{n}_{01}v+\bar{n}_{11}uv+\bar{n}_{31}u^3v+u^2-u^4v+o(|u,v|^5),
\end{aligned}
\right.
\end{array}
\end{equation}
where
\begin{equation*}
\begin{array}{ll}
\begin{aligned}
\bar{n}_{00}&=\frac{\bar{m}_{00}\bar{m}_{41}^{\frac{4}{7}}}{\bar{m}_{20}^{\frac{9}{7}}},\quad \bar{n}_{10}=\frac{\bar{m}_{10}\bar{m}_{41}^{\frac{2}{7}}}{\bar{m}_{20}^{\frac{8}{7}}},
\quad \bar{n}_{01}=-\frac{\bar{m}_{01}\bar{m}_{41}^{\frac{1}{7}}}{\bar{m}_{20}^{\frac{4}{7}}},\\[2ex]
\bar{n}_{11}&=-\frac{\bar{m}_{11}}{\bar{m}_{20}^{\frac{3}{7}}\bar{m}_{41}^{\frac{1}{7}}},\quad \bar{n}_{31}=-\frac{\bar{m}_{31}}{\bar{m}_{20}^{\frac{1}{7}}\bar{m}_{41}^{\frac{5}{7}}}.
\end{aligned}
\end{array}
\end{equation*}

Finally, making $u=x-\frac{n_{10}}{2}$ and $v=y$, system \eqref{2.26} can be written as
\begin{equation}
\begin{array}{ll}\label{2.26}
\left\{
\begin{aligned}
\dot{x}&=y,\\[2ex]
\dot{y}&=\bar{\chi}_1+\bar{\chi}_2y+\bar{\chi}_3xy+\bar{\chi}_4x^3y+x^2-x^4y+o(|x,y|^5),
\end{aligned}
\right.
\end{array}
\end{equation}
where
\begin{equation*}
\begin{array}{ll}
\begin{aligned}
\bar{\chi}_1&=n_{00}-\frac{n_{10}^2}{4},\quad \bar{\chi}_2=n_{01}-\frac{n_{10}^4}{16}-\frac{n_{10}^3n_{31}}{8}-\frac{n_{10}n_{11}}{2},\\[2ex]
\bar{\chi}_3&=n_{11}+\frac{n_{10}^3}{2}+\frac{3n_{10}^2n_{31}}{4},\quad \bar{\chi}_4=2n_{10}+n_{31}.
\end{aligned}
\end{array}
\end{equation*}
Through calculation, we have
\begin{equation}
\begin{array}{ll}\label{2.27}
\begin{aligned}
\bigg|\frac{\partial(\bar{\chi}_1,\bar{\chi}_2,\bar{\chi}_3,\bar{\chi}_4)}{\partial(\lambda_1,\lambda_2,\lambda_3,\lambda_4)}\bigg|_{\lambda=0}\neq0.
\end{aligned}
\end{array}
\end{equation}
Therefore, we conclude that system \eqref{2.1} can undergo Bogdanov-Takens bifurcation of codimension 4.
\end{proof}

For the case $\alpha=\beta$, we can prove that system \eqref{2.1} may undergo the cusp type Bogdanov-Takens bifurcation of codimension 3 near $E^{**}$. Similar to Theorem \ref{t5}, we have the following result.
\begin{theorem}\label{t12}
System \eqref{2.1} can undergo the cusp bifurcation of codimension 3 near $E^{**}$ as $(\beta,\delta,\eta)$ varies near $(\frac{(2+\gamma)^2}{4(1+\gamma)^3},\frac{\gamma^2(2+\gamma)}{4(1+\gamma)^3},\frac{\gamma}{4(1+\gamma)^2})$. There exist a series of bifurcation with codimension 2 and 3 originating from $E^{**}$.

(i) If $\beta=\frac{1+\gamma\eta-2\sqrt{\gamma\eta}}{1+\eta}$, $\delta=\frac{\gamma\sqrt{\gamma\eta}-\eta\gamma^2}{1+\gamma}$ and $\eta\neq\frac{\gamma}{4(1+\gamma)^2}$, then system \eqref{2.1} can undergo the cusp  bifurcation of codimension 2 near $E^{**}$ for $0<\gamma<\frac{1}{\eta}$.

(ii) If $\beta=\frac{(2+\gamma)^2}{4(1+\gamma)^3}$, $\delta=\frac{\gamma^2(2+\gamma)}{4(1+\gamma)^3}$ and $\eta=\frac{\gamma}{4(1+\gamma)^2}$, then system \eqref{2.1} can undergo the cusp  bifurcation of codimension 3 near $E^{**}$ for $\gamma>0$.
\end{theorem}

Above analysis shows that bifurcations with high-codimension exhibit rich dynamical behaviors and high sensitivity to parameters. Although we cannot plot the high-dimensional bifurcation diagram, changes in the phase portraits in Figures \ref{z3} and \ref{3cycles} indicate the presence of multiple low-codimension bifurcations.
\begin{figure}[ht!]
\centering
\begin{subfigure}{0.45\linewidth}
\centering
\includegraphics[width=0.9\linewidth]{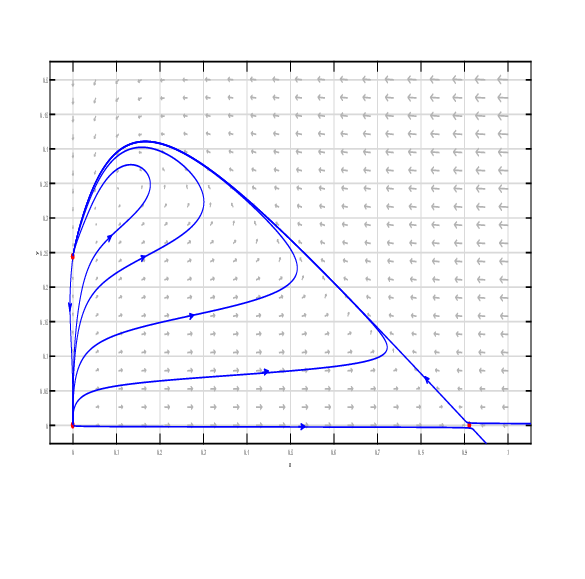}
\put(-101,10){$(a)$}
\end{subfigure}
\centering
\begin{subfigure}{0.45\linewidth}
\centering
\includegraphics[width=0.9\linewidth]{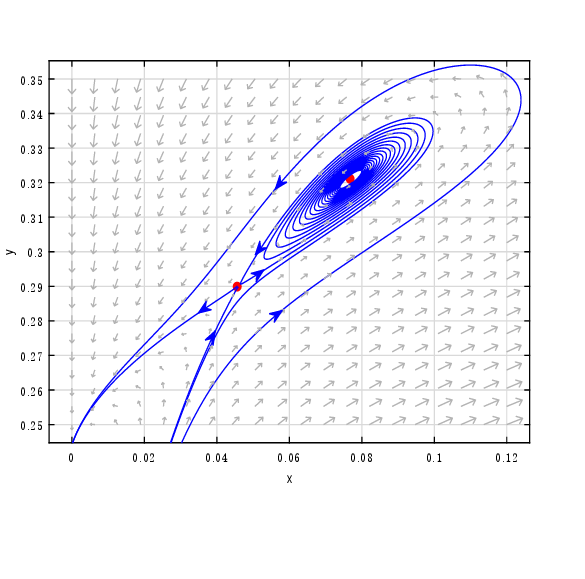}
\put(-90,10){$(b)$}
\end{subfigure}

\centering
\begin{subfigure}{0.45\linewidth}
\centering
\includegraphics[width=0.9\linewidth]{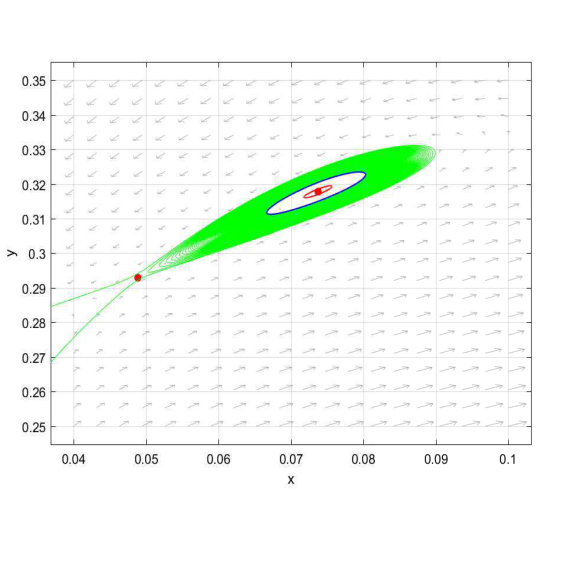}
\put(-101,10){$(c)$}
\end{subfigure}
\centering
\begin{subfigure}{0.45\linewidth}
\centering
\includegraphics[width=0.9\linewidth]{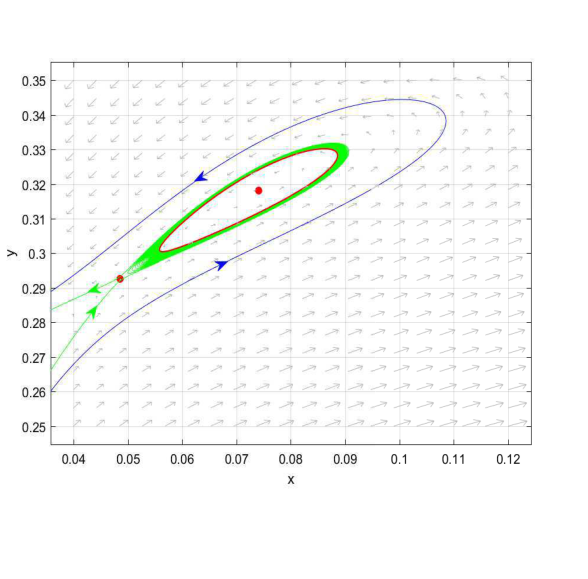}
\put(-90,10){$(d)$}
\end{subfigure}

\centering
\begin{subfigure}{0.45\linewidth}
\centering
\includegraphics[width=0.9\linewidth]{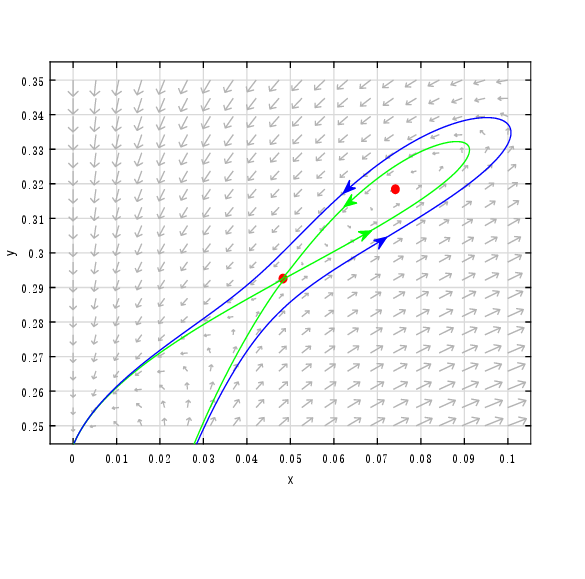}
\put(-101,10){$(e)$}
\end{subfigure}
\centering
\begin{subfigure}{0.45\linewidth}
\centering
\includegraphics[width=0.9\linewidth]{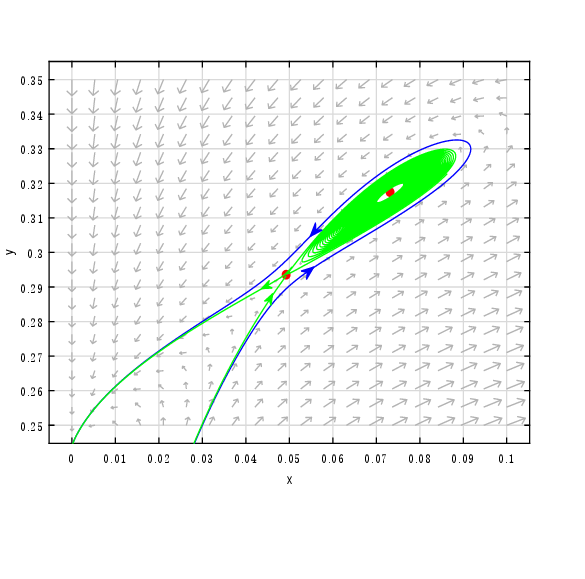}
\put(-90,10){$(f)$}
\vspace{3mm}
\end{subfigure}
\captionsetup{justification=centering}
\caption{For $\alpha=\frac{29542}{225625}$, $\gamma=\frac{3}{2}$, $\delta=\frac{31861}{361000}$ and $\eta=\frac{35239}{144400}$. (a) $\beta=\frac{240960049}{2606420000}$; (b) $\beta=\frac{59816469}{651605000}$; (c) $\beta=\frac{599598221}{6516050000}$; (d) $\beta=\frac{5994679}{65160500}$; (e) $\beta=\frac{59816469}{651605000}$; (f) $\beta=\frac{239917481}{2606420000}$. When $\beta$ decreases, system \eqref{2.1} undergoes successively saddle-node bifurcation, double limit cycle bifurcation, Hopf bifurcation and homoclinic bifurcation.}
\label{z3}
\end{figure}
\begin{figure}
\begin{center}
\begin{overpic}[scale=0.50]{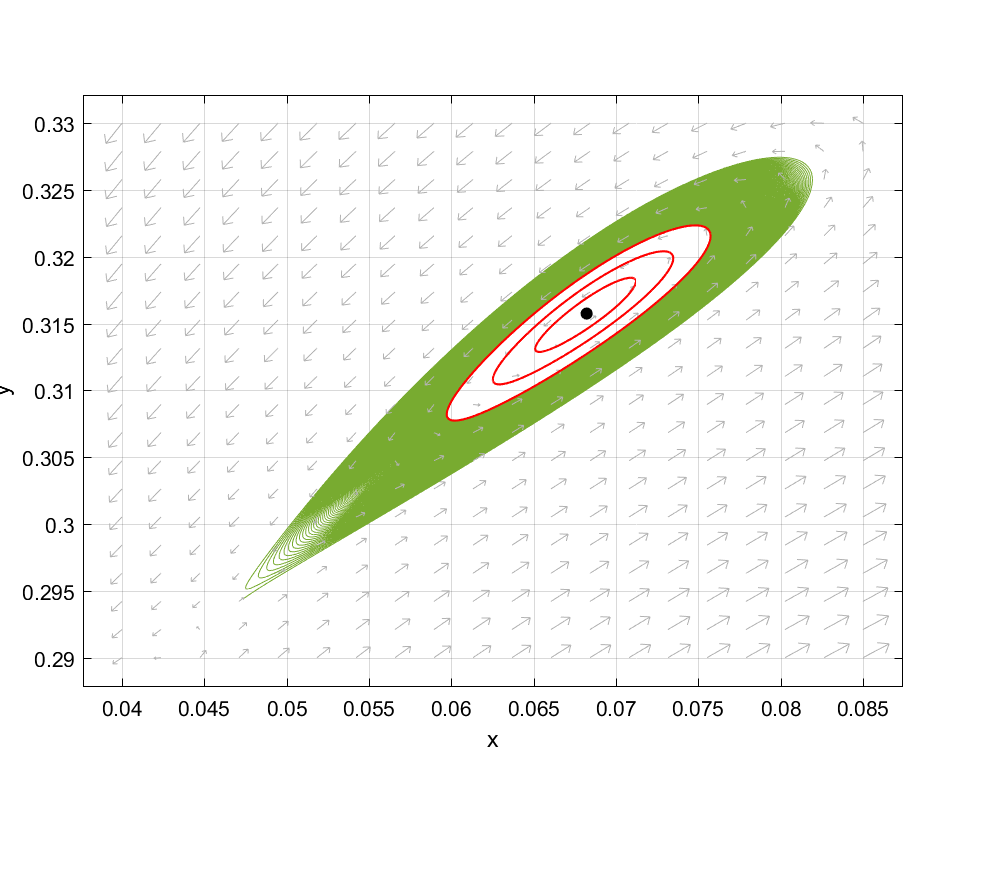}
\end{overpic}
\vspace{-15mm}
\end{center}
\caption{ Three limit cycles enclosing a stable hyperbolic focus for system \eqref{2.1} with $\alpha=\frac{49}{361}$, $\beta=\frac{12250}{130321}$, $\gamma=\frac{3}{2}$, $\delta=\frac{63}{722}$ and $\eta=\frac{886029}{3610000}$, in which the outermost and innermost limit cycles are unstable and the middle one is stable.}
\label{3cycles}
\end{figure}

\section{Hopf bifurcation}

In this section, we consider Hopf bifurcation of system \eqref{2.1} at equilibrium $E_2^{**}$ (or $E_2^*$), which satisfing $Tr(E_2^{**})=0$ and $Det(J(E_2^{**}))>0$ when $\Delta_3>0$ (or $Tr(E_2^{*})=0$ and $Det(J(E_2^{*}))>0$ when $\Delta_2>0$). First, letting $t=(x+\alpha)(x+\eta)\tau$,  system \eqref{2.1} becomes (still denote $\tau$ by $t$)
\begin{equation}
\begin{array}{ll}\label{3.1}
\left\{
\begin{aligned}
\dot{x}&=x(x+\eta)\big((1-x)(x+\alpha)-\gamma y(x+\alpha)-\beta\big),\\[2ex]
\dot{y}&=\delta y(x+\alpha)\big((x+\eta)-y\big).
\end{aligned}
\right.
\end{array}
\end{equation}
To simplify the notation,  let $E_2^{*}=(z,z+\eta)$. From \eqref{2.2} and $Tr(E_2^{*})=0$, $\alpha$ and $\gamma$ can be expresses by $\delta$, $\beta$, $\eta$ and $z$ as follows
\begin{equation}
\begin{array}{ll}\label{4.14}
\begin{aligned}
\alpha&=\alpha_0:=\frac{z(1-2z-z\gamma+\delta+\gamma\eta)}{z+\delta},\\[2ex]
\beta&=\beta_0:=\frac{z(1-z-z\gamma-\gamma\eta)^2}{z+\delta}.
\end{aligned}
\end{array}
\end{equation}

It is not difficult to show that $\alpha_0>0$, $\beta_0>0$ and $Det(J(E_2^{*}))>0$ if and only if $(\delta,\beta,\eta,z)\in\Omega^*$, where
\begin{equation}
\begin{array}{ll}\label{3.2}
\begin{aligned}
\Omega^*:&=\bigg\{(z,\delta,\gamma,\eta)\in R_+^4\bigg|0<z<\frac{1}{2},0<\delta<\frac{z(1-2z)}{2z+\eta},\frac{\delta}{z}<\gamma<\frac{1-2z-\delta}{z+\eta},\eta>0\bigg\}.
\end{aligned}
\end{array}
\end{equation}

\subsection{Case 1: $\alpha=\beta$}

In this case, we have
\begin{equation}
\begin{array}{ll}\label{5.1}
\begin{aligned}
\gamma&=\gamma_0:=\frac{z(1-z-\beta)}{(z+\beta)(z+\eta)},\\[2ex]
\delta&=\delta_0:=\frac{z\big(\beta-(z+\beta)^2\big)}{(z+\beta)^2}.
\end{aligned}
\end{array}
\end{equation}
and
\begin{equation*}
\begin{array}{ll}
\begin{aligned}
\Omega^*:&=\bigg\{(z,\beta,\eta)\in R_+^3\bigg|0<\beta<1,0<z<\sqrt{\beta}-\beta,0<\eta<\frac{z^2}{\beta-(z+\beta)^2}\bigg\}.
\end{aligned}
\end{array}
\end{equation*}
 Making the following linear transformations successively
\begin{equation*}
\begin{array}{ll}
\begin{aligned}
x&=\bar{X}+z,\quad y=\bar{Y}+z+\eta;\\[2ex]
\bar{X}&=\bar{u}-\frac{\sqrt{\beta-\beta^2-2z\beta-z^2}\sqrt{(1+\eta)z^2+2\beta\eta z+\eta\beta^2-\beta\eta}}{(z^2+2\beta z+\beta^2-\beta)\sqrt{z+\eta}}\bar{v},\quad \bar{Y}=\bar{u},\quad t=\frac{\bar{\tau}}{\sqrt{\bar{d}}},
\end{aligned}
\end{array}
\end{equation*}
where $\bar{d}=Det(J(E_2^{**}))=\frac{z^2(z+\eta)(\beta-\beta^2-2z\beta-z^2)\big((1+\eta)z^2+2\beta\eta z+\eta\beta^2-\beta\eta\big)}{(z+\beta)^2}$,  system \eqref{3.1} can be transformed as (still $\bar{\tau}$ by $t$)
\begin{equation}
\begin{array}{ll}\label{4.4}
\left\{
\begin{aligned}
\dot{\bar{u}}&=\bar{v}+\bar{p}_{11}\bar{u}\bar{v}+\bar{p}_{02}\bar{v}^2+\bar{p}_{21}\bar{u}^2\bar{v}+\bar{p}_{12}\bar{u}\bar{v}^2,\\[2ex]
\dot{\bar{v}}&=-\bar{u}+\bar{q}_{20}\bar{u}^2+\bar{q}_{11}\bar{u}\bar{v}+\bar{q}_{02}\bar{v}^2+\bar{q}_{30}\bar{u}^3+\bar{q}_{21}\bar{u}^2\bar{v}+\bar{q}_{12}\bar{u}\bar{v}^2+\bar{q}_{03}\bar{v}^3\\[2ex]
&\quad+\bar{q}_{40}\bar{u}^4+\bar{q}_{31}\bar{u}^3\bar{v}+\bar{q}_{22}\bar{u}^2\bar{v}^2+\bar{q}_{13}\bar{u}\bar{v}^3+\bar{q}_{04}\bar{v}^4,
\end{aligned}
\right.
\end{array}
\end{equation}
where $\bar{p}_{ij}$ and $\bar{q}_{ij}$ are given in Appendix I.

According to Formal Series Method (see \cite{ZT}), we can have the first four focal values as follows
\begin{equation*}
\begin{array}{ll}
\begin{aligned}
L_1&=\frac{(1-z-\beta)L_{11}}{4(z+\beta)(z+\eta)^{3/2}(\beta-\beta^2-2\beta z+z^2)^{3/2}\big((1+\eta)z^2+2\beta\eta z+\eta\beta^2-\beta\eta\big)^{3/2}},\\[2ex]
L_2&=\frac{(1-z-\beta)L_{22}}{48z(z+\beta)^2(z+\eta)^{9/2}(\beta-\beta^2-2\beta z+z^2)^{7/2}\big((1+\eta)z^2+2\beta\eta z+\eta\beta^2-\beta\eta\big)^{7/2}},\\[2ex]
L_3&=\frac{(1-z-\beta)L_{33}}{9216z^4(z+\beta)^5(z+\eta)^{15/2}(\beta-\beta^2-2\beta z+z^2)^{11/2}\big((1+\eta)z^2+2\beta\eta z+\eta\beta^2-\beta\eta\big)^{11/2}},\\[2ex]
L_4&=\frac{(1-z-\beta)L_{44}}{2211840z^6(z+\beta)^7(z+\eta)^{21/2}(\beta-\beta^2-2\beta z+z^2)^{15/2}\big((1+\eta)z^2+2\beta\eta z+\eta\beta^2-\beta\eta\big)^{15/2}},\\[2ex]
\end{aligned}
\end{array}
\end{equation*}
where
\begin{equation*}
\begin{array}{ll}
\begin{aligned}
L_{11}&=(\beta^2-\beta^3+3\beta z^2+2z^3)\beta\eta^2-(\beta^4-2\beta^5+\beta^6-6\beta^4z+6\beta^5z-9\beta^3z^2+15\beta^4z^2-\beta z^3\\[2ex]
&\quad-8\beta^2z^3+20\beta^3z^3-3\beta z^4+15\beta^2z^4+6\beta z^5+z^6)\eta-z^3(\beta^2+\beta^3+3\beta^2z+3\beta z^2+z^3),
\end{aligned}
\end{array}
\end{equation*}
and the expressions of $L_{ii}$ ($i=2,3,4$) are omitted here for brevity.  $\mathrm{Res}(f,g,x)$ denotes the Sylvester resultant of $f$ and $g$ with respect to $x$, the algebraic variety $V(\eta_1,\eta_2,\cdots,\eta_n)$ denotes the set of common zeros of $\eta_i$ $(i=1,2,\cdots,n)$, $\mathrm{Lcoeff}(f,x)$ denotes the leading coefficient of the polynomial $f$ with respect to $x$, and $\mathrm{Prem}(f,g,x)$ denotes the pseudo-remainder of $f$ with respect to $g$ in $x$. Note that if $(\beta,\eta,z)\in\bar{\Omega}^*$, then the denominators of $L_i$ are greater than zero, and all factors except $L_{ii}$, in the numerators of $L_i$ are not zero. Thus we have
\begin{equation*}
\begin{array}{ll}
\begin{aligned}
V(L_1,L_2,L_3,L_4)\cap\Omega^*=V(L_{11},L_{22},L_{33},L_{44})\cap\bar{\Omega}^*.
\end{aligned}
\end{array}
\end{equation*}
Through calculation, we have
\begin{equation}\label{4.5}
\begin{array}{ll}
\begin{aligned}
r_{12}:&=\mathrm{Res}(L_{11},L_{22},\eta)=16\beta^5z^{12}(z+\beta)^{15}(1-z-\beta)^8\kappa_1\kappa_2^3g_1,\\[2ex]
r_{13}:&=\mathrm{Res}(L_{11},L_{33},\eta)=16\beta^6z^{18}(z+\beta)^{25}(1-z-\beta)^{13}\kappa_1\kappa_2^3g_2,\\[2ex]
r_{14}:&=\mathrm{Res}(L_{11},L_{44},\eta)=1024\beta^7z^{24}(z+\beta)^{35}(1-z-\beta)^{18}\kappa_1\kappa_2^3g_3,\\[2ex]
r_{23}:&=\mathrm{Res}(g_1,g_2,z)=C_1\beta^{442}(1-\beta)^{121}(1+2\beta)^{11}(1+27\beta)^{25}\kappa_3S_1S_2,\\[2ex]
r_{24}:&=\mathrm{Res}(g_1,g_3,z)=C_2\beta^{719}(1-\beta)^{181}(1+2\beta)^{17}(1+27\beta)^{40}\kappa_3S_3S_4,\\[2ex]
\mathrm{Res}(R_3,R_4,m)&=-4209701645981302506135597237760,
\end{aligned}
\end{array}
\end{equation}
where
\begin{equation*}
\begin{array}{ll}
\begin{aligned}
\kappa_1&=(z^2+2\beta z+\beta^2-\beta)^2(\beta^3-\beta^2-3\beta z^2-2z^3),\\[2ex]
\kappa_2&=z^3+3\beta z^2+3\beta^2z+\beta z+\beta^3-\beta^2,\\[2ex]
\kappa_3&=(1+\beta)^2(63+\beta)(1+\beta^2)(27\beta^2+18\beta+4),\\[2ex]
C_1&=202552894210739726470935999925510560062373888000,\\[2ex]
C_2&=14815297405128391421874175994551629535990775808000,
\end{aligned}
\end{array}
\end{equation*}
and $g_i$ ($i=1,2,3$), $S_j$ ($j=1,2,3,4$) are omitted for brevity.

In the following, we use four steps to show that $V(L_{11},L_{22},L_{33},L_{44})\cap\bar{\Omega}^*=\emptyset$.

$\mathbf{Step\ I:}$ Simplify the algebraic variety $V(L_{11},L_{22},L_{33})\cap\Omega^*$ and $V(L_{11},L_{22},L_{33},L_{44})\cap\bar{\Omega}^*$.

Through analysis, we have the following decompositions
\begin{equation*}
\begin{array}{ll}
\begin{aligned}
V(L_{11},L_{22},L_{33})=V(L_{11},L_{22},L_{33},\mathrm{Lcoeff}(L_{11},\beta))\cup\bigg(\frac{L_{11},L_{22},L_{33},r_{12},r_{13}}{\mathrm{Lcoeff}(L_{11},\beta)}\bigg),
\end{aligned}
\end{array}
\end{equation*}
and
\begin{equation*}
\begin{array}{ll}
\begin{aligned}
V(L_{11},L_{22},L_{33},L_{44})=V(L_{11},L_{22},L_{33},L_{44},\mathrm{Lcoeff}(L_{11},\beta))\cup\bigg(\frac{L_{11},L_{22},L_{33},L_{44},r_{12},r_{13},r_{14}}{\mathrm{Lcoeff}(L_{11},\beta)}\bigg),
\end{aligned}
\end{array}
\end{equation*}
where $V\big(\frac{\xi_1,\cdots,\xi_n}{\eta_1,\cdots,\eta_m}\big)=V(\xi_1,\cdots,\xi_n)\setminus(\mathop{\cup}\limits_{i=1}^{m} V(\eta_i))$, denotes the set of common zeros of $\xi_i$ ($i=1,2,\cdots,n$) when $\eta_i$ ($i=1,2,\cdots,m$) do not vanish.

For $(\beta,\eta,z)\in\bar{\Omega}^*$, it follows that
\begin{equation}\label{5.5}
\begin{array}{ll}
\begin{aligned}
\mathrm{Lcoeff}(L_{11},h)=2\beta z^3+3\beta^2z^2-\beta^4+\beta^3>0,\quad \mathrm{Lcoeff}(g_1,z)=4>0,
\end{aligned}
\end{array}
\end{equation}
and the first five factors of the right-hand side of $r_{12}$, $r_{13}$ and $r_{14}$ are all negative since $(\beta,z)\in\bar{\Omega}^*$. From \eqref{4.5}, we have
\begin{equation}\label{5.6}
\begin{array}{ll}
\begin{aligned}
V(L_{11},L_{22},L_{33})\cap\bar{\Omega}^*=V(L_{11},L_{22},L_{33},r_{12},r_{13})\cap\bar{\Omega}^*=\{V_1,V_2,V_3\}\cap\bar{\Omega}^*,
\end{aligned}
\end{array}
\end{equation}
where
\begin{equation}\label{5.7}
\begin{array}{ll}
\begin{aligned}
V_1=V(L_{11},L_{22},L_{33},\kappa_1),\quad V_2=V(L_{11},L_{22},L_{33},\kappa_2),\quad V_3=V(L_{11},L_{22},L_{33},g_1,g_2).
\end{aligned}
\end{array}
\end{equation}
Similarly, by \eqref{4.5} and \eqref{5.5}, we can decompose $V_3$ as
\begin{equation}\label{5.8}
\begin{array}{ll}
\begin{aligned}
V_3=V(L_{11},L_{22},L_{33},g_1,g_2,r_{23})=V_{33}\cup V_{34},
\end{aligned}
\end{array}
\end{equation}
where
\begin{equation}\label{5.9}
\begin{array}{ll}
\begin{aligned}
V_{33}&=V(L_{11},L_{22},L_{33},g_1,g_2,S_{11}),\quad V_{34}=V(L_{11},L_{22},L_{33},g_1,g_2,S_{22}).
\end{aligned}
\end{array}
\end{equation}
Similarly, we can decompose $V(L_{11},L_{22},L_{33},L_{44})$ as
\begin{equation}\label{5.10}
\begin{array}{ll}
\begin{aligned}
V(L_{11},L_{22},L_{33},L_{44})\cap\bar{\Omega}^*&=V(L_{11},L_{22},L_{33},L_{44},r_{12},r_{13},r_{14})\cap\bar{\Omega}^*\\[2ex]
&=\{\bar{V}_1,\bar{V}_2,\bar{V}_3\}\cap\bar{\Omega}^*,
\end{aligned}
\end{array}
\end{equation}
where
\begin{equation*}
\begin{array}{ll}
\begin{aligned}
\bar{V}_{1}&=V(L_{11},L_{22},L_{33},L_{44},\kappa_1),\\[2ex]
\bar{V}_{2}&=V(L_{11},L_{22},L_{33},L_{44},\kappa_2),\\[2ex]
\bar{V}_{3}&=V(L_{11},L_{22},L_{33},L_{44},g_1,g_2,g_3).
\end{aligned}
\end{array}
\end{equation*}
Note that the first five factors of $r_{23}$ and $r_{24}$ in \eqref{4.5} are all positive. By \eqref{4.5} and \eqref{5.5}, we can decompose $\bar{V}_3$ as
\begin{equation}\label{5.12}
\begin{array}{ll}
\begin{aligned}
\bar{V}_{33}&=V(L_{11},L_{22},L_{33},L_{44},g_1,g_2,g_3,S_1S_2,S_3S_4).
\end{aligned}
\end{array}
\end{equation}

$\mathbf{Step\ II:}$ Prove that $V_1\cap\bar{\Omega}^*=\emptyset$ and $V_2\cap\bar{\Omega}^*=\emptyset$.
\\
Define
\begin{equation}\label{4.8}
\begin{array}{ll}
\begin{aligned}
z^*=\frac{\kappa^2-3\beta\kappa-3\beta}{3\kappa},\quad \kappa=\sqrt[3]{27\beta^2+3\beta\sqrt{3(\beta+\beta^2)}},\quad 0<\beta<1.
\end{aligned}
\end{array}
\end{equation}
It is easy to check that $\kappa_1<0$ for $(\beta,z)\in\bar{\Omega}^*$, which implies
\begin{equation*}
\begin{array}{ll}
\begin{aligned}
V_1\cap\bar{\Omega}^*=\emptyset.
\end{aligned}
\end{array}
\end{equation*}
While there exists a unique root $z=z^*$ such that $\kappa_2=0$ for $(\beta,z)\in\bar{\Omega}^*$. Next, we show that $V_2\cap\bar{\Omega}^*=\emptyset$.

\begin{lemma}\label{l11}
If $(\delta,m,z)\in\bar{\Omega}^*$ and \eqref{5.1} hold, then  $V_2\cap\bar{\Omega}^*=V(L_{11},L_{22},L_{33},R_1)=\emptyset$. Moreover, for $\kappa_2=0$ (i.e. $z=z^*$),
the following statements hold:

(i) If $\eta\neq\eta_2$, then $E_2^{**}$ is a  weak focus of order 1;

(ii) If $\eta=\eta_2$, then $E_2^{**}$ is a stable weak focus of order 2.
\end{lemma}
\begin{proof}
By substituting $z=z^*$ (i.e. $\kappa_2=0$) into $L_{ii}$ ($i=1, 2$), we have
\begin{equation}\label{4.6}
\begin{array}{ll}
\begin{aligned}
L_{11}=\frac{1}{729\kappa^6}\bar{L}_{11},\quad L_{22}=\frac{1}{387420489\kappa^{18}}\bar{L}_{22},
\end{aligned}
\end{array}
\end{equation}
where
\begin{equation*}
\begin{array}{ll}
\begin{aligned}
\bar{L}_{11}:&=-27\beta\kappa^3\eta^2(54\beta^3+81\beta^3\kappa-54\beta^2\kappa^2-81\beta^2\kappa^3+18\beta\kappa^4+9\beta\kappa^5-2\kappa^6)-\eta(729\beta^6\\[2ex]
&\quad-3645\beta^5\kappa^2+729\beta^4\kappa^3-2916\beta^5\kappa^3+6318\beta^4\kappa^4-2187\beta^5\kappa^4-729\beta^3\kappa^5+5103\beta^4\kappa^5\\[2ex]
&\quad-3456\beta^3\kappa^6+2916\beta^4\kappa^6+243\beta^2\kappa^7-1701\beta^3\kappa^7+702\beta^2\kappa^8-243\beta^3\kappa^8-27\beta\kappa^9\\[2ex]
&\quad+108\beta^2\kappa^9-45\beta\kappa^{10}+\kappa^{12})-(3\beta+3\beta\kappa-\kappa^2)^3(27\beta^3-27\beta^2\kappa^2-27\beta^2\kappa^3+9\beta\kappa^4-\kappa^6),
\end{aligned}
\end{array}
\end{equation*}
while we omit the expression of $\bar{L}_{22}$ and $\bar{L}_{33}$ for brevity.
Substituting $z=z^*$ into $0<\eta<\frac{z^2}{\beta-(z+\beta)^2}$, we have
\begin{equation*}
\begin{array}{ll}
\begin{aligned}
0<\eta<\eta^*:=\frac{(\kappa^2-3\beta\kappa-3\beta)^2}{-\kappa^4+15\beta\kappa^2-9\beta^2}.
\end{aligned}
\end{array}
\end{equation*}
Next, we treat $\bar{L}_{11}$ as a quadratic functions of $\eta$, whose discriminant is
\begin{equation*}
\begin{array}{ll}
\begin{aligned}
\Delta_\eta&=59049\beta^{10}-551124\beta^9\kappa^2-354294\beta^8\kappa^3-472392\beta^9\kappa^3-354294\beta^9\kappa^4+472392\beta^7\kappa^5\\[2ex]
&\quad-669222\beta^8\kappa^5+59049\beta^6\kappa^6+288684\beta^7\kappa^6+354294\beta^8\kappa^6+118098\beta^6\kappa^7+1850202\beta^7\kappa^7\\[2ex]
&\quad+708588\beta^8\kappa^7-78732\beta^5\kappa^8+809190\beta^6\kappa^8+944784\beta^7\kappa^8+531441\beta^8\kappa^8-288684\beta^5\kappa^9\\[2ex]
&\quad-293058\beta^6\kappa^9+118098\beta^7\kappa^9+39366\beta^4\kappa^{10}-590490\beta^5\kappa^{10}-150903\beta^6\kappa^{10}-354294\beta^7\kappa^{10}\\[2ex]
&\quad+96228\beta^4\kappa^{11}+97686\beta^5\kappa^{11}-39366\beta^6\kappa^{11}-8748\beta^3\kappa^{12}+89910\beta^4\kappa^{12}+104976\beta^5\kappa^{12}\\[2ex]
&\quad+59049\beta^6\kappa^{12}-4374\beta^3\kappa^{13}-68526\beta^4\kappa^{13}-26244\beta^5\kappa^{13}+729\beta^2\kappa^{14}+3564\beta^3\kappa^{14}+\kappa^{20}\\[2ex]
&\quad+4374\beta^4\kappa^{14}-1944\beta^2\kappa^{15}+2754\beta^3\kappa^{15}-486\beta^3\kappa^{16}+162\beta\kappa^{17}+216\beta^2\kappa^{17}-84\beta\kappa^{18}.
\end{aligned}
\end{array}
\end{equation*}
We find that $\Delta_\eta>0$ for $0<\beta<1$. Thus, $\bar{L}_{11}$ has two real roots, $\eta_1<0$ and $0<\eta_2<\eta^*$ when $0.1667<\beta<1$. According to the monotonicity of $\bar{L}_{11}$ with respect to $\eta$, we have that $\bar{L}_{11}<0$ $(>0)$  when $0<\eta<\eta_2$ ($\eta_2<\eta<\eta^*$), i.e., $E_2^{**}$ is a stable (an unstable) weak focus of order 1 when $z=z^*$ and $0<\eta<\eta_2$ ($\eta_2<\eta<\eta^*$).

When $\eta=\eta_2$, $\bar{L}_{11}=0$. By direct calculation, we have $\bar{L}_{22}<0$, then $E^{**}$ is a weak focus of order 2.
\end{proof}

$\mathbf{Step\ III:}$  We show that $V(L_{11},L_{22},L_{33},L_{44})\cap\bar{\Omega}^*=\emptyset$.
\begin{lemma}\label{l9}
If $(\beta,\eta,z)\in\bar{\Omega}^*$ and \eqref{5.1} hold, then  $V(L_{11},L_{22},L_{33},L_{44})\cap\bar{\Omega}^*=\emptyset$. Moreover,
if $\kappa_2\neq0$, then $E_2^{**}$ is a weak focus of order at most 4, which is exhibited in Fig. \ref{Hopf0}.
\end{lemma}
\begin{figure}
\begin{center}
\begin{overpic}[scale=0.70]{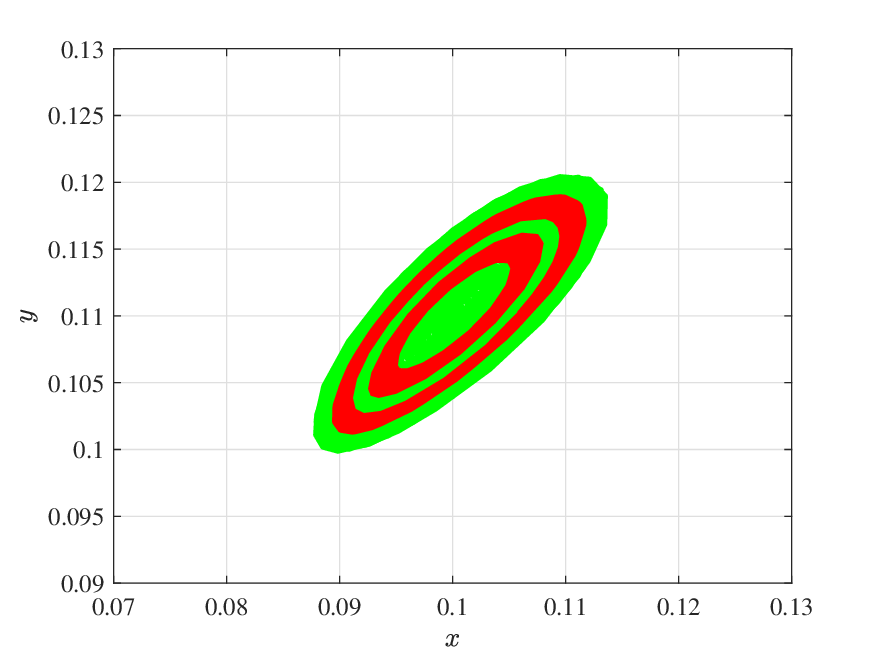}
\end{overpic}
\vspace{-6mm}
\end{center}
\caption{ For $\alpha=\frac{1}{2}$, $\beta=\frac{1}{2}$, $\gamma=\frac{20}{33}$, $\delta=\frac{7}{180}$ and $\eta=\frac{1}{100}$, system \eqref{2.1} exists four limit cycles created by Hopf bifurcation around $E_2^{**}$ when $\alpha=\beta$.\label{Hopf0}}
\end{figure}
\begin{proof}
From \eqref{5.7}-\eqref{5.10}, Lemmas \ref{l11}, we have that $\bar{V}_1\cap\bar{\Omega}^*\subseteq V_1\cap\bar{\Omega}^*$ and $\bar{V}_2\cap\bar{\Omega}^*\subseteq V_2\cap\bar{\Omega}^*$. Thus $\bar{V}_1\cap\bar{\Omega}^*=\emptyset$ and $\bar{V}_2\cap\bar{\Omega}^*=\emptyset$. Besides, we have
\begin{equation*}
\begin{array}{ll}
\begin{aligned}
\mathrm{Res}(S_1S_2,S_3S_4,\beta)\neq0,
\end{aligned}
\end{array}
\end{equation*}
and $V(S_1S_2,S_3S_4)=\emptyset$, which implies $\bar{V}_{33}\cap\bar{\Omega}^*\subseteq V(S_1S_2,S_3S_4)=\emptyset$. From \eqref{5.10}, it yields that $V(L_{11},L_{22},L_{33},L_{44})\cap\bar{\Omega}^*=\emptyset$. Hence, $E_2^*$ is a weak focus of order at most 4.
\end{proof}

From Lemma \ref{l9}, we have the following results.
\begin{lemma}\label{l10}
If $(\beta,\eta,z)\in\bar{\Omega}^*$ and \eqref{5.1} hold, then $E_2^{**}$ is a weak focus of order at most 4, and

(i) its order is 1 if $(\beta,\eta,z)\in\Gamma_1:=\{(\beta,\eta,z)\in\bar{\Omega}^*:L_{11}(\beta,\eta,z)\neq0\}$;

(ii) its order is 2 if $(\beta,\eta,z)\in\Gamma_2:=\{(\beta,\eta,z)\in\bar{\Omega}^*:L_{11}(\beta,\eta,z)=0, L_{22}(\beta,\eta,z)\neq0\}$;

(iii)its order is 3 if $(\beta,\eta,z)\in\Gamma_3:=\{(\beta,\eta,z)\in\bar{\Omega}^*:L_{11}(\beta,\eta,z)=0, L_{22}(\beta,\eta,z)=0, L_{33}(\beta,\eta,z)\neq0\}$;

(iv) its order is if $(\beta,\eta,z)\in\Gamma_4:=\{(\beta,\eta,z)\in\bar{\Omega}^*\setminus(\Gamma_1\cup\Gamma_2\cup\Gamma_3)\}$.
\end{lemma}

Next, we show that there exist one set of parameter values in $\bar{\Omega}^*$ such that $L_{11}=L_{22}=0$ and $L_{33}\neq0$, i.e. $E_2^{**}$ is a weak focus of order 3.

\begin{lemma}
If $(\beta,\eta,z)\in\bar{\Omega}^*$ and \eqref{5.1} hold, then $\Gamma_i\neq\emptyset$ ($i=1, 2, 3$).
\end{lemma}
\begin{proof}
Obviously, by Lemma \ref{l11}(i) and (ii), we obtain $\Gamma_1\neq\emptyset$ and $\Gamma_2\neq\emptyset$. Next, we consider Lemma \ref{l9} to show $\Gamma_3\neq\emptyset$.

If $\kappa_2\neq0$, for some suitable parameter values, $E_2^{**}$ can be a weak focus of order 3. Fixing $\beta=\frac{1}{2}$,  $L_{ii}$ ($i=1, 2, 3,$) can be simplified as
\begin{equation}
\begin{array}{ll}\label{4.16}
\begin{aligned}
L_{11}=\frac{l_{11}}{64},\quad L_{22}=\frac{l_{22}}{65536},\quad L_{33}=\frac{l_{33}}{67108864},
\end{aligned}
\end{array}
\end{equation}
where
\begin{equation*}
\begin{array}{ll}
\begin{aligned}
l_{11}&=-\eta+4\eta^2+12\eta hz+12\eta z^2+48\eta^2z^2-24z^3+64\eta^2z^3-48z^4\\[2ex]
&\quad-144\eta z^4-96z^5-192\eta z^5-64z^6-64\eta z^6
\end{aligned}
\end{array}
\end{equation*}
and $l_{22}$, $l_{33}$ are omitted for brevity. Through calculate, we have
\begin{equation}
\begin{array}{ll}\label{4.17}
\begin{aligned}
\breve{r}_{12}&=\mathrm{Res}(l_{11},l_{22},\eta)=131072z^{12}k_{11}l_{12},\\[2ex]
\breve{r}_{13}&=\mathrm{Res}(l_{11},l_{33},\eta)=8388608z^{12}k_{11}l_{13},\\[2ex]
\breve{r}_{23}&=\mathrm{Res}(l_{12},l_{13},\eta)<0,
\end{aligned}
\end{array}
\end{equation}
where
\begin{equation*}
\begin{array}{ll}
\begin{aligned}
k_{11}&=(1-2z)^8(1+2z)^{15}(4z^2+4z-1)^2(8z^3+12z^2+10z-1)^3(16z^3+12z^2+1),\\[2ex]
k_{22}&=(1-2z)^{13}(1+2z)^{25}(4z^2+4z-1)^2(8z^3+12z^2+10z-1)^3(16z^3+12z^2+1),\\[2ex]
l_{12}&=12+137z-342z^2-6012z^3-18872z^4-50608z^5-189024 z^6-814272z^7-1534848z^8\\[2ex]
&\quad-2886912z^9-9378304z^{10}-19174400z^{11}-24307712z^{12}-29691904z^{13}-38412288z^{14}\\[2ex]
&\quad-35799040z^{15}-19300352z^{16}-5242880z^{17}-524288z^{18},
\end{aligned}
\end{array}
\end{equation*}
and $l_{13}$ is omitted for brevity. Note that $\mathrm{lcoeff}(l_{11},\eta)=4(1+12z^2+16z^3)>0$, $\mathrm{lcoeff}(l_{12},z)=-524288\neq0$ and the five factors in $k_{11}$ and $k_{22}$ are all nonzero. Then, we get
\begin{equation}\label{4.18}
\begin{array}{ll}
\begin{aligned}
V(L_{11},L_{22},L_{33})\cap\bar{\Omega}^*=V(L_{11},L_{22},L_{33},\breve{r}_{12},\breve{r}_{13})\cap\bar{\Omega}^*=V(L_{11},L_{22},L_{33},\breve{r}_{12},\breve{r}_{13},,\breve{r}_{23})\cap\bar{\Omega}^*.
\end{aligned}
\end{array}
\end{equation}
From $\breve{r}_{23}\neq0$, it yields that $V(L_{11},L_{22},L_{33})\cap\bar{\Omega}^*=\emptyset$.

In what follows, we prove $V(L_{11},L_{22})\cap\bar{\Omega}^*=\emptyset$. It is easy to show that $l_{12}$ has exactly one real zero $\hat{z}$ in $(0,\sqrt{\frac{1}{2}}-\frac{1}{2})$, which is covered by
\begin{equation*}
\begin{array}{ll}
\begin{aligned}
\hat{I}&=\bigg[\frac{2289359021}{17179869184},\frac{4578718043}{34359738368}\bigg]\in(0,\sqrt{\frac{1}{2}}-\frac{1}{2}).
\end{aligned}
\end{array}
\end{equation*}

Similarly, from  Sturm's Theorem, it follows that
\begin{equation*}
\begin{array}{ll}
\begin{aligned}
V(L_{11},L_{22})\cap\{\beta=\frac{1}{2}\}\cap\bar{\Omega}^*=V(l_{11},l_{22},l_{12})\cap\{\beta=\frac{1}{2}\}\cap\bar{\Omega}^*\neq\emptyset,
\end{aligned}
\end{array}
\end{equation*}
which implies that $E_2^{**}$ is a weak focus of order exactly 3.

\end{proof}

\subsection{Case 2: $\alpha\neq\beta$}

Now, we compute the first five focal values of system \eqref{3.1}. Make the following linear transformations successively
\begin{equation*}
\begin{array}{ll}
\begin{aligned}
x&=X+z,\quad y=Y+z+\eta;\\[2ex]
X&=u+\frac{\sqrt{z\gamma-\delta}}{\sqrt{\delta}}v,\quad Y=u,\quad t=\frac{\tau}{\sqrt{d}},
\end{aligned}
\end{array}
\end{equation*}
where $d=Det(J(E_2^{*}))=\frac{z^2\delta(z\gamma-\delta)(z+\eta)^2(1-z-z\gamma-\gamma\eta)^2}{(z+\delta)^2}$,  system \eqref{3.1} can be transformed into (still $\tau$ by $t$)
\begin{equation}
\begin{array}{ll}\label{3.3}
\left\{
\begin{aligned}
\dot{u}&=v+p_{11}uv+p_{02}v^2+p_{21}u^2v+p_{12}uv^2,\\[2ex]
\dot{v}&=-u+q_{20}u^2+q_{11}uv+q_{02}v^2+q_{30}u^3+q_{21}u^2v+q_{12}uv^2+q_{03}v^3\\[2ex]
&\quad+q_{40}u^4+q_{31}u^3v+q_{22}u^2v^2+q_{13}uv^3+q_{04}v^4,
\end{aligned}
\right.
\end{array}
\end{equation}
where $p_{ij}$ and $q_{ij}$ are given Appendix J.

According to Formal Series Method (see \cite{ZT}), we can obtain the first five focal values as follows
\begin{equation}
\begin{array}{ll}\label{4.15}
\begin{aligned}
L_1&=\frac{\gamma L_{11}}{4z\sqrt{\delta^3(z\gamma-\delta)^3}(z+\eta)(1-z-z\gamma-\gamma\eta)^2},\\[2ex]
L_2&=\frac{\gamma L_{22}}{48z^2\sqrt{\delta^7(z\gamma-\delta)^7}(z+\eta)^3(1-z-z\gamma-\gamma\eta)^4},\\[2ex]
L_3&=\frac{\gamma L_{33}}{9216z^5\sqrt{\delta^{11}(z\gamma-\delta)^{11}}(z+\eta)^5(1-z-z\gamma-\gamma\eta)^6},\\[2ex]
L_4&=\frac{\gamma L_{44}}{2211840z^7\sqrt{\delta^{15}(z\gamma-\delta)^{15}}(z+\eta)^7(1-z-z\gamma-\gamma\eta)^8},\\[2ex]
L_5&=\frac{\gamma L_{55}}{25480396800z^9\sqrt{\delta^{19}(z\gamma-\delta)^{19}}(z+\eta)^9(1-z-z\gamma-\gamma\eta)^{10}},
\end{aligned}
\end{array}
\end{equation}
where
\begin{equation*}
\begin{array}{ll}
\begin{aligned}
L_{11}&=(4\gamma^2+7\gamma+2)z^5+(\gamma^3\delta+10\gamma^2\delta+6\gamma^2\eta+5\gamma\eta+15\gamma\delta-4\gamma+2\delta+2\eta)z^4\\[2ex]
&\quad+(2\gamma^3\delta\eta+14\gamma^2\delta\eta+2\gamma^2\eta^2+3\gamma^2\delta^2+5\gamma\delta^2-2\gamma^2\delta-9\gamma\delta\eta+4\delta\eta-2\gamma\eta\\[2ex]
&\quad-4\delta^2-10\gamma\delta+3\delta)z^3+(\gamma^3\delta\eta^2+4\gamma^2\delta\eta^2-\gamma\delta\eta^2+4\gamma^2\delta^2\eta+\gamma\delta^2\eta+\delta^2\eta\\[2ex]
&\quad-2\gamma^2\delta\eta-4\gamma\delta\eta+\delta\eta-4\delta^3-2\gamma\delta^2+8\delta^2+\gamma\delta)z^2+(\gamma^2\delta^2\eta^2-2\gamma\delta^2\eta^2\\[2ex]
&\quad-\gamma\delta^3\eta-2\delta^3\eta+2\delta^2\eta-\delta^4+3\delta^3-\delta^2)z-\delta^3\eta(\delta+\gamma\eta-1).
\end{aligned}
\end{array}
\end{equation*}
and the expressions of $L_{ii}$ ($i=2,3,4,5$) are omitted here for brevity.

Due to the extremely complex calculations, we only consider the special case  $\delta=\frac{1}{10}$. From \eqref{4.15}, we have
\begin{equation*}
\begin{array}{ll}
\begin{aligned}
L_1&=\frac{\gamma L_{11}}{40z(z+\eta)(\gamma z+z+\eta\gamma-1)^2(10\gamma z-1)^{3/2}},\\[2ex]
L_2&=\frac{\gamma L_{22}}{4800z^2(z+\eta)^3(\gamma z+z+\eta\gamma-1)^4(10\gamma z-1)^{7/2}},\\[2ex]
L_3&=\frac{\gamma L_{33}}{23040000z^5(z+\eta)^5(\gamma z+z+\eta\gamma-1)^6(10\gamma z-1)^{11/2}},\\[2ex]
L_4&=\frac{\gamma L_{44}}{552960000000z^7(z+\eta)^7(\gamma z+z+\eta\gamma-1)^8(10\gamma z-1)^{15/2}},\\[2ex]
L_5&=\frac{\gamma L_{55}}{318504960000000000z^9(z+\eta)^9(\gamma z+z+\eta\gamma-1)^{10}(10\gamma z-1)^{19/2}},
\end{aligned}
\end{array}
\end{equation*}
where
\begin{equation*}
\begin{array}{ll}
\begin{aligned}
L_{11}&=10\eta^2\gamma(10\gamma z-1-20z-100z^2+400\gamma z^2+100\gamma^2z^2+2000\gamma z^3)+\eta(9+180z-10\gamma z\\[2ex]
&\quad+1100z^2-3900\gamma z^2-1600\gamma^2z^2+4000z^3-11000\gamma z^3+14000\gamma^2z^3+2000\gamma^3z^3\\[2ex]
&\quad+20000z^4+50000\gamma z^4+60000\gamma^2z^4)+z(760z-71+800\gamma z+2600z^2-9500\gamma z^2\\[2ex]
&\quad-1700\gamma^2z^2+2000z^3-25000\gamma z^3+10000\gamma^2z^3+1000\gamma^3z^3+20000z^4+70000\gamma z^4\\[2ex]
&\quad+40000\gamma^2z^4),
\end{aligned}
\end{array}
\end{equation*}
and the expressions of $L_{ii}$ ($i=2,3,4,5$) are omitted here for brevity.
Note that if $(\gamma,\eta,z)\in\Omega^*$, then the denominators of $L_i$ are greater than zero, and all factors except $L_{ii}$, in the numerators of $L_i$ are not zero. Thus we have
\begin{equation*}
\begin{array}{ll}
\begin{aligned}
V(L_1,L_2,L_3,L_4)\cap\Omega^*=V(L_{11},L_{22},L_{33},L_{44})\cap\Omega^*.
\end{aligned}
\end{array}
\end{equation*}

Through calculation, we have
\begin{equation*}
\begin{array}{ll}
\begin{aligned}
r_{12}:&=\mathrm{Res}(L_{11},L_{22},\eta)=C_1\gamma^4z^3(1-z)^3(1+10z)^{10}(10\gamma z-1)^6r_1r_2g_1,\\[2ex]
r_{13}:&=\mathrm{Res}(L_{11},L_{33},\eta)=C_2\gamma^7z^8(1-z)^5(1+10z)^{14}(10\gamma z-1)^8r_1r_2g_2,\\[2ex]
r_{14}:&=\mathrm{Res}(L_{11},L_{44},\eta)=C_3\gamma^8z^{10}(1-z)^7(1+10z)^{18}(10\gamma z-1)^{10}r_1r_2g_3,\\[2ex]
r_{15}:&=\mathrm{Res}(L_{11},L_{55},\eta)=C_4\gamma^{10}z^{12}(1-z)^9(1+10z)^{22}(10\gamma z-1)^{12}r_1r_2g_4,\\[2ex]
r_{23}:&=\mathrm{Res}(g_1,g_2,z)=C_5\gamma^{36}(1+\gamma)^{44}(9+5\gamma)r_3h_1^3h_2h_3h_4S_1S_2,\\[2ex]
r_{24}:&=\mathrm{Res}(g_1,g_3,z)=C_6\gamma^{66}(1+\gamma)^{66}(9+5\gamma)^2r_3h_1^5h_2h_3h_4S_3S_4,\\[2ex]
r_{25}:&=\mathrm{Res}(g_1,g_4,z)=C_7\gamma^{77}(1+\gamma)^{88}(9+5\gamma)^3r_3h_1^7h_2h_3h_4S_5S_6,\\[2ex]
\tilde{r}_{12}:&=\mathrm{Res}(h_1,h_2,\gamma)=-1600,\quad\tilde{r}_{13}:=\mathrm{Res}(h_1,h_3,\gamma)=-11552,\\[2ex]
\tilde{r}_{14}:&=\mathrm{Res}(h_1,h_4,\gamma)=-12848,\quad\tilde{r}_{23}:=\mathrm{Res}(h_2,h_3,\gamma)=-27392,\\[2ex]
\tilde{r}_{24}:&=\mathrm{Res}(h_2,h_4,\gamma)=-49968,\quad\tilde{r}_{34}:=\mathrm{Res}(h_3,h_4,\gamma)=42496,
\end{aligned}
\end{array}
\end{equation*}
where
\begin{equation*}
\begin{array}{ll}
\begin{aligned}
r_1&=10\gamma z-1-20z-100z^2+400\gamma z^2+100\gamma^2z^2+2000\gamma z^3,\\[2ex]
r_2&=100z^2+25z-4,\quad C_4=6553600000000000000000000000,\quad C_5>0,\quad C_6>0,\quad C_7>0,\\[2ex]
r_3&=(2+r)^8(21+10r)(3+3r+r^2)^2(1+4r+2r^2)^2(-3+17r+9r^2)(-1+21r+11r^2),\\[2ex]
C_1&=32000000,\quad C_2=32000000000000,\quad C_3=20480000000000000000,\\[2ex]
h_1&=8\gamma^2-5\gamma-2,\quad h_2=8\gamma^2+5\gamma-2,\quad h_3=2\gamma^3-9\gamma^2-20\gamma-4,\quad h_4=6\gamma^3+\gamma^2-20\gamma-4,
\end{aligned}
\end{array}
\end{equation*}
and $g_i$ ($i=1,2,3,4$), $S_j$ ($j=1,2,3,4,5,6$) are omitted for brevity.

$\mathbf{Step\ I:}$ Show that $V_1\cap\Omega^*=\emptyset$ and $V_2\cap\Omega^*=\emptyset$.

Define
\begin{equation}\label{4.11}
\begin{array}{ll}
\begin{aligned}
z^*=\frac{\sqrt{89}-5}{40}.
\end{aligned}
\end{array}
\end{equation}
It is easy to check that $r_1>0$ for $(\gamma,\eta,z)\in\Omega^*$, which implies
\begin{equation*}
\begin{array}{ll}
\begin{aligned}
V_1\cap\Omega^*=\emptyset.
\end{aligned}
\end{array}
\end{equation*}
While there exists a unique root $z=z^*$ such that $r_2=0$ for $(\gamma,\eta,z)\in\Omega^*$. Next, we show that $V_2\cap\Omega^*=\emptyset$.

\begin{lemma}\label{l12}
If $(\gamma,\eta,z)\in\Omega^*$ and \eqref{5.1} hold, then  $V_2\cap\Omega^*=V(L_{11},L_{22},L_{33},r_2)=\emptyset$. Moreover, for $r_2=0$ (i.e. $z=z^*$),
the following statements hold.

(i) If $\eta\neq\eta_2$, then $E_2^*$ is a  weak focus of order 1;

(ii) If $\eta=\eta_2$, then $E_2^*$ is a  weak focus of order 2 when $0.25<\gamma<0.476347$ ($0.990687<\gamma<6.099848$).
\end{lemma}
\begin{proof}
By substituting $z=z^*$ (i.e. $\kappa_2=0$) into $L_{ii}$ ($i=1,2,3$), we have
\begin{equation}\label{4.12}
\begin{array}{ll}
\begin{aligned}
L_{11}&=\frac{1}{640}\bar{L}_{11},\quad L_{22}=\frac{1}{8192000}\bar{L}_{22},
\end{aligned}
\end{array}
\end{equation}
where
\begin{equation*}
\begin{array}{ll}
\begin{aligned}
\bar{L}_{11}:&=(-36000\gamma+800\sqrt{89}\gamma-117600\gamma^2+18400\sqrt{89}\gamma^2+45600\gamma^3-4000\sqrt{89}\gamma^3)\eta^2\\[2ex]
&\quad+(92600-6360\sqrt{89}+257260\gamma-31100\sqrt{89}\gamma+51080\gamma^2-4840\sqrt{89}\gamma^2\\[2ex]
&\quad-29200\gamma^3+3280\sqrt{89}\gamma^3)\eta-25726+3110\sqrt{89}-98345\gamma+10037\sqrt{89}\gamma\\[2ex]
&\quad+1460\gamma^2-164\sqrt{89}\gamma^2+5474\gamma^3-570\sqrt{89}\gamma^3
\end{aligned}
\end{array}
\end{equation*}
while we omit the expression of $\bar{L}_{22}$, $\bar{L}_{33}$ and $\bar{L}_{44}$ for brevity.
Substituting $z=z^*$ into $0<\eta<8z-20z^2$, we have
\begin{equation*}
\begin{array}{ll}
\begin{aligned}
0<\eta<\eta^*:=\frac{13\sqrt{89}-97}{40}.
\end{aligned}
\end{array}
\end{equation*}
Next, we treat $\bar{L}_{11}$ as a quadratic functions of $\eta$, whose discriminant is
\begin{equation*}
\begin{array}{ll}
\begin{aligned}
\Delta_\eta&=800(15218468-1472340\sqrt{89}+97827460\gamma-10627380\sqrt{89}\gamma+147137761\gamma^2\\[2ex]
&\quad-15539085\sqrt{89}\gamma^2-73344940\gamma^3+7829820\sqrt{89}\gamma^3+8068868\gamma^4-848820\sqrt{89}\gamma^4).
\end{aligned}
\end{array}
\end{equation*}
We have that $\Delta_\eta>0$ for $\gamma>\frac{1}{4}$. Thus, $\bar{L}_{11}$ has two real roots
\begin{equation*}
\begin{array}{ll}
\begin{aligned}
\eta_1&=\frac{5-\sqrt{89}}{40}<0,\quad \eta_2=\frac{-3825-157\sqrt{89}-1424\gamma+720\sqrt{89}\gamma+128\gamma^2+160\gamma^3-32\sqrt{89}\gamma^3}{10(-265\gamma-21\sqrt{89}\gamma+232\gamma^2+72\sqrt{89}\gamma^2+128\gamma^3)}.
\end{aligned}
\end{array}
\end{equation*}
If $0.25<\gamma<0.476347$ or $0.990687<\gamma<6.099848$, then $\eta_2>0$. Therefore, according to the monotonicity of $\bar{L}_{11}$ with respect to $\eta$, we have $\bar{L}_{11}<0$ ($>0$)  when $0<\eta<\eta_2$ ($\eta_2<\eta<\eta^*$), i.e. $E_2^*$ is a stable (an unstable) weak focus of order 1 when $z=z^*$ and $0<\eta<\eta_2$ ($\eta_2<\eta<\eta^*$).

When $\eta=\eta_2$, $\bar{L}_{11}=0$. Substitute $z=z^*$ and $\eta=\eta_2$ into $\bar{L}_{22}$, we have
\begin{equation}\label{4.13}
\begin{array}{ll}
\begin{aligned}
\bar{L}_{22}&=-\frac{128l_{22}}{125\gamma^2(128\gamma^2+72\sqrt{89}\gamma+232\gamma-21\sqrt{89}-265)^6},
\end{aligned}
\end{array}
\end{equation}
where $l_{22}$ are given in Appendix K.

Through analysis, it yields that $\bar{L}_{22}<0(>0)$ when $0.25<\gamma<0.476347(0.990687<\gamma<6.099848)$. Thus, $E_2^*$ is a stable (an unsatble) weak focus of order 2 when $0.25<\gamma<0.476347$ ($0.990687<\gamma<6.099848$).
\end{proof}

$\mathbf{Step\ II:}$ Prove that $V(L_{11},L_{22},L_{33},L_{44},g_1,g_2,g_3,h_3)\cap\Omega^*=\emptyset$.

\begin{lemma}\label{l13}
$V(L_{11},L_{22},L_{33},L_{44},g_1,g_2,g_3,h_3)\cap\Omega^*=\emptyset$.
\end{lemma}
\begin{proof}
We compute the following resultants
\begin{equation*}
\begin{array}{ll}
\begin{aligned}
\mathrm{Res}(h_3,g_1,\gamma)&=256k_0h_{31},\quad\mathrm{Res}(h_3,g_2,\gamma)=262144k_0h_{32},\\[2ex]
\mathrm{Res}(h_3,g_3,\gamma)&=4294967296k_0h_{33},\quad\mathrm{Res}(h_{31},h_{32},z)=\tilde{c}_1,\quad\mathrm{Res}(h_{31},h_{33},z)=\tilde{c}_2,
\end{aligned}
\end{array}
\end{equation*}
where
\begin{equation*}
\begin{array}{ll}
\begin{aligned}
k_0=1+80z+1300z^2+4000z^3,
\end{aligned}
\end{array}
\end{equation*}
$\tilde{c}_1$ and $\tilde{c}_2$ are positive constants. For brevity we omit the detailed expressions of $\tilde{c}_1$, $\tilde{c}_2$, $h_{31}$, $h_{32}$ and $h_{33}$.  $\mathrm{Lcoeff}(h_3,\gamma)=2>0$,  $\mathrm{Lcoeff}(h_{31},z)>0$ and $k_0>0$. Therefore, $V(g_1,g_2,g_3,h_3,h_{31},h_{32},h_{33})\cap\Omega^*=\subseteq V(g_1,g_2,g_3,h_3)\cap\Omega^*=\emptyset$.
\end{proof}

$\mathbf{Step\ III:}$ Show that $V(L_{11},L_{22},L_{33},L_{44},g_1,g_2,g_3,h_i)\cap\Omega^*=\emptyset$ ($i=1,2,4$).

\begin{lemma}\label{l14}
$V(L_{11},L_{22},L_{33},L_{44},g_1,g_2,g_3,h_i)\cap\Omega^*=\emptyset$ ($i=1,2,4$).
\end{lemma}
\begin{proof}
Through calculation, we have
\begin{equation*}
\begin{array}{ll}
\begin{aligned}
r_{21}:&=\mathrm{Res}(h_1,g_1,\gamma)=512R_{21}R_{22},\\[2ex]
r_{22}:&=\mathrm{Res}(h_1,g_2,\gamma)=262144R_{21}^3R_{23},\\[2ex]
r_{23}:&=\mathrm{Res}(h_1,g_3,\gamma)=536870912R_{21}^5R_{24}
\end{aligned}
\end{array}
\end{equation*}
where $R_{21}=100z^2+25z-4$, $R_{22}$, $R_{23}$ and $R_{24}$ are omitted for brevity. Since $\mathrm{Lcoeff}(h_1,\gamma)=8>0$, we have
\begin{equation*}
\begin{array}{ll}
\begin{aligned}
V(L_{11},L_{22},L_{33},L_{44},g_1,g_2,g_3,h_1,R_{21})\cup V(L_{11},L_{22},L_{33},L_{44},g_1,g_2,g_3,h_1,R_{22},R_{23},R_{24}).
\end{aligned}
\end{array}
\end{equation*}
By straightforward calculation, we have $\mathrm{Res}(R_{22},R_{23},z)>0$, $\mathrm{Res}(R_{22},R_{24},z)<0$ and
\begin{equation*}
\begin{array}{ll}
\begin{aligned}
V(L_{11},L_{22},L_{33},L_{44},g_1,g_2,g_3,h_1,R_{22},R_{23},R_{24})=\emptyset,
\end{aligned}
\end{array}
\end{equation*}
which implies $V(L_{11},L_{22},L_{33},L_{44},g_1,g_2,g_3,h_1)=V(L_{11},L_{22},L_{33},L_{44},g_1,g_2,g_3,h_1,R_{21})$.

Furthermore, $h_1=0$ has a root $\gamma\in(0.9,1)$ and $R_{21}=0$ has a root $z\in(0.1,0.2)$, which means when $h_1=0$ and $R_{21}=0$, $V(g_1,g_2,g_3)\neq\emptyset$. Next, we show that $V(L_{11},L_{22},L_{33},L_{44},g_1,g_2,g_3,h_1,R_{21})\cap\Omega^*=\emptyset$. By directly computing, we have
\begin{equation*}
\begin{array}{ll}
\begin{aligned}
r_2(1):&=\mathrm{Res}(R_{21},\mathrm{Res}(h_1,L_{11},\gamma),z)=k_1^3k_2,\\[2ex]
r_2(2):&=\mathrm{Res}(R_{21},\mathrm{Res}(h_1,L_{22},\gamma),z)=k_1^5k_3,\\[2ex]
r_2(3):&=\mathrm{Res}(R_{21},\mathrm{Res}(h_1,L_{33},\gamma),z)=k_1^7k_4,\\[2ex]
r_2(4):&=\mathrm{Res}(R_{21},\mathrm{Res}(h_1,L_{44},\gamma),z)=k_1^9k_5,
\end{aligned}
\end{array}
\end{equation*}
where $k_1=100\eta^2-25\eta-4$, here we omit the detailed expressions of $k_i$ (i=2, 3, 4, 5). It is easy to show that $\mathrm{Res}(k_i,k_j,\eta)\neq0$ ($i=1,2,3,4$, $j=2,3,4,5$, $j>i$). Hence,
\begin{equation*}
\begin{array}{ll}
\begin{aligned}
V(L_{11},L_{22},L_{33},L_{44},g_1,g_2,g_3,h_1,R_{21})\cap\Omega^*\subseteq V(L_{11},L_{22},L_{33},L_{44},g_1,g_2,g_4,h_1,R_{21},k_1)\cap\Omega^*.
\end{aligned}
\end{array}
\end{equation*}
If $h_1=0$, $R_{21}=0$ and $k_1=0$, then $Det(J(E_2^*))=0$, which is contradict with $Det(J(E_2^*))>0$. Hence, $V(L_{11},L_{22},L_{33},L_{44},g_1,g_2,g_3,h_1)\cap\Omega^*=\emptyset$. The proof for $i=2, 4$ is similar to $i=1$, which are omitted here.
\end{proof}

$\mathbf{Step\ IV:}$  Prove that $V(L_{11},L_{22},L_{33},L_{44},L_{55})\cap\Omega^*=\emptyset$.
\begin{lemma}\label{l15}
If $(\gamma,\eta,z)\in\Omega^*$ and \eqref{4.14} hold, then  $V(L_{11},L_{22},L_{33},L_{44},L_{55})\cap\Omega^*=\emptyset$. Moreover,

(i) if $h_1=0$ (or $h_2=0$, or $h_3=0$, or $h_4=0$), then $E_2^*$ is a weak focus of order at most 4;

(ii) if $r_2\neq0$ and $h_i\neq0$ ($i=1, 2, 3, 4$), then $E_2^*$ can be a weak focus of order 5.
\end{lemma}
\begin{proof}
(i) From Lemmas \ref{l12}, \ref{l13} and \ref{l14}, we have $V(L_{11},L_{22},L_{33},L_{44},g_1,g_2,g_3)\cap\Omega^*=\emptyset$ when $h_1=0$ or $h_2=0$ or $h_3=0$ or $h_4=0$, which implies $V_{33}\cap\Omega^*=\emptyset$ and $V_{34}\cap\Omega^*=\emptyset$. Thus, we have $V(L_{11},L_{22},L_{33},L_{44})\cap\Omega^*=\emptyset$ and $E_2^*$ is a weak focus of order at most 4 if $h_1=0$ or $h_2=0$ or $h_3=0$ or $h_4=0$.

(ii) From Lemmas \ref{l12}, \ref{l13} and \ref{l14}, we have that
\begin{equation*}
\begin{array}{ll}
\begin{aligned}
V(L_{11},L_{22},L_{33},L_{44},L_{55},r_1)\cap\Omega^*\subseteq V(L_{11},L_{22},L_{33},L_{44},r_1)\cap\Omega^*,\\[2ex]
V(L_{11},L_{22},L_{33},L_{44},L_{55},r_2)\cap\Omega^*\subseteq V(L_{11},L_{22},L_{33},L_{44},r_2)\cap\Omega^*,\\[2ex]
V(L_{11},L_{22},L_{33},L_{44},L_{55},g_1,g_2,g_3,S_1)\cap\Omega^*\subseteq V(L_{11},L_{22},L_{33},L_{44},g_1,g_2,g_3,S_1)\cap\Omega^*,\\[2ex]
V(L_{11},L_{22},L_{33},L_{44},L_{55},g_1,g_2,g_3,S_2)\cap\Omega^*\subseteq V(L_{11},L_{22},L_{33},L_{44},g_1,g_2,g_3,S_2)\cap\Omega^*.\\[2ex]
\end{aligned}
\end{array}
\end{equation*}
Besides, we have
\begin{equation*}
\begin{array}{ll}
\begin{aligned}
\mathrm{Res}(S_1S_2,S_3S_4,\gamma)\neq0,\quad \mathrm{Res}(S_1S_2,S_5S_6,\gamma)\neq0,
\end{aligned}
\end{array}
\end{equation*}
$V(S_1S_2,S_3S_4)=\emptyset$ and $V(S_1S_2,S_5S_6)=\emptyset$. Form above analysis, we obtain $V(L_{11},L_{22},L_{33},L_{44},L_{55})\cap\Omega^*=\emptyset$. Hence, $E_2^*$ is a weak focus of order at 5 under appropriate parameter conditions.
\end{proof}

From above analysis, we obtain the following theorem for the main results of this section.
\begin{theorem}\label{Hopf}
System \eqref{2.1} may undergo Hopf bifurcation, which can arise up to five limit cycles.
\end{theorem}

\begin{remark}
To illustrate our results more intuitively, let $\alpha=\frac{8}{625}$, $\beta=\frac{19881}{781250}$, $\gamma=\frac{281}{50}$, $\delta=\frac{1}{10}$, and $\eta=\frac{1}{50}$, then $L_{11}\approx0$, $L_{22}\approx0$, $L_{33}\approx0$ and $L_{44}=-\frac{128463}{125000}\approx-1.0277<0$ while $\alpha=\frac{1441}{5000}$, $\beta=\frac{1507}{5000}$ $\gamma=\frac{103}{100}$, $\delta=\frac{1}{10}$, and $\eta=\frac{1}{50}$, we have $L_{11}\approx0$, $L_{22}\approx0$, $L_{33}\approx0$ and $L_{44}=\frac{100291}{100000}\approx1.0029>0$, which means there exist a set of parameters $(\alpha,\beta,\gamma,\delta,\eta)$ such that $L_{11}=L_{22}=L_{33}=L_{44}\approx0$. Hence, we need to further calculate the fifth order focal quantity. Under small parameter perturbations, we obtained that system \eqref{2.1} has five limit cycles, which is exhibited in  Fig. \ref{5cycles}.
\end{remark}

\begin{figure}[ht!]
\centering
\begin{subfigure}{0.45\linewidth}
\centering
\includegraphics[width=0.9\linewidth]{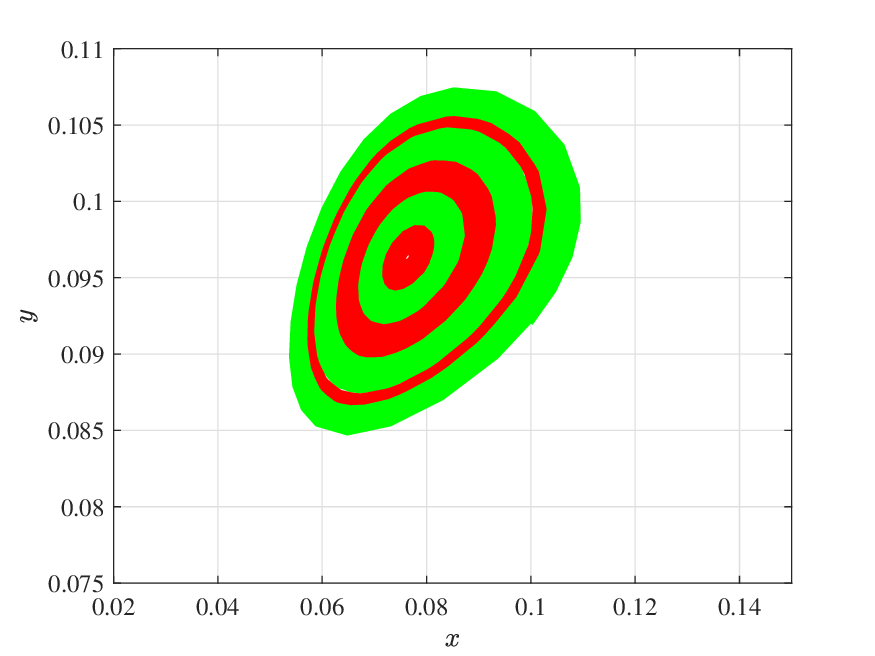}
\put(-101,10){$(a)$}
\end{subfigure}
\centering
\begin{subfigure}{0.45\linewidth}
\centering
\includegraphics[width=0.9\linewidth]{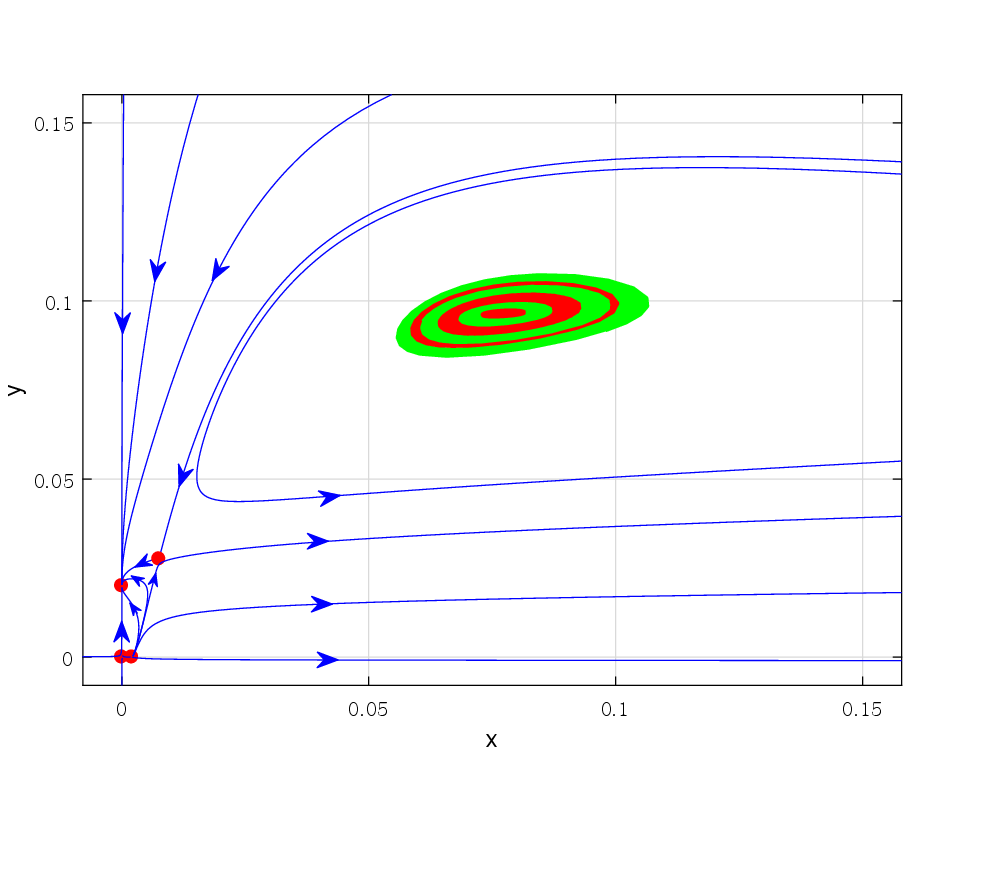}
\put(-90,10){$(b)$}
\end{subfigure}
\captionsetup{justification=centering}
\caption{For $\alpha=\frac{19}{1000}$, $\beta=\frac{21}{1000}$, $\gamma=\frac{73}{10}$, $\delta=\frac{1}{10}$ and $\eta=\frac{1}{50}$, system \eqref{2.1} exists five limit cycles created by Hopf bifurcation around $E_2^*$: (a) five limit cycles near $E_2^*$; (b) the global phase portraits in this case.}
\label{5cycles}
\end{figure}
\begin{remark}
The coexistence of multiple limit cycles shows complicated population oscillations of the system. It contributes to the high dependence of oscillation amplitudes on initial values, because the positions of different limit cycles restrict the amplitudes, which is shown in Fig. \ref{amplitude}.
\end{remark}
\begin{figure}
\begin{center}
\begin{overpic}[scale=0.60]{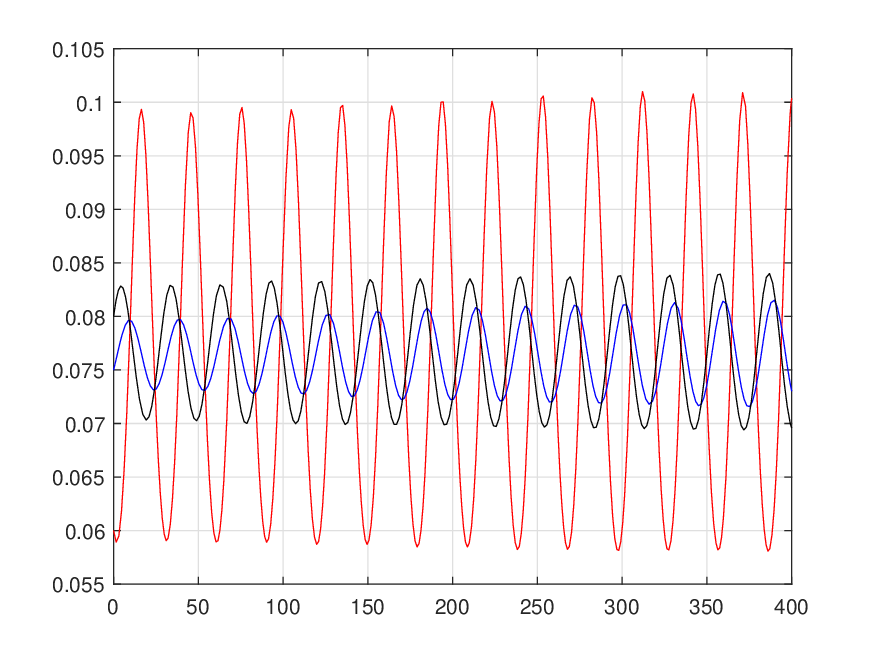}
\end{overpic}
\vspace{-6mm}
\end{center}
\caption{ Different population oscillation amplitudes for initial values in different regions between limit cycles in Fig. \ref{5cycles} }
\label{amplitude}
\end{figure}

\section{Conclusions}

For a modified Leslie-Gower type predator-prey model with Lotka-Volterra type functional response and additive Allee effect in prey, we consider the existence and exact classifications of non-hyperbolic equilibria, and the corresponding high codimension bifurcations. The details are exhibited in Section 2.3 for determining the exact types and codimension of nilpotent equilibria, and in Section 5 to show the center-type equilibrium $E_2^*$ is a weak focus, where system \eqref{2.1} can undergo degenerate Hopf bifurcation and emerge up to 5 limit cycles.

When system \eqref{2.1} exhibits nilpotent cusp bifurcation of codimension 4 and Hopf bifurcation, there exist a sequence of lower codimension bifurcations originating from the nilpotent cusp of codimension 4 and Hopf bifurcation, which can be seen as the organizing centers of the bifurcation set. For the detailed description of focus type and cusp type nilpotent bifurcations of codimension 4, one can refer to literatures \cite{CMJ,MJH,JMC}.
Different from \cite{JMC} and \cite{ZZ}, who found the coexistence of 3 limit cycles, we show that a center-type equilibrium of system \eqref{2.1} is a weak focus with order up to 5, and system \eqref{2.1} can exhibit 5 small-amplitude limit cycles by Hopf bifurcation. This phenomenon is rarely studied in the existing literature of predator-prey systems.

From the results in Section 2, the codimension of the non-hyperbolic positive equilibria depends on not only the parameter values, but also the locations of the equilibria. The bifurcation processes in Sections 3 and 4 show the great changes of dynamical behaviors of the system with
small perturbations of either parameters or initial conditions. It indicates the high sensitivities of the population dynamics on both parameter values and initial densities. Moreover, it is notable that, the additive Allee effects in prey may cause coextinction of both populations. Our results also indicate more complexity for the coexistence and oscillations of both populations in predator-prey system, including the final population sizes, the amplitudes as well as the period of the population oscillations.

The additive Allee effect helps to understand the extinction risk faced by small populations, especially endangered species or early invasive populations. At the same time, it also provides theoretical support for studying the bistable behavior of populations, such as stable equilibrium and extinction states. By quantifying the additive Allee effect, we can help analyze the stability of ecosystems under external disturbances and develop more effective conservation measures, such as artificial intervention to increase population density.

Introducing the additive Allee effect can better describe the dynamic behavior of low-density populations, predict population fate, and provide theoretical basis for ecological conservation and species management. At the same time, it enriches the theoretical framework of ecological models, reveals the complexity of ecosystems, and has important scientific and practical significance.

\section*{Declarations}
\subsection*{Author Contributions} XW conceptualized, validated the numerical results and prepared the original manuscript draft. KW was responsible for the qualitative analysis and contributed to the development and investigation. LZ contributed to the development, manuscript revision and editing. All the authors reviewed the paper.

\subsection*{Data Availability} No applicable to this article, since no datasets were generated or analyzed during the current study.

\subsection*{Conflict of interest} The authors declare that they have no competing interest.

\subsection*{Funding}
This work was supported by the National Natural Science Foundation of China (Nos. 12571541, 12361034), Science and Technology Plan Project of Guizhou Province(ZK[2022]G118) and Postgraduate Research Fund of Guizhou Province(2024YJSKYJJ057).

\begin{appendices}
\section{Proof of Lemma \ref{l2}}\label{secA1}

\begin{proof}
We only consider the equilibrium $E_4\big(\frac{1}{2}(1-\alpha),0\big)$.
Translating the equilibrium $E_4$ to the origin, system \eqref{2.1} can be changed as follows
\begin{equation}
\begin{array}{ll}\label{1.8}
\left\{
\begin{aligned}
\dot{X}&=-\frac{\gamma(1-\alpha)}{2}Y-\gamma XY-\frac{1-\alpha}{1+\alpha}X^2+o(|X,Y|^3),\\[2ex]
\dot{Y}&=\delta Y-\frac{2\delta Y^2}{1+2\eta-\alpha}+o(|X,Y|^3).
\end{aligned}
\right.
\end{array}
\end{equation}
Making the transformation $x=X+\frac{\gamma(1-\alpha)}{2\delta^2}Y$, $y=-\frac{1}{\delta}Y$ and $t=\frac{1}{\delta}\tau$ (still denote $\tau$ by $t$), system \eqref{1.8} can be transformed into
\begin{equation}
\begin{array}{ll}\label{1.10}
\left\{
\begin{aligned}
\dot{x}&=-\frac{1-\alpha}{\delta(1+\alpha)}x^2-\frac{\gamma(1-2\alpha+\alpha^2-\delta-\alpha\delta)}{(1+\alpha)\delta^3}xy\\[2ex]
&\quad+\frac{\gamma(\alpha-1)\big((\alpha-2\eta-1)((1-\alpha)^2-2\delta(1+\alpha))\gamma-4\delta^2(1+\alpha)\big)}{4\delta^5(1+\alpha)(\alpha-2\eta-1)}y^2+o(|x,y|^3),\\[2ex]
\dot{y}&=y+\frac{2}{\delta(1+2\eta-\alpha)}y^2+o(|x,y|^3).
\end{aligned}
\right.
\end{array}
\end{equation}
Note that $-\frac{1-\alpha}{\delta(1+\alpha)}<0$ when $0<\alpha<1$. By Theorem 7.1 in chapter 2 of \cite{ZT}, it yields that $E_4$ is a saddle-node with a stable parabolic sector in the right half plane of $\mathbf{R}^2_+$.
\end{proof}

\section{Proof of Lemma \ref{t4}}\label{secA2}

\begin{proof}
Step 1: Note that $m_{20}=d_{20}\neq0$. Setting $y_1=x+\frac{m_{21}}{3m_{20}}xy+\frac{5m_{21}^2}{54m_{20}}x^4$, $y_2=y+\frac{m_{21}}{3m_{20}}y^2+\frac{m_{21}}{3}x^3+\frac{m_{21}m_{30}}{3m_{20}}x^4+\frac{10m_{21}^2}{27m_{20}}x^3y$, system \eqref{2.8} can be transformed into
\begin{equation}
\begin{array}{ll}\label{2.10}
\left\{
\begin{aligned}
\dot{x}&=y+n_{50}x^5+n_{41}x^4y+o(|x,y|^5),\\[2ex]
\dot{y}&=r_{20}x^2+r_{30}x^3+r_{40}x^4+r_{31}x^3y+r_{50}x^5\\[2ex]
&\quad+r_{41}x^4y+r_{32}x^3y^2+r_{23}x^2y^3+o(|x,y|^5),
\end{aligned}
\right.
\end{array}
\end{equation}
where
\begin{equation*}
\begin{array}{ll}
\begin{aligned}
n_{50}&=-\frac{m_{21}m_{40}}{3m_{20}},\quad n_{41}=\frac{m_{21}(m_{21}m_{30}-m_{20}m_{31})}{3m_{20}^2},\quad r_{31}=m_{31}-\frac{m_{21}m_{30}}{m_{20}},\\[2ex]
r_{20}&=m_{20},\quad r_{30}=m_{30},\quad r_{50}=m_{50}+\frac{4m_{21}^2}{27},\quad r_{41}=m_{41}+\frac{2m_{21}m_{40}}{3m_{20}},\\[2ex]
r_{40}&=m_{40},\quad r_{32}=\frac{m_{21}(3m_{21}m_{30}+2m_{20}m_{31})}{3m_{20}^2},\quad r_{23}=\frac{m_{21}^3}{3m_{20}^2}.
\end{aligned}
\end{array}
\end{equation*}

Step 2: Letting $x=X+\frac{4n_{41}+r_{32}}{20}X^5+\frac{r_{23}}{12}X^4Y$, $y=Y-n_{50}X^5+\frac{r_{32}}{4}X^4Y+\frac{r_{23}}{3}X^3Y^2$, system \eqref{2.10} can be written as
\begin{equation}
\begin{array}{ll}\label{2.11}
\left\{
\begin{aligned}
\dot{X}&=Y+o(|X,Y|^5),\\[2ex]
\dot{Y}&=s_{20}X^2+s_{30}X^3+s_{40}X^4+s_{31}X^3Y+s_{50}X^5+s_{41}X^4Y+o(|X,Y|^5),
\end{aligned}
\right.
\end{array}
\end{equation}
where
\begin{equation*}
\begin{array}{ll}
\begin{aligned}
s_{20}&=r_{20},\quad s_{30}=r_{30},\quad s_{40}=r_{40},\quad s_{31}=r_{31},\quad s_{50}=r_{50},\quad s_{41}=r_{41}+5n_{50}.
\end{aligned}
\end{array}
\end{equation*}

Step 3: Making the following transformation
\begin{equation*}
\begin{array}{ll}
\begin{aligned}
X&=x-\frac{s_{30}}{4s_{20}}x^2+\frac{15s_{30}^2-16s_{20}s_{40}}{80s_{20}^2}x^3-\frac{175s_{30}^3-336s_{20}s_{30}s_{40}+160s_{20}^2s_{50}}{960s_{20}^3}x^4,\\[2ex]
Y&=y,\\[2ex] t&=(1-\frac{s_{30}}{2s_{20}}x+\frac{45s_{30}^2-48s_{20}s_{40}}{80s_{20}^2}x^2-\frac{175s_{30}^3-336s_{20}s_{30}s_{40}+160s_{20}^2s_{50}}{240s_{20}^3}x^3)\tau.
\end{aligned}
\end{array}
\end{equation*}
System \eqref{2.11} can be changed into (still denote $\tau$ by $t$)
\begin{equation}
\begin{array}{ll}\label{2.12}
\left\{
\begin{aligned}
\dot{x}&=y+o(|x,y|^5),\\[2ex]
\dot{y}&=w_{20}x^2+w_{31}x^3y+w_{41}x^4y+o(|x,y|^5),
\end{aligned}
\right.
\end{array}
\end{equation}
where
\begin{equation*}
\begin{array}{ll}
\begin{aligned}
w_{20}=s_{20},\quad w_{31}s_{31},\quad w_{41}=s_{41}-\frac{5s_{30}s_{31}}{4s_{20}}.
\end{aligned}
\end{array}
\end{equation*}

Step 4: From above analysis, we have $w_{20}=d_{20}=-\frac{\gamma^3(1-\gamma\eta)}{2(1+\gamma)(2+3\gamma)}<0$. Setting $x=-X$,$y=-\sqrt{-w_{20}}Y$ and $t=\frac{1}{\sqrt{-w_{20}}}\tau$, system \eqref{2.12} can be written as (still denote $\tau$ by $t$)
\begin{equation}
\begin{array}{ll}\label{2.13}
\left\{
\begin{aligned}
\dot{X}&=Y+o(|X,Y|^5),\\[2ex]
\dot{Y}&=X^2+MX^3Y+NX^4Y+o(|X,Y|^5),
\end{aligned}
\right.
\end{array}
\end{equation}
where
\begin{equation*}
\begin{array}{ll}
\begin{aligned}
M&=-\frac{(1+\gamma)(2+3\gamma)^3}{4\gamma^2(1-\gamma\eta)^3\big((2+\gamma)(1+2\gamma)\eta+\gamma\big)}\sqrt{\frac{\gamma(1-\gamma\eta)}{6\gamma^2+10\gamma+4}}\varrho_1,\\[2ex]
N&=-\frac{(1+\gamma)(2+3\gamma)^4}{32\gamma^3(1-\gamma\eta)^4\big((2+\gamma)(1+2\gamma)\eta+\gamma\big)^2}\sqrt{\frac{\gamma(1-\gamma\eta)}{6\gamma^2+10\gamma+4}}\varrho_2,\\[2ex]
\varrho_1&=4\eta\gamma^3+(19\eta-1)\gamma^2+(20\eta-8)\gamma+4(\eta-2),\\[2ex]
\varrho_2&=24\eta^2\gamma^6+(102\eta^2-18\eta)\gamma^5+(71\eta^2-118\eta+3)\gamma^4-(156\eta^2+158\eta-46)\gamma^3\\[2ex]
&\quad-(248\eta^2-8\eta-88)\gamma^2-(112\eta^2-104\eta-58)\gamma+16\eta(2-\eta).
\end{aligned}
\end{array}
\end{equation*}

If $0<\gamma<\frac{1}{\eta}$ and $\eta\neq\frac{\gamma^2+8\gamma+8}{4\gamma^3+19\gamma^2+20\gamma+4}$, i.e., $M\neq0$, then $E^*$ is a cusp of  codimension 3 of system \eqref{2.1}. If $\eta=\eta_0=\frac{\gamma^2+8\gamma+8}{4\gamma^3+19\gamma^2+20\gamma+4}$, i.e.,$M=0$ and $N\neq0$, then $E^*$ is a cusp of codimension 4.
\end{proof}

\section{The coefficients for system \eqref{2.4}}\label{secA3}

\begin{equation*}
\begin{array}{ll}
\begin{aligned}
a_{02}&=\frac{4(1+\gamma)^2}{\gamma(1-\alpha-\alpha\gamma+2\eta+\gamma\eta)(\gamma\eta+\alpha\gamma+\alpha-1)},\quad a_{12}=\frac{8(1+\gamma)^3}{\gamma(1-\alpha-\alpha\gamma+2\eta+\gamma\eta)^2(1-\alpha-\alpha\gamma-\gamma\eta)},\\[2ex]
a_{03}&=\frac{16(1+\gamma)^4}{\gamma^2(1-\alpha-\alpha\gamma+2\eta+\gamma\eta)^2(\gamma\eta+\alpha\gamma+\alpha-1)^2},\quad a_{22}=\frac{16(1+\gamma)^4}{\gamma(1-\alpha-\alpha\gamma+2\eta+\gamma\eta)^3(\gamma\eta+\alpha\gamma+\alpha-1)},\\[2ex]
a_{13}&=\frac{64(1+\gamma)^5}{\gamma^2(\alpha+\alpha\gamma-2\eta-\gamma\eta-1)^3(\gamma\eta+\alpha\gamma+\alpha-1)^2},\quad a_{04}=\frac{64(1+\gamma)^6}{\gamma^3(1-\alpha-\alpha\gamma+2\eta+\gamma\eta)^3(\gamma\eta+\alpha\gamma+\alpha-1)^3},\\[2ex]
a_{14}&=\frac{384(1+\gamma)^7}{\gamma^3(1-\alpha-\alpha\gamma+2\eta+\gamma\eta)^4(1-\alpha-\alpha\gamma-\gamma\eta)^3},\quad
a_{23}=\frac{192(1+\gamma)^6}{\gamma^2(1-\alpha-\alpha\gamma+2\eta+\gamma\eta)^4(1-\alpha-\alpha\gamma-\gamma\eta)^2},\\[2ex]
a_{32}&=\frac{32(1+\gamma)^5}{\gamma(1-\alpha-\alpha\gamma+2\eta+\gamma\eta)^4(1-\alpha-\alpha\gamma-\gamma\eta)},\quad a_{05}=\frac{256(1+\gamma)^8}{\gamma^4(1-\alpha-\alpha\gamma+2\eta+\gamma\eta)^4(1-\alpha-\alpha\gamma-\gamma\eta)^4},\\[2ex]
b_{20}&=-\frac{\gamma(1-\alpha-\alpha\gamma-\gamma\eta)^2}{2(1+\alpha+\alpha\gamma-\gamma\eta)},\quad b_{11}=\frac{\alpha(1+\gamma)(2+3\gamma)+(2+\gamma)(\gamma\eta-1)}{1+\alpha+\alpha\gamma-\gamma\eta},\quad b_{21}=\frac{12\alpha(1+\gamma)^3}{(1+\alpha+\alpha\gamma-\gamma\eta)^2},\\[2ex]
b_{02}&=2(1+\gamma)\bigg(\frac{1}{1-\alpha(1+\gamma)+(2+\gamma)\eta}-\frac{1}{\alpha+\alpha\gamma+\gamma\eta-1}-\frac{1+\gamma}{\gamma(1+\alpha+\alpha\gamma-\alpha\gamma\eta)}\bigg),\\[2ex]
b_{12}&=\bigg(\frac{12(1+\gamma)^2}{(1+\alpha+\alpha\gamma-\gamma\eta)^2}+\frac{12(1+\gamma)^2\big(\gamma+\alpha(1+\gamma)(2+\gamma)-\gamma^2\eta\big)}
{\gamma(1+\alpha+\alpha\gamma-\gamma\eta)^2(\alpha+\alpha\gamma+\gamma\eta-1)}-\frac{4(1+\gamma)^2}{\big(\alpha\gamma+\alpha-1-\eta(2+\gamma)\big)^2}\bigg),\\[2ex]
b_{03}&=\frac{8(1+\gamma)^3}{\gamma^2(1+\alpha+\alpha\gamma-\gamma\eta)^2(\alpha+\alpha\gamma+\gamma\eta-1)^2(\alpha+\alpha\gamma-1-2\eta-\gamma\eta)^2}\big((\gamma\eta-1)^3\gamma-\alpha^3(1+\gamma)^3\\[2ex]
&\quad+\alpha^2(1+\gamma)^2(4+5\gamma+8\eta+12\gamma\eta+3\gamma^2\eta)-\alpha(1+\gamma)(2+3\gamma+8\eta+12\gamma\eta+2\gamma^2\eta+8\eta^2\\[2ex]
&\quad+16\gamma\eta^2+10\gamma^2\eta^2+3\gamma^3\eta^2)\big),\quad b_{40}=\frac{4\alpha\gamma(1+\gamma)^3(1-\alpha-\alpha\gamma-\gamma\eta)}{(1+\alpha+\alpha\gamma-\gamma\eta)^3},\quad b_{31}=\frac{32\alpha(1+\gamma)^4}{(1+\alpha+\alpha\gamma-\gamma\eta)^3},\\[2ex]
b_{13}&=\frac{32(1+\gamma)^4}{\gamma^2(1+\alpha+\alpha\gamma-\gamma\eta)^3(\alpha+\alpha\gamma+\gamma\eta-1)^2(\alpha+\alpha\gamma-1-2\eta-\gamma\eta)^3}\big(\alpha^4(1+\gamma)^4(4+5\gamma)-\gamma(\gamma\eta-1)^4\\[2ex]
&\quad+12(1+\gamma)^3(\alpha+2\alpha\eta+\alpha\gamma\eta)^2-2\alpha^3(1+\gamma)^3(6+5\gamma+12\eta+18\gamma\eta+7\gamma^2\eta)-2\alpha(1+\gamma)(2+3\gamma+12\eta\\[2ex]
&\quad+18\gamma\eta+3\gamma^2\eta+24\eta^2+48\gamma\eta^2+30\gamma^2\eta^2+9\gamma^3\eta^2+16\eta^3+40\gamma\eta^3+36\gamma^2\eta^3+14\gamma^3\eta^3+\gamma^4\eta^3)\big),\\[2ex]
b_{50}&=\frac{8\alpha\gamma(1+\gamma)^4(\alpha+\alpha\gamma+\gamma\eta-1)}{(1+\alpha+\alpha\gamma-\gamma\eta)^4},\quad b_{41}=-\frac{80\alpha(1+\gamma)^5}{(1+\alpha+\alpha\gamma-\gamma\eta)^4},\quad b_{30}=\frac{2\alpha\gamma(1+\gamma)^2(\alpha+\alpha\gamma+\gamma\eta-1)}{(1+\alpha+\alpha\gamma-\gamma\eta)^2},\\[2ex]
b_{04}&=\frac{32(1+\gamma)^5}{\gamma^3(1+\alpha+\alpha\gamma-\gamma\eta)^3(1-\alpha-\alpha\gamma-\gamma\eta)^3(\alpha+\alpha\gamma-1-2\eta-\gamma\eta)^3}\big(\alpha(1+\gamma)^4(2+3\gamma)-\gamma(\gamma\eta-1)^4\\[2ex]
&\quad+6(1+\gamma)^3(\alpha+2\alpha\eta+\alpha\gamma\eta)^2-2\alpha^3(1+\gamma)^3(3+2\gamma+6\eta+9\gamma\eta+4\gamma^2\eta)-2\alpha(1+\gamma)(1+2\gamma+6\eta\\[2ex]
&\quad+9\gamma\eta+12\eta^\eta2+24\gamma\eta^2+15\gamma^2\eta^2+6\gamma^3\eta^2+8\eta^3+20\gamma\eta^3+18\gamma^2\eta^3+7\gamma^3\eta^3)\big),\\[2ex]
\end{aligned}
\end{array}
\end{equation*}
\begin{equation*}
\begin{array}{ll}
\begin{aligned}
b_{14}&=\frac{64(1+\gamma)^6}{\gamma^3(1+\alpha+\alpha\gamma-\gamma\eta)^4(\alpha+\alpha\gamma-1-2\eta-\gamma\eta)^4(\alpha+\alpha\gamma+\gamma\eta-1)^3}\big(\alpha^5(1+\gamma)^5-3\gamma(\gamma\eta-1)^5\\[2ex]
&\quad-\alpha^4(1+\gamma)^4(40+49\gamma+80\eta+120\gamma\eta+31\gamma^2\eta)+6\alpha^3(1+\gamma)^3(10+9\gamma+(40+60\gamma+22\gamma^2)\eta\\[2ex]
&\quad+(40+80\gamma+50\gamma^2+9\gamma^3)\eta^2)+2\alpha^2(1+\gamma)^2(23\gamma^4\eta^3+20(1+2\eta)^3+\gamma^3\eta^2(51+140\eta)\\[2ex]
&\quad+3\gamma^2\eta(23+20\eta(5+6\eta))+\gamma(17+20\eta(9+4\eta(6+5\eta))))+\alpha(1+\gamma)(19\gamma^5\eta^4+10(1+2\eta)^4\\[2ex]
&\quad+2\gamma^4\eta^3(2+45\eta)+\gamma^3\eta^2(57+20\eta(7+8\eta))+4\gamma^2\eta(1+5\eta(15+4\eta(9+7\eta)))\\[2ex]
&\quad+\gamma(19+40\eta(3+4\eta(3+\eta(5+3\eta)))))\big),\\[2ex]
b_{32}&=\frac{16(1+\gamma)^4}{\gamma(1+\alpha+\alpha\gamma-\gamma\eta)^4(\alpha+\alpha\gamma-2\eta-\gamma\eta-1)^4(\alpha+\alpha\gamma+\gamma\eta-1)}\big(\alpha^5(1+\gamma)^5(20+19\gamma)\\[2ex]
&\quad-\gamma(\gamma\eta-1)^5-\alpha^4(1+\gamma)^4(80+83\gamma+(160+\gamma(240+77\gamma))\eta)+2\alpha^3(1+\gamma)^3(60+59\gamma\\[2ex]
&\quad+2(120+\gamma(180+61\gamma))\eta+(240+\gamma(480+\gamma(300+59\gamma)))\eta^2)-2\alpha^2(1+\gamma)^2(41\gamma^4\eta^3\\[2ex]
&\quad+40(1+2\eta)^3+\gamma^3\eta^2(117+280\eta)+3\gamma^2\eta(41+40\eta(5+6\eta))+\gamma(39+40\eta(9+4\eta(6+5\eta))))\\[2ex]
&\quad+\alpha(1+\gamma)(23\gamma^5\eta^4+20(1+2\eta)^4+4\gamma^4\eta^3(17+45\eta)+2\gamma^3\eta^2(69+40\eta(7+8\eta))\\[2ex]
&\quad+4\gamma^2\eta(17+10\eta(15+4\eta(9+7\eta)))+\gamma(23+80\eta(3+4\eta(3+\eta(5+3\eta)))))\big),\\[2ex]
b_{23}&=\frac{-32(1+\gamma)^5}{\gamma^2(1+\alpha+\alpha\gamma-\gamma\eta)^4(\alpha+\alpha\gamma-1-2\eta-\gamma\eta)^4(\alpha+\alpha\gamma+\gamma\eta-1)^2}\big(\alpha^5(1+\gamma)^5(20+17\gamma)\\[2ex]
&\quad-3\gamma(-1+\gamma\eta)^5-\alpha^4(1+\gamma)^4(80+89\gamma+(160+\gamma(240+71\gamma))\eta)+6\alpha^3(1+\gamma)^3(20+19\gamma\\[2ex]
&\quad+80\eta+6\gamma(20+7\gamma)\eta+(80+\gamma(160+\gamma(100 +19\gamma)))\eta^2)-2\alpha^2(1+\gamma)^2(43\gamma^4\eta^3\\[2ex]
&\quad+40(1+2\eta)^3 +\gamma^3\eta^2(111+280\eta)+3\gamma^2\eta(43+40\eta(5+6\eta))+\gamma(37+40\eta(9+4\eta(6\\[2ex]
&\quad+5\eta))))+\alpha(1+\gamma)(29\gamma^5\eta^4+20(1+2\eta)^4+4\gamma^4\eta^3(11+45\eta)+2\gamma^3\eta^2(87+40\eta(7+8\eta))\\[2ex]
&\quad+4\gamma^2\eta(11+10\eta(15+4\eta(9+7\eta)))+\gamma(29+80\eta(3+4\eta(3+\eta(5+3\eta)))))\big),\\[2ex]
b_{05}&=\frac{-128(1+\gamma)^7}{\gamma^4(1+\alpha+\alpha\gamma-\gamma\eta)^4(\alpha+\alpha\gamma-1-2\eta-\gamma\eta)^4(\alpha+\alpha\gamma+\gamma\eta-1)^4}\big(\alpha^5(1+\gamma)^5(2+\gamma)\\[2ex]
&\quad-\gamma(\gamma\eta-1)^5-\alpha^4(1+\gamma)^4(8+11\gamma+(4+\gamma)(4+5\gamma)\eta)+2\alpha^3(1+\gamma)^3(6+5\gamma\\[2ex]
&\quad+2(12+\gamma(18+7\gamma))\eta+(24+\gamma(48+5\gamma(6+\gamma)))\eta^2)-2\alpha^2(1+\gamma)^2(4+3\gamma\\[2ex]
&\quad+3(8+\gamma(12+5\gamma))\eta+3(16+\gamma(4+\gamma)(8+3\gamma))\eta^2+(32+\gamma(80+\gamma(72+\gamma(28+5\gamma))))\eta^3)\\[2ex]
&\quad+\alpha(1 +\gamma)(5\gamma^5\eta^4+2(1+2\eta)^4+2\gamma^4\eta^3(-2+9\eta)+2\gamma^3\eta^2(15+4\eta(7+8\eta))\\[2ex]
&\quad+4\gamma^2\eta(-1+\eta(15+4\eta(9+7\eta)))+\gamma(5+8\eta(3+4\eta(3+\eta(5+3\eta)))))\big).
\end{aligned}
\end{array}
\end{equation*}

\section{The coefficients for system \eqref{2.5}}

\begin{equation*}
\begin{array}{ll}
\begin{aligned}
c_{30}&=-a_{02}b_{20},\quad c_{21}=-a_{02}b_{11},\quad c_{50}=-\frac{1}{4}a_{02}\big(b_{02}(5b_{02}b_{20}-2b_{30})+4b_{40}\big),\\[2ex]
c_{40}&=a_{02}(b_{02}b_{20}-b_{30}),\quad c_{31}=a_{02}(\frac{b_{02}b_{11}}{2}-b_{21}-a_{02}b_{20}),\quad c_{04}=a_{04}-a_{02}a_{03},\\[2ex]
c_{22}&=\frac{3a_{12}b_{02}}{2}+a_{22}-a_{02}(3b_{02}^2+b_{12}),\quad c_{13}=a_{13}-a_{02}b_{03}-2b_{02}(a_{02}^2-a_{03}),\\[2ex]
c_{12}&=a_{12}+2a_{02}b_{02},\quad c_{32}=\frac{1}{2}b_{02}(a_{12}b_{02}-3a_{02}^2b_{11}-a_{02}b_{12})+2b_{02}(a_{22}+2a_{02}b_{02}^2)\\[2ex]
&\quad-a_{02}(a_{02}b_{21}+b_{22})+a_{32},\quad c_{05}=a_{05}+a_{02}(a_{02}a_{03}-a_{04}),\quad c_{03}=a_{03},\\[2ex]
c_{23}&=\frac{1}{2}b_{02}(3a_{02}a_{12}+5a_{13})+b_{02}^2(7a_{02}^2+a_{03})+a_{02}(a_{22}-b_{02}b_{03}-b_{13})+a_{23},\\[2ex]
c_{41}&=-\frac{1}{2}a_{02}(b_{02}^2b_{11}+4a_{02}b_{30}+2b_{31}),\quad c_{14}=a_{14}+a_{02}^2b_{03}-a_{02}b_{04}+b_{02}(2a_{02}^3+3a_{04}),\\[2ex]
\end{aligned}
\end{array}
\end{equation*}
and
\begin{equation*}
\begin{array}{ll}
\begin{aligned}
d_{20}&=b_{20},\quad d_{11}=b_{11},\quad d_{30}=b_{30},\quad d_{21}=\frac{b_{02}b_{11}}{2}+2a_{02}b_{20}+b_{21},\\[2ex]
d_{12}&=2b_{02}^2+a_{02}b_{11}+b_{12},\quad d_{03}=b_{03},\quad d_{40}=\frac{1}{4}b_{02}(b_{02}b_{20}+2b_{30})+b_{40},\\[2ex]
d_{31}&=b_{31}+b_{02}b_{21}+3a_{02}b_{30},\quad d_{13}=a_{02}(b_{12}-2b_{02}^2)-b_{02}(a_{12}-3b_{03})+b_{13},\\[2ex]
d_{22}&=-\frac{1}{2}b_{02}(2b_{02}^2-2a_{02}b_{11}-3b_{12})+a_{02}(a_{02}b_{20}+2b_{21})+b_{22},\\[2ex]
d_{50}&=-{1}{4}b_{02}^2(b_{02}b_{20}-b_{30})+b_{02}b_{40}+b_{50},\quad d_{05}=b_{05}+b_{02}(a_{02}a_{03}-a_{04}),\\[2ex]
d_{32}&=b_{02}^4-\frac{1}{2}b_{02}^2(3a_{02}b_{11}-b_{12})+2b_{02}(a_{02}b_{21}+b_{22})+3a_{02}(a_{02}b_{30}+b_{31})+b_{32},\\[2ex]
d_{41}&=\frac{1}{4}b_{02}(b_{02}b_{21}+6b_{31})-a_{02}\big(b_{02}(b_{02}b_{20}-b_{30})-4b_{40}\big)+b_{41},\\[2ex]
d_{14}&=b_{02}^2(2a_{02}^2-a_{03})-b_{02}(a_{13}-3b_{04})+a_{02}(b_{02}b_{03}+b_{13})+b_{14},\\[2ex]
d_{23}&=-\frac{1}{2}b_{02}(a_{12}b_{02}-5b_{13})+b_{02}^2(5a_{02}b_{02}+b_{03})-b_{02}(a_{22}-2a_{02}b_{12})\\[2ex]
&\quad+a_{02}(a_{02}b_{21}+2b_{22})+b_{23},\quad d_{04}=b_{04}-a_{03}b_{02}.
\end{aligned}
\end{array}
\end{equation*}

\section{The coefficients $e_{ij}$ and $h_{ij}$ for system \eqref{2.6}}

\begin{equation*}
\begin{array}{ll}
\begin{aligned}
e_{40}&=c_{40}-\frac{1}{4}d_{20}(c_{12}+d_{03}),\quad e_{31}=c_{31}-2c_{03}d_{20}-\frac{1}{2}d_{11}(c_{12}+d_{03}),\\[2ex]
e_{22}&=c_{22}-2c_{03}d_{11},\quad e_{13}=c_{13},\quad e_{04}=c_{04},\quad e_{14}=c_{14}+c_{03}(c_{12}+3d_{03}),\\[2ex]
\end{aligned}
\end{array}
\end{equation*}
\begin{equation*}
\begin{array}{ll}
\begin{aligned}
e_{50}&=\frac{1}{2}\big(2c_{50}+c_{30}d_{12}-d_{30}(c_{12}+d_{03})\big),\quad e_{05}=c_{05},\\[2ex]
e_{41}&=\frac{1}{6}\big(c_{21}(4c_{21}+5d_{12})-3d_{21}(c_{12}+d_{03})\big)-2(c_{12}c_{30}+c_{03}d_{30})+c_{41},\\[2ex]
e_{32}&=\frac{1}{6}c_{12}(8c_{21}+7d_{12})-2c_{03}(3c_{30}+d_{21})+2c_{21}d_{03}+c_{32},\\[2ex]
e_{23}&=\frac{1}{2}\big(c_{12}(c_{12}+5d_{03})+c_{03}(4c_{21}+3d_{12})+2c_{23}\big)
\end{aligned}
\end{array}
\end{equation*}
and
\begin{equation*}
\begin{array}{ll}
\begin{aligned}
h_{20}&=d_{20},\quad h_{11}=d_{11},\quad h_{30}=d_{30},\quad h_{21}=d_{21}+3c_{30},\quad h_{13}=d_{13}+c_{03}d_{11},\\[2ex]
h_{40}&=\frac{1}{6}d_{20}(4c_{21}-d_{12})-c_{30}d_{11}+d_{40},\quad h_{22}=\frac{1}{2}d_{11}(c_{12}-d_{03})+2c_{03}d_{20}+d_{22},\\[2ex]
h_{31}&=\frac{1}{6}d_{11}(2c_{21}+d_{12})+d_{20}(c_{12}-d_{03})+d_{31},\quad h_{50}=d_{50}-c_{30}d_{21}+c_{21}d_{30},\\[2ex]
h_{04}&=d_{04},\quad h_{32}=\frac{1}{6}d_{12}(2c_{21}+7d_{12})+c_{12}d_{21}+3(c_{03}d_{30}-3c_{30}d_{03})+d_{32},\\[2ex]
h_{23}&=\frac{d_{12}}{2}(c_{12}+8d_{03})+2c_{03}d_{21}+d_{23},\quad h_{14}=3d_{03}^2+c_{03}d_{12}+d_{14},\quad h_{05}=d_{05},\\[2ex]
h_{41}&=\frac{1}{6}\big(2d_{21}(2c_{21}+d_{12})+3d_{30}(3c_{12}-d_{03})-21c_{30}d_{12}\big)+d_{41}.
\end{aligned}
\end{array}
\end{equation*}

\section{The coefficients $\bar{a}_{ij}$ and $\bar{b}_{ij}$ for system \eqref{2.17}}

\begin{equation*}
\begin{array}{ll}
\begin{aligned}
\bar{a}_{00}&=-\frac{(2+\gamma)\chi_9}{\chi_3\chi_8},\quad \bar{a}_{01}=-\frac{\gamma^2(2+\gamma)}{\chi_3},\quad \bar{a}_{20}=-1+\frac{(\chi_8-2\gamma-\gamma^2)\chi_{10}}{\chi_8^3},\quad \bar{a}_{11}=-\gamma,\\[2ex] \bar{a}_{10}&=1-\frac{2\gamma(2+\gamma)}{\chi_3}-\frac{2\gamma(1+\gamma)(4+\gamma)}{\chi_3}+\bigg(\mu_1+\frac{4(1+\gamma)^3(2+\gamma)^2}{\chi_3^2}\bigg)\frac{(\gamma^2+2\gamma-\chi_8)\chi_3-\chi_8}{\chi_8^2},\\[2ex] \bar{a}_{30}&=\bigg(\frac{4(1+\gamma)^3(2+\gamma)^2}{\chi_3^2}+\lambda_1\bigg)\frac{(\gamma^2+2\gamma-\chi_8)\chi_3^3}{\chi_8^4},\quad \bar{b}_{00}=\frac{2(1+\gamma)(4+\gamma)\chi_2\lambda_4}{\chi_1\chi_3},\\[2ex]
\bar{a}_{40}&=-\bigg(\frac{4(1+\gamma)^3(2+\gamma)^2}{\chi_3^2}+\lambda_1\bigg)\frac{(\gamma^2+2\gamma-\chi_8)\chi_3^4}{\chi_8^5},\quad \bar{b}_{01}=-\frac{\chi_2\chi_4}{\chi_1\chi_3},\quad \bar{b}_{20}=-\frac{\chi_2\chi_5}{\chi_1^3},\\[2ex]
\bar{a}_{50}&=\bigg(\frac{4(1+\gamma)^3(2+\gamma)^2}{\chi_3^2}+\lambda_1\bigg)\frac{(\gamma^2+2\gamma-\chi_8)\chi_3^5}{\chi_8^6},\quad \bar{b}_{02}=-\frac{\chi_2}{\chi_1},\quad \bar{b}_{21}=-\frac{\chi_2\chi_6}{\chi_1^3},\\[2ex]
\bar{b}_{10}&=\frac{4(1+\gamma)^2(4+\gamma)^2(\chi_3\lambda_3+\gamma^2(2+\gamma))}{\chi_1^2\chi_3},\quad \bar{b}_{11}=\frac{4(1+\gamma)(4+\gamma)\big(\gamma^2(2+\gamma)+\chi_3\lambda_3\big)}{\chi_1^2},\\[2ex]
\bar{b}_{30}&=\frac{4(1+\gamma)^2(4+\gamma)^2\big(\gamma^2(2+\gamma)+\chi_3\lambda_3\big)\chi_3}{\chi_1^4},\quad \bar{b}_{12}=\frac{\big(\gamma^2(2+\gamma)+\chi_3\lambda_3\big)\chi_3}{\chi_1^2},\\[2ex]
\bar{b}_{40}&=-\frac{4(1+\gamma)^2(4+\gamma)^2\big(\gamma^2(2+\gamma)+\chi_3\lambda_3\big)\chi_3^2}{\chi_1^5},\quad \bar{b}_{31}=\frac{4(1+\gamma)(4+\gamma)\big(\gamma^2(2+\gamma)+\chi_3\lambda_3\big)\chi_3^2}{\chi_1^4},\\[2ex]
\end{aligned}
\end{array}
\end{equation*}
\begin{equation*}
\begin{array}{ll}
\begin{aligned}
\bar{b}_{50}&=\frac{4(1+\gamma)^2(4+\gamma)^2\big(\gamma^2(2+\gamma)+\chi_3\lambda_3\big)\chi_3^3}{\chi_1^6},\quad \bar{b}_{32}=\frac{\big(\gamma^2(2+\gamma)+\chi_3\lambda_3\big)\chi_3^3}{\chi_1^4},\\[2ex]
\bar{b}_{41}&=-\frac{4(1+\gamma)(4+\gamma)\big(\gamma^2(2+\gamma)+\chi_3\lambda_3\big)\chi_3^3}{\chi_1^5},\quad \bar{b}_{22}=-\frac{\chi_2\chi_7}{\chi_1^3}.
\end{aligned}
\end{array}
\end{equation*}
and
\begin{equation*}
\begin{array}{ll}
\begin{aligned}
\chi_1&=2(1+\gamma)(4+\gamma)+\chi_3\lambda_4,\quad\chi_3=4\gamma^3+19\gamma^2+20\gamma+4,\quad\chi_{10}=4(1+\gamma)^3(2+\gamma)^2+\chi_7\lambda_1\\[2ex]
\chi_5&=4\gamma^4+40\gamma^3+132\gamma^2+160\gamma+64,\quad\chi_6=16\gamma^5+156\gamma^4+524\gamma^3+720\gamma^2+400\gamma+64,\\[2ex]
\chi_7&=16\gamma^6+152\gamma^5+521\gamma^4+792\gamma^3+552\gamma^2+160\gamma+16,\quad\chi_4=2(1+\gamma)(4+\gamma)-\chi_3\lambda_4,\\[2ex]
\chi_8&=2\gamma^2+6\gamma+4+\chi_3\lambda_2,\quad\chi_9=\gamma\chi_3\lambda_1-2\gamma(1+\gamma)^2(2+\gamma)\lambda_2,\quad\chi_2=\gamma^2(2+\gamma)+\chi_3\lambda_3.
\end{aligned}
\end{array}
\end{equation*}

\section{The coefficients $\bar{c}_{ij}$ for system \eqref{2.18}}

\begin{equation*}
\begin{array}{ll}
\begin{aligned}
\bar{c}_{00}&=\bar{a}_{01}\bar{b}_{00}-\bar{a}_{00}\bar{b}_{01}+\frac{\bar{a}_{00}^2\bar{b}_{02}}{\bar{a}_{01}},\quad \bar{c}_{01}=\bar{a}_{10}+\bar{b}_{01}-\frac{\bar{a}_{00}(\bar{a}_{11}+2\bar{b}_{02})}{\bar{a}_{01}},\\[2ex]
\bar{c}_{10}&=\bar{a}_{11}\bar{b}_{00}-\bar{a}_{10}\bar{b}_{01}+\bar{a}_{01}\bar{b}_{10}-\bar{a}_{00}\bar{b}_{11}+\frac{\bar{a}_{00}}{\bar{a}_{01}}(2\bar{a}_{10}\bar{b}_{02}+\bar{a}_{00}\bar{b}_{12})-\frac{\bar{a}_{00}^2\bar{a}_{11}\bar{b}_{02}}{\bar{a}_{01}^2},\\[2ex]
\bar{c}_{20}&=\bar{a}_{11}\bar{b}_{10}+\bar{a}_{01}\bar{b}_{20}-\bar{a}_{20}\bar{b}_{01}-\bar{a}_{10}\bar{b}_{11}-\bar{a}_{00}\bar{b}_{21}+\frac{\bar{a}_{00}\bar{a}_{11}(2\bar{a}_{10}\bar{b}_{02}+\bar{a}_{00}\bar{b}_{12})}{\bar{a}_{01}^2}\\[2ex]
&\quad+\frac{\big(\bar{b}_{02}(\bar{a}_{10}^2+2\bar{a}_{00}\bar{a}_{20})+\bar{a}_{00}(2\bar{a}_{10}\bar{b}_{12}+\bar{a}_{00}\bar{b}_{22})\big)}{\bar{a}_{01}}+\frac{\bar{a}_{00}^2\bar{a}_{11}^2\bar{b}_{02}}{\bar{a}_{01}^3},\\[2ex]
\bar{c}_{11}&=2\bar{a}_{20}+\bar{b}_{11}-\frac{\bar{a}_{10}(\bar{a}_{11}+2\bar{b}_{02})+2\bar{a}_{00}\bar{b}_{12}}{\bar{a}_{01}}+\frac{\bar{a}_{00}\bar{a}_{11}(\bar{a}_{11}+2\bar{b}_{02})}{\bar{a}_{01}^2},\\[2ex]
\bar{c}_{02}&=\frac{\bar{a}_{11}+\bar{b}_{02}}{\bar{a}_{01}},\quad \bar{c}_{32}=\frac{\bar{b}_{32}}{\bar{a}_{01}}-\frac{\bar{a}_{11}\bar{b}_{22}}{\bar{a}_{01}^2}+\frac{\bar{a}_{11}^2\bar{b}_{12}}{\bar{a}_{01}^3}-\frac{\bar{a}_{11}^3(\bar{a}_{11}+\bar{b}_{02})}{\bar{a}_{01}^4},\\[2ex]
\bar{c}_{30}&=\bar{a}_{11}\bar{b}_{20}+\bar{a}_{01}\bar{b}_{30}-\bar{a}_{30}\bar{b}_{01}-\bar{a}_{20}\bar{b}_{11}-\bar{a}_{10}\bar{b}_{21}-\bar{a}_{00}\bar{b}_{31}-\frac{\bar{a}_{00}^2\bar{a}_{11}^3\bar{b}_{02}}{\bar{a}_{01}^4}\\[2ex]
&\quad+\frac{2\bar{b}_{02}(\bar{a}_{10}\bar{a}_{20}+\bar{a}_{00}\bar{a}_{30})+\bar{b}_{12}(\bar{a}_{10}^2+2\bar{a}_{00}\bar{a}_{20})+\bar{a}_{00}(2\bar{a}_{10}\bar{b}_{22}+\bar{a}_{00}\bar{b}_{32})}{\bar{a}_{01}}\\[2ex]
&\quad+\frac{\bar{a}_{11}\bar{a}_{10}(\bar{a}_{10}\bar{b}_{02}+2\bar{a}_{00}\bar{b}_{12})+\bar{a}_{11}\bar{a}_{00}(2\bar{a}_{20}\bar{b}_{02}+\bar{a}_{00}\bar{b}_{22})}{\bar{a}_{01}^2}\\[2ex]
&\quad+\frac{\bar{a}_{00}\bar{a}_{11}^2(2\bar{a}_{10}\bar{b}_{02}+\bar{a}_{00}\bar{b}_{12})}{\bar{a}_{01}^3},\quad \bar{c}_{22}=\frac{\bar{b}_{22}}{\bar{a}_{01}}-\frac{\bar{a}_{11}\bar{b}_{12}}{\bar{a}_{01}^2}+\frac{\bar{a}_{11}^2(\bar{a}_{11}+\bar{b}_{02})}{\bar{a}_{01}^3},\\[2ex]
\bar{c}_{21}&=3\bar{a}_{30}+\bar{b}_{21}-\frac{\bar{a}_{20}(\bar{a}_{11}+2\bar{b}_{02})+2(\bar{a}_{10}\bar{b}_{12}+\bar{a}_{00}\bar{b}_{22})}{\bar{a}_{01}}-\frac{\bar{a}_{00}\bar{a}_{11}^2(\bar{a}_{11}+2\bar{b}_{02})}{\bar{a}_{01}^3}\\[2ex]
&\quad+\frac{\bar{a}_{11}\bar{a}_{10}(\bar{a}_{11}+2\bar{b}_{02})+2\bar{a}_{11}\bar{a}_{10}\bar{a}_{00}\bar{b}_{12}}{\bar{a}_{01}^2},\quad \bar{c}_{12}=\frac{\bar{b}_{12}}{\bar{a}_{01}}-\frac{\bar{a}_{11}(\bar{a}_{11}+\bar{b}_{02})}{\bar{a}_{01}^2},\\[2ex]
\end{aligned}
\end{array}
\end{equation*}
\begin{equation*}
\begin{array}{ll}
\begin{aligned}
\bar{c}_{41}&=5\bar{a}_{50}+\bar{b}_{41}-\frac{\bar{a}_{40}(\bar{a}_{11}+2\bar{b}_{02})+2(\bar{a}_{30}\bar{b}_{12}+\bar{a}_{20}\bar{b}_{22}+\bar{a}_{10}\bar{b}_{32})}{\bar{a}_{01}}\\[2ex]
&\quad+\frac{\bar{a}_{11}\bar{a}_{30}(\bar{a}_{11}+2\bar{b}_{02})+2\bar{a}_{11}(\bar{a}_{20}\bar{b}_{12}+\bar{a}_{10}\bar{b}_{22}+\bar{a}_{00}\bar{b}_{32})}{\bar{a}_{01}^2}-\frac{\bar{a}_{00}\bar{a}_{11}^4(\bar{a}_{11}+2\bar{b}_{02})}{\bar{a}_{01}^5}\\[2ex]
&\quad-\frac{\bar{a}_{11}^2\big(\bar{a}_{20}(\bar{a}_{11}+2\bar{b}_{02})+2(\bar{a}_{10}\bar{b}_{12}+\bar{a}_{00}\bar{b}_{22})\big)}{\bar{a}_{01}^3}+\frac{\bar{a}_{11}^3\big(\bar{a}_{10}(\bar{a}_{11}+2\bar{b}_{02})+2\bar{a}_{00}\bar{b}_{12}\big)}{\bar{a}_{01}^4},\\[2ex]
\bar{c}_{31}&=4\bar{a}_{40}+\bar{b}_{31}-\frac{\bar{a}_{30}(\bar{a}_{11}+2\bar{b}_{02})+2(\bar{a}_{20}\bar{b}_{12}+\bar{a}_{10}\bar{b}_{22}+\bar{a}_{00}\bar{b}_{32})}{\bar{a}_{01}}+\frac{\bar{a}_{00}\bar{a}_{11}^3(\bar{a}_{11}+2\bar{b}_{02})}{\bar{a}_{01}^4}\\[2ex]
&\quad+\frac{\bar{a}_{11}\bar{a}_{20}(\bar{a}_{11}+2\bar{b}_{02})+2\bar{a}_{11}(\bar{a}_{10}\bar{b}_{12}+\bar{a}_{00}\bar{b}_{22})}{\bar{a}_{01}^2}+\frac{\bar{a}_{11}^2\bar{a}_{10}(\bar{a}_{11}+2\bar{b}_{02})+2\bar{a}_{11}^2\bar{a}_{00}\bar{b}_{12}}{\bar{a}_{01}^3},\\[2ex]
&\quad+\frac{\bar{b}_{02}(\bar{a}_{20}^2+2\bar{a}_{10}\bar{a}_{30})+\bar{a}_{10}(2\bar{a}_{20}\bar{b}_{12}+\bar{a}_{10}\bar{b}_{22})+2\bar{a}_{00}(\bar{a}_{40}\bar{b}_{02}+\bar{a}_{30}\bar{b}_{12}+\bar{a}_{20}\bar{b}_{22}+\bar{a}_{10}\bar{b}_{32})}{\bar{a}_{01}}\\[2ex]
\end{aligned}
\end{array}
\end{equation*}
\begin{equation*}
\begin{array}{ll}
\begin{aligned}
&\quad-\frac{\bar{a}_{11}\big(\bar{a}_{10}^2\bar{b}_{12}+2\bar{a}_{10}(\bar{a}_{20}\bar{b}_{02}+\bar{a}_{00}\bar{b}_{22})+\bar{a}_{00}(2\bar{a}_{30}\bar{b}_{02}+2\bar{a}_{20}\bar{b}_{12}+\bar{a}_{00}\bar{b}_{32})\big)}{\bar{a}_{01}^2}\\[2ex]
&\quad+\frac{\bar{a}_{11}^2\bar{a}_{10}(\bar{a}_{10}\bar{b}_{02}+2\bar{a}_{00}\bar{b}_{12})+\bar{a}_{11}^2\bar{a}_{00}(2\bar{a}_{20}\bar{b}_{02}+\bar{a}_{00}\bar{b}_{22})}{\bar{a}_{01}^3}-\frac{\bar{a}_{00}\bar{a}_{11}^3(2\bar{a}_{10}\bar{b}_{02}+\bar{a}_{00}\bar{b}_{12})}{\bar{a}_{01}^4},\\[2ex]
\bar{c}_{50}&=\bar{a}_{11}\bar{b}_{40}+\bar{a}_{01}\bar{b}_{50}-\bar{a}_{50}\bar{b}_{01}-\bar{a}_{40}\bar{b}_{11}-\bar{a}_{30}\bar{b}_{21}-\bar{a}_{20}\bar{b}_{31}-\bar{a}_{10}\bar{b}_{41}-\frac{\bar{a}_{00}^2\bar{a}_{11}^5\bar{b}_{02}}{\bar{a}_{01}^6}\\[2ex]
&\quad+\frac{2\bar{b}_{02}(\bar{a}_{20}\bar{a}_{30}+\bar{a}_{10}\bar{a}_{40}+\bar{a}_{00}\bar{a}_{50})}{\bar{a}_{01}}+\frac{\bar{b}_{12}(\bar{a}_{20}^2+2\bar{a}_{10}\bar{a}_{30}+2\bar{a}_{00}\bar{a}_{40})}{\bar{a}_{01}}\\[2ex]
&\quad+\frac{2\bar{b}_{22}(\bar{a}_{10}\bar{a}_{20}+\bar{a}_{00}\bar{a}_{30})+\bar{b}_{32}(\bar{a}_{10}^2+2\bar{a}_{00}\bar{a}_{20})}{\bar{a}_{01}}+\frac{\bar{a}_{11}\bar{b}_{02}(\bar{a}_{20}^2+2\bar{a}_{10}\bar{a}_{30})}{\bar{a}_{01}^2}\\[2ex]
&\quad+\frac{2\bar{a}_{11}\bar{b}_{12}(\bar{a}_{10}\bar{a}_{20}+\bar{a}_{00}\bar{a}_{30})+\bar{a}_{11}\bar{b}_{22}(\bar{a}_{10}^2+2\bar{a}_{00}\bar{a}_{20})+2\bar{a}_{00}\bar{a}_{11}(\bar{a}_{40}\bar{b}_{02}+\bar{a}_{10}\bar{b}_{32})}{\bar{a}_{01}^2}\\[2ex]
&\quad+\frac{\bar{a}_{11}^2\big(2\bar{b}_{02}(\bar{a}_{10}\bar{a}_{20}+\bar{a}_{00}\bar{a}_{30})+\bar{b}_{12}(\bar{a}_{10}^2+2\bar{a}_{00}\bar{a}_{20}+\bar{a}_{00}(2\bar{a}_{10}\bar{b}_{22}+\bar{a}_{00}\bar{b}_{32}))\big)}{\bar{a}_{01}^3}\\[2ex]
&\quad-\frac{\bar{a}_{11}^3\big(\bar{a}_{10}(\bar{a}_{10}\bar{b}_{02}+2\bar{a}_{00}\bar{b}_{12})+\bar{a}_{00}(2\bar{a}_{20}\bar{b}_{02}+\bar{a}_{00}\bar{b}_{22})\big)}{\bar{a}_{01}^4}+\frac{\bar{a}_{00}\bar{a}_{11}^4(2\bar{a}_{10}\bar{b}_{02}+\bar{a}_{00}\bar{b}_{12})}{\bar{a}_{01}^5},\\[2ex]
\bar{c}_{40}&=\bar{a}_{11}\bar{b}_{30}+\bar{a}_{01}\bar{b}_{40}-\bar{a}_{40}\bar{b}_{01}-\bar{a}_{30}\bar{b}_{11}-\bar{a}_{20}\bar{b}_{21}-\bar{a}_{10}\bar{b}_{31}-\bar{a}_{00}\bar{b}_{41}+\frac{\bar{a}_{00}^2\bar{a}_{11}^4\bar{b}_{02}}{\bar{a}_{01}^5}.\\[2ex]
\end{aligned}
\end{array}
\end{equation*}

\section{The coefficients $\bar{l}_{ij}$ for system \eqref{2.23}}
\begin{equation*}
\begin{array}{ll}
\begin{aligned}
\bar{l}_{20}&=\bar{k}_{20}+\frac{9\bar{k}_{00}\bar{k}_{30}^2}{16\bar{k}_{20}^2}-\frac{3\bar{k}_{00}\bar{k}_{40}}{5\bar{k}_{20}}-\frac{3\bar{k}_{10}\bar{k}_{30}}{4\bar{k}_{20}},\quad \bar{l}_{40}=-\frac{\bar{k}_{10}(55\bar{k}_{30}^3-96\bar{k}_{20}\bar{k}_{30}\bar{k}_{40}+40\bar{k}_{20}^2\bar{k}_{50})}{48\bar{k}_{20}^3},\\[2ex]
\end{aligned}
\end{array}
\end{equation*}
\begin{equation*}
\begin{array}{ll}
\begin{aligned}
\bar{l}_{21}&=\bar{k}_{21}-\frac{3\big(5\bar{k}_{30}(4\bar{k}_{11}\bar{k}_{20}-3\bar{k}_{01}\bar{k}_{30})+16\bar{k}_{01}\bar{k}_{20}\bar{k}_{40}\big)}{80\bar{k}_{20}^2},\quad \bar{l}_{01}=\bar{k}_{01},\quad \bar{l}_{11}=\bar{k}_{11}-\frac{\bar{k}_{01}\bar{k}_{30}}{2\bar{k}_{20}},\\[2ex]
\bar{l}_{30}&=\frac{6\bar{k}_{10}\bar{k}_{20}(35\bar{k}_{30}^2-32\bar{k}_{20}\bar{k}_{40})-\bar{k}_{00}(175\bar{k}_{30}^3-336\bar{k}_{20}\bar{k}_{30}\bar{k}_{40}+160\bar{k}_{20}^2\bar{k}_{50})}{240\bar{k}_{20}^3},\quad \bar{l}_{00}=\bar{k}_{00},\\[2ex]
\bar{l}_{31}&=\frac{1}{240\bar{k}_{20}^3}\big(16\bar{k}_{20}^2(15\bar{k}_{20}\bar{k}_{31}-12\bar{k}_{11}\bar{k}_{40}-10\bar{k}_{01}\bar{k}_{50})-175\bar{k}_{01}\bar{k}_{30}^3\\[2ex]
&\quad-6\bar{k}_{20}\bar{k}_{30}(40\bar{k}_{20}\bar{k}_{21}-35\bar{k}_{11}\bar{k}_{30}-56\bar{k}_{01}\bar{k}_{40})\big),\quad l_{10}=\bar{k}_{10}-\frac{\bar{k}_{00}\bar{k}_{30}}{2\bar{k}_{20}},\\[2ex]
\bar{l}_{50}&=\frac{\bar{k}_{10}\big(2425\bar{k}_{30}^4+768\bar{k}_{20}^2\bar{k}_{40}^2+1600\bar{k}_{20}\bar{k}_{30}(\bar{k}_{20}\bar{k}_{50}-3\bar{k}_{30}\bar{k}_{40})\big)}{6400\bar{k}_{20}^4},\\[2ex]
\bar{k}_{41}&=\frac{1}{48\bar{k}_{20}^3}\big(8\bar{k}_{20}^2(6\bar{k}_{20}\bar{k}_{41}-5\bar{k}_{11}\bar{k}_{50})-48\bar{k}_{20}\bar{k}_{40}(\bar{k}_{20}\bar{k}_{21}-2\bar{k}_{11}\bar{k}_{30})\\[2ex]
&\quad-5\bar{k}_{30}(11\bar{k}_{11}\bar{k}_{30}^2+12\bar{k}_{20}^2\bar{k}_{31}-12\bar{k}_{20}\bar{k}_{21}\bar{k}_{30})\big).
\end{aligned}
\end{array}
\end{equation*}

\section{The coefficients $\bar{p}_{ij}$ and $\bar{q}_{ij}$ for system \eqref{4.4}}
\begin{equation*}
\begin{array}{ll}
\begin{aligned}
\bar{p}_{11}&=\frac{2z+\beta+\eta}{(z+\eta)(z+\beta)},\quad \bar{p}_{02}=\frac{\sqrt{(1+\eta)z^2+2\beta\eta z+\beta^2\eta-\beta\eta}}{(z+\beta)\sqrt{z+\eta}\sqrt{\beta-\beta^2-2\beta z-z^2}},\\[2ex]
\bar{p}_{12}&=\frac{\sqrt{(1+\eta)z^2+2\beta\eta z+\beta^2\eta-\beta\eta}}{(z+\beta)(z+\eta)^{3/2}\sqrt{\beta-\beta^2-2\beta z-z^2}},\quad \bar{p}_{21}=\frac{1}{(z+\beta)(z+\eta)},\\[2ex]
\bar{q}_{20}&=\frac{(1-\beta)\beta\eta^2-\beta\eta(2\beta+3\eta-2)z-\eta(2+5\beta+2\eta)z^2-3(1+h)z^3}{z(z+\eta)\big((1+\eta)z^2+2\beta\eta z+\beta^2\eta-\beta\eta\big)},\\[2ex]
\bar{q}_{11}&=\frac{\bar{\omega}_1}{z(z+\beta)(z+\eta)^{3/2}\bar{\omega}_4},\quad \bar{q}_{12}=\frac{\bar{\omega}_3}{z(z+\beta)(z+\eta)^2(\beta^2+2\beta z+z^2-\beta)},\\[2ex]
\bar{q}_{02}&=\frac{2z^4+(6\beta+\eta)z^3+\beta(6\beta+3\eta-1)z^2+\beta^2(2\beta+3\eta-2)z+(\beta-1)\beta^2\eta}{z(z+\beta)(z+\eta)(z^2+2\beta z+\beta^2-\beta)},\\[2ex]
\bar{q}_{21}&=\frac{\bar{\omega}_2}{z(z+\beta)(z+\eta)^{3/2}\bar{\omega}_4},\quad \bar{q}_{30}=\frac{\beta\eta(1-\beta-\eta)-(1+4\beta+\eta)\eta z-3(1+\eta)z^2}{z(z+\eta)((1+\eta)z^2+2\beta\eta z+\beta^2\eta-\beta\eta)},\\[2ex]
\bar{q}_{03}&=\frac{\big(\beta(1-\beta-\eta)-(4\beta+\eta)z-3z^2\big)\sqrt{(1+\eta)z^2+2\beta\eta z+\beta^2\eta-\beta\eta}}{z(z+\eta)^{3/2}(\beta-\beta^2-2\beta z-z^2)^{3/2}},\\[2ex]
\bar{q}_{40}&=-\frac{\beta\eta+\eta z+z}{z(z+\eta)((1+\eta)z^2+2\beta\eta z+\beta^2\eta-\beta\eta)},\quad \bar{q}_{22}=\frac{3(2\beta\eta+z+\beta z+2\eta z+z^2)}{z(z+\eta)^2(\beta^2+2\beta z+z^2-\beta)},\\[2ex]
\bar{q}_{13}&=-\frac{(4\beta\eta+z+3\beta z+4\eta z+3z^2)\sqrt{(1+\eta)z^2+2\beta\eta z+\beta^2\eta-\beta\eta}}{z(z+\eta)^{5/2}(\beta-\beta^2-2\beta z-z^2)^{3/2}},\\[2ex]
\bar{q}_{04}&=-\frac{(z+\beta)((1+\eta)z^2+2\beta\eta z+\beta^2\eta-\beta\eta)}{z(z+\eta)^2(z^2+2\beta z+\beta^2-\beta)^2},\quad \bar{q}_{31}=-\frac{4\beta\eta+3z+\beta z+4\eta z+z^2}{z(z+\eta)^{3/2}\bar{\omega}_4},\\[2ex]
\end{aligned}
\end{array}
\end{equation*}
\begin{equation*}
\begin{array}{ll}
\begin{aligned}
\omega_1&=(1-\beta)\beta\eta^2-(4\beta^2-\eta+7\beta\eta-4\beta)\beta\eta z-(\beta^2+13\beta\eta+8\eta^2-\beta)\beta z^2\\[2ex]
&\quad-(3\beta+3\beta^2+2\eta+14\beta\eta+3\eta^2)z^3-(3+3\beta+5\eta)z^4-z^5,\\[2ex]
\omega_2&=3\beta^2\eta(1-\beta-\eta)-(\beta^2+15\beta\eta+6\eta^2-\beta)\beta z-(6\beta+4\beta^2+2\eta+21\beta\eta+3\eta^2)z^2\\[2ex]
&\quad-(6+5\beta+9\eta)z^3-2z^4,\quad \omega_4=\sqrt{\beta-\beta^2-2\beta z-z^2}\sqrt{(1+\eta)z^2+2\beta\eta z+\beta^2\eta-\beta\eta},\\[2ex]
\omega_3&=3\beta^2\eta(\beta+\eta-1)+(2\beta^2+16\beta\eta+6\eta^2-\eta-2\beta)\beta z+(2\beta+9\beta^2+\eta+23\beta\eta+3\eta^2)z^2\\[2ex]
&\quad+(3+12\beta+10\eta)z^3+5z^4.
\end{aligned}
\end{array}
\end{equation*}

\section{The coefficients $p_{ij}$ and $q_{ij}$ for system \eqref{3.3}}
\begin{equation*}
\begin{array}{ll}
\begin{aligned}
p_{11}&=\frac{1}{z+\eta}-\omega_1,\quad p_{02}=-\frac{\sqrt{z\gamma-\delta}\omega_1}{\sqrt{\delta}},\quad p_{21}=-\omega_2,\quad p_{12}=-\frac{\sqrt{z\gamma-\delta}\omega_2}{\sqrt{\delta}},\quad q_{12}=\frac{\omega_2\omega_7}{\delta(z+\delta)},\\[2ex]
q_{20}&=-\frac{2z+\eta}{z(z+\eta)}+\frac{(1+\gamma)z\omega_1}{z\gamma-\delta},\quad q_{11}=\frac{\omega_2\omega_5}{(z+\delta)\sqrt{\delta(z\gamma-\delta)}},\quad q_{02}=\frac{2z+\eta}{z(z+\eta)}+\frac{(z+\delta)\omega_1}{\delta},\\[2ex]
q_{30}&=\frac{\omega_2\omega_4}{(z+\delta)(z\gamma-\delta)},\quad q_{21}=\frac{\omega_2\omega_6}{(z+\delta)\sqrt{\delta(z+\delta)}},\quad q_{04}=\frac{(z\gamma-\delta)\omega_2}{\delta^2},\quad q_{03}=\frac{\sqrt{z\gamma-\delta}\omega_2\omega_3}{\delta^{3/2}(z+\delta)},\\[2ex]
q_{40}&=\frac{(1+\gamma)\omega_2}{z\gamma-\delta},\quad q_{31}=\frac{(4+3\gamma)\omega_2}{\sqrt{\delta(z\gamma-\delta)}},\quad q_{22}=\frac{3\omega_2(2+\gamma)}{\delta},\quad q_{13}=\frac{(4+\gamma)\sqrt{z\gamma-\delta}\omega_2}{\delta^{3/2}},\\[2ex]
\omega_1&=-\frac{z+\delta}{z(1-z-z\gamma-\gamma\eta)},\quad\omega_2=-\frac{z+\delta}{z(z+\eta)(1-z-z\gamma-\gamma\eta)},\quad \omega_6=\omega_3+2\omega_4+\delta(z+\delta),\\[2ex]
\omega_3&=2z^2+(\gamma\delta+3\delta+\eta)z+(\gamma\eta+\eta-1)\delta,\quad \omega_7=2\omega_3+\omega_4+\delta(z+\delta),\\[2ex]
\omega_4&=(2+\gamma-\gamma^2)z^2+(3\delta+\eta+\gamma+3\gamma\delta+\gamma\eta-\gamma^2\eta)z+(2\gamma\eta+\eta-1)\delta,\\[2ex]
\omega_5&=(2-\gamma-2\gamma^2)z^3+(6\delta+2\eta+2\gamma+4\gamma\delta-3\gamma^2\eta)z^2+(2\gamma\eta+\delta-2)\delta\eta\\[2ex]
&\quad+(\delta^2+\gamma\eta-\gamma^2\eta^2-3\delta+5\delta\eta+6\gamma\delta\eta)z.
\end{aligned}
\end{array}
\end{equation*}

\section{The factor $l_{22}$ for system \eqref{4.13}}
\begin{equation*}
\begin{array}{ll}
\begin{aligned}
l_{22}&=2531480345458623423 + 269823922452908179\sqrt{89}-24737925952523844529\gamma\\[2ex]
&\quad-2539236122628595229\sqrt{89}\gamma+73958236899762541604\gamma^2+9447664321574111828\sqrt{89}\gamma^2\\[2ex]
&\quad-182388316387159279040\gamma^3-3670919638546547392\sqrt{89}\gamma^3-399742567746445586944\gamma^4\\[2ex]
&\quad+44233822437435913728\sqrt{89}\gamma^4-1194548294202179645440\gamma^5+152811286517601521664\sqrt{89}\gamma^5\\[2ex]
&\quad-2469142638364472459264\gamma^6+236521231547966144512\sqrt{89}\gamma^6-1762598345423458140160\gamma^7\\[2ex]
&\quad+193258260590076100608\sqrt{89}\gamma^7+393200766880938065920\gamma^8-38853225167809019904\sqrt{89}\gamma^8\\[2ex]
&\quad+1077490108584619016192\gamma^9-116378157743404613632\sqrt{89}\gamma^9+188948390068969930752\gamma^{10}\\[2ex]
\end{aligned}
\end{array}
\end{equation*}
\begin{equation*}
\begin{array}{ll}
\begin{aligned}
&\quad-19850495794229411840\sqrt{89}\gamma^{10}-144752116256103989248\gamma^{11}+15549277750838689792\sqrt{89}\gamma^{11}\\[2ex]
&\quad-28488856548252975104\gamma^{12}+3047521373377658880\sqrt{89}\gamma^{12}+9887894139202174976\gamma^{13}\\[2ex]
&\quad-1046198106515832832\sqrt{89}\gamma^{13}+164450307778019328\gamma^{14}-17751260494036992\sqrt{89}\gamma^{14}\\[2ex]
&\quad-92388989340745728\gamma^{15}+9738715778777088\sqrt{89}\gamma^{15}.
\end{aligned}
\end{array}
\end{equation*}
\end{appendices}

\end{document}